\documentclass[12pt,a4]{amsart}
\usepackage{amsmath,amsthm, amssymb}
\usepackage{graphicx}
\usepackage{xcolor}
\usepackage{enumitem}
\usepackage{hyperref}

\catcode`@=11 \catcode`!=11

\expandafter\ifx\csname fiverm\endcsname\relax
  \let\fiverm\fivrm
\fi
  
\let\!latexendpicture=\endpicture 
\let\!latexframe=\frame
\let\!latexlinethickness=\linethickness
\let\!latexmultiput=\multiput
\let\!latexput=\put
 
\def\@picture(#1,#2)(#3,#4){%
  \@picht #2\unitlength
  \setbox\@picbox\hbox to #1\unitlength\bgroup 
  \let\endpicture=\!latexendpicture
  \let\frame=\!latexframe
  \let\linethickness=\!latexlinethickness
  \let\multiput=\!latexmultiput
  \let\put=\!latexput
  \hskip -#3\unitlength \lower #4\unitlength \hbox\bgroup}

\catcode`@=12 \catcode`!=12

\input pictex.tex

\catcode`@=11 \catcode`!=11
  
\let\!pictexendpicture=\endpicture 
\let\!pictexframe=\frame
\let\!pictexlinethickness=\linethickness
\let\!pictexmultiput=\multiput
\let\!pictexput=\put

\def\beginpicture{%
  \setbox\!picbox=\hbox\bgroup%
  \let\endpicture=\!pictexendpicture
  \let\frame=\!pictexframe
  \let\linethickness=\!pictexlinethickness
  \let\multiput=\!pictexmultiput
  \let\put=\!pictexput
  \let\input=\@@input   
  \!xleft=\maxdimen  
  \!xright=-\maxdimen
  \!ybot=\maxdimen
  \!ytop=-\maxdimen}

\let\frame=\!latexframe

\let\pictexframe=\!pictexframe

\let\linethickness=\!latexlinethickness
\let\pictexlinethickness=\!pictexlinethickness

\let\\=\@normalcr
\catcode`@=12 \catcode`!=12

\def\bitem{\begin{itemize}[topsep=0.1cm,itemsep=0.05ex,leftmargin=0.4cm]}

\newtheorem{theorem}{Theorem}[section]

\newtheorem*{theoremn}{Theorem}
\newtheorem*{theoremmain}{Main Theorem}
\newtheorem*{theorema}{Theorem A} 
\newtheorem{remark}[theorem]{Remark}
\newtheorem{conjecture}[theorem]{Conjecture}
\newtheorem*{conjecturen}{Conjecture}

\theoremstyle{definition} 
\newtheorem{example}[theorem]{Example}
\theoremstyle{plain}
\newtheorem{definition}[theorem]{Definition}
\newtheorem{proposition}[theorem]{Proposition}
\newtheorem{lemma}[theorem]{Lemma}

\def\EC{\PH}
\def\PartHyp{\mathcal{KS}}
\newenvironment{proofof}[1]{
\noindent{\em Proof of #1.}}{ \hfill\qed\\ }

\def\ii{\underline i}
\def\ie{{\em i.e.,\ }}
\def\eg{{\em e.g.\ }}
 
\def\PH{{\mathcal{PH}}} 
\def\T{{\mathcal{H}}}
\def\ST{\Psi} 
\def\SS{{\mathcal S}} 
\def\shape{\epsilon}
\newfont\bbf{msbm10 at 12pt}

\def\PB{P^b}
\def\lh0{h_0^-}
\def\gh0{h_0^+}
\def\bh0{\partial h_0}

\def\naturall{{\bbf \flat}}
\def\cell#1{{\langle #1  \rangle }}
\def\cellN#1{{[#1]_\naturall}}
\def\eps{\varepsilon}
\def\phi{\varphi}
\def\R{{\mathbb R}}
\def\Z{{\mathbb Z}}
\def\C{{\mathbb C}}

\def\B{W}
\def\basin{preplateau}

\newcommand{\I}{\mbox{\bf I}}
\newcommand{\J}{\mathcal{J}}

\newcommand{\D}{\mathbb{D}}

\newcommand{\bbb}{{\rm b}}
\newcommand{\aaa}{{\rm a}}

\newcommand{\interior}{\mbox{{\rm int}}}

\newcommand{\bfA}{\mbox{\bf A}}

\def\N{{\mathbb N}}

\def\parabolic{{\mathcal A}_\naturall}

\def\Crit{\mbox{Crit}}

\def\dist{\mbox{dist}}

\def\orb{\mbox{\rm orb}}
\def\supp{\mbox{supp}}

\def\st{\mbox{ \ ; \ }}
\def\sgn{\mbox{sgn}}
\def\hsgn{\widehat{\mbox{sgn}}}

\def\Ei{\Gamma}
\def\ei{\gamma}
\def\ed{\delta}
\def\hed{\hat\delta}
\def\edh{\hat\delta}

\def\Eih{\hat\Gamma}
\def\Edh{\Delta}

\begin{document}
\bibliographystyle{plain}
\title{Monotonicity of entropy for real multimodal maps}
\author{Henk Bruin  and Sebastian van Strien}
\address{Henk Bruin,  Faculty of Mathematics,  University of Vienna,  Oskar Morgenstern Platz 1, A-1090 Vienna, Austria. \newline
Sebastian van Strien,  Department of Mathematics, Imperial College, 180 Queen's Gate, London SW7 2AZ, UK.}
\email{henk.bruin@univie.ac.at \mbox{\it and\, }s.van-strien@imperial.ac.uk}
\urladdr{http://www.mat.univie.ac.at/\~{}bruin/  \mbox{\it and\, }http://www2.imperial.ac.uk/~svanstri/}
\thanks{HB was supported by
 EPSRC grants  GR/S91147/01 and EP/F037112/1, 
SvS was supported  by a Royal Society Leverhulme Trust Senior Research Fellowship, a Visitor's Travel grant from the  Netherlands Organisation for Scientific Research (NWO)
and the  Marie Curie grant MRTN-CT-2006-035651 (CODY)}
\date{To appear in the Journal of the AMS (submitted 20 May 2009, accepted 5 Dec 2013)}
\begin{abstract}
In 1992,  Milnor \cite{Mil} posed the {\em Monotonicity Conjecture}
that within a family of real multimodal polynomial interval maps with only real critical points,
the {\em isentropes}, \ie 
the sets of parameters for which the topological entropy is constant,
are connected. This conjecture was already proved in the mid-1980s for quadratic maps by a number
of different methods, see \cite{MT,DH,Dou,MS,Tsu}.  
In 2000, Milnor \& Tresser \cite{MTr}, provided a proof for the case of 
cubic maps. 
In this paper we will prove the general case of this 20 year old conjecture.
\end{abstract}
\maketitle

\section{Introduction and Statement of Results.}

Given a family of continuous maps $f$ from an interval $I$ to itself, 
one can ask how its 
{\lq}dynamical complexity{\rq} depends on $f$. Let us assume that $I$ can be 
decomposed into finitely many subintervals $I_0,\dots,I_b$  on which $f$ is monotone. 
The smallest number $b+1$ of such intervals is called the  {\em lap number}  $\ell(f)$ of $f$.
Note that $b$ is the number of extrema of $f$, and is often called the  {\em modality} of $f$. 
Maps for which $b$ is equal to one or two are called {\em  unimodal} and {\em bimodal} respectively. 
One natural measurement 
of the dynamical complexity of
$f$ is the rate of exponential growth of the lap numbers $\ell(f^n)$ where $f^n$
denotes the $n$-th iterate of $f$. This growth rate
$\lim_{n\to \infty} \frac1n \log {\ell(f^n)}$ exists
and is equal to the usual notion of {\em topological entropy} $h_{top}(f)$ of $f$, see
\cite{MSz} and also \cite[Section II.7]{MS}.  
Topological entropy can be used to
classify maps with finite modality, up to semi-conjugacy, see \cite{MT} -- 
a bit like the rotation number enables a classification up to semi-conjugacy 
of degree one circle maps.
Conventionally, a map is called {\em chaotic} if and only if $h_{top}(f)>0$.

For continuous interval maps, $h_{top}(f)$ coincides with the exponential growth 
rate of the number of $n$-periodic orbits.
Therefore if we consider a family of continuous interval maps $f_t$, $t\in [0,1]$ 
and $h_{top}(f_1)>h_{top}(f_0)$, then many periodic orbits
are created as the parameter $t$ increases from $0$ to $1$.
However, it should be noted that entropy is only a coarse indicator of the birth
of periodic orbits, because periodic orbits can both appear and disappear in
parameter ranges of constant entropy.
This is clearly true in modality $\geq 2$ (since entropy is only
a one-dimensional observable in higher dimensional parameter space),
but already in modality one, 
entropy is constant on any parameter interval within 
a period doubling cascade. 
In order to obtain a complete picture on the emergence 
and disappearance of periodic orbits, it is therefore necessary to
combine monotonicity of entropy with so-called Thurston Rigidity, which we
explain later on in this introduction.

In this paper we consider the space of $b$-modal continuous interval maps. 
Obviously, $h_{top}(f)\in [0,\log(b+1)]$ for a $b$-modal map. It turns out 
that $f\mapsto h_{top}(f)$ is not continuous on the entire space
of $b$-modal maps, but
if we restrict to $C^1$-smooth maps then it is, see \cite{MT,Y,Mis,MS}.

The question whether $h_{top}(f)$ {\lq}increases{\rq} with $f$ goes at least back to 
the early 70s, see \cite{MetSS}. In the unimodal situation, 
one of the simplest ways of asking this question is 
as follows. Let $I=[0,1]$ and consider a smooth unimodal map $f\colon [0,1]\to [0,1]$ with $f(0)=f(1)=0$, $f(1/2)=1$
and the family $f_a(x)=af(x)$, $a\in [0,1]$.
\begin{equation}
\label{eq:incent}
\text{Does the topological entropy of $f_a(x)=af(x)$ increase with $a\in [0,1]$?}
\end{equation}
As mentioned before, entropy {\em cannot} be {\em strictly} increasing with $a$.
It has been conjectured  in the 90's that if a $C^3$ unimodal convex map $f\colon [0,1]\to [0,1]$ 
as above has negative Schwarzian and is symmetric around the critical point, then the answer to 
\eqref{eq:incent} is positive.  
This conjecture is subtle: there are $C^3$ close maps $f,g\colon [0,1]\to [0,1]$ of this type for which $f\le g$ yet 
$h_{top}(f)> h_{top}(g)$, see \cite{Bru}. Moreover, none of the assumptions can be dropped,  
see the examples in \cite{Zdu,Kol,NY} and also  
\cite[Section II.10]{MS}.
 
\begin{figure}
\begin{center}
\unitlength=4.3mm
\begin{picture}(25,15)
\put(3,0){\resizebox{10cm}{7cm}{\includegraphics{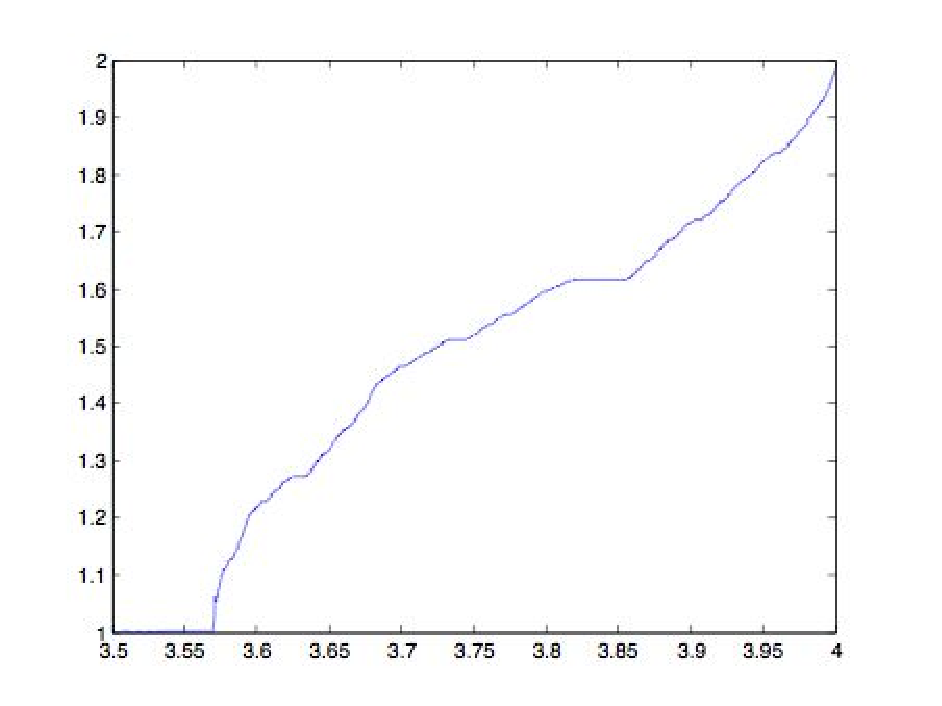}}}
\put(27,1.5){$a$}
\put(-2,12){$\exp(h_{top}(f_a))$}
\end{picture}
\end{center}
\vskip -0.4cm
\caption{\label{fig:mono} Monotonicity of entropy for the maps
$f_a(x) = ax(1-x)$, $a\in [3.5,4]$.}
\end{figure}

It was shown in the mid 1980's by Douady \& Hubbard \cite{Dou,DH} and 
Milnor \& Thurston \cite{MT}
that for the quadratic family $f_a(x)=4ax(1-x)$ the
entropy $h_{top}(f_a)$ depends monotonically  on $a\in [0,1]$. All known proofs of this use that 
the quadratic map can be extended to the complex plane and require
tools from complex analysis, see the above references and
also Tsujii's proof \cite{Tsu} and
\cite[Section II.10]{MS}. 
To show how subtle this question is, let us note that it was only very recently shown that the topological entropy of $f_a(x)=a\sin(\pi x)$ depends monotonically on $a$, 
 see \cite{RvS}. In fact, that paper shows that \eqref{eq:incent} holds for each unimodal  $f\colon [0,1]\to [0,1]$ with $f(0)=f(1)=0$
which extends to an entire transcendental map on the complex plane, with a finite number of singular values and satisfying the so-called sector condition. 

The above proofs not only show that
the topological entropy of $f_a(x)=4ax(1-x)$ increases with $a$,
but also that periodic orbits do not disappear when $a$ increases. 
In other words,  as $a$ increases, only new periodic orbits
are created  (by period doubling and saddle-node
bifurcations). 
That this is true, follows essentially from the following property:
\begin{quote}
{\bf Thurston Rigidity (combinatorially equivalent critically finite maps 
are unique):} Consider $f_a$ and $f_{a'}$ for which their critical 
points $c=1/2$ have finite orbits $O$ and $O'$.
If there exists an order preserving homeomorphism 
$h\colon I\to I$ with $h(O)=O'$ and $h\circ f_a=f_{a'}\circ h$, 
then $a=a'$. 
\end{quote}
In fact, much more is known: whenever $f_a$ and $f_{a'}$ have the
same {\lq}kneading invariant{\rq} and $f_a$ has no periodic attractor, then $a=a'$. 
This result was proven in  \cite{GS,Lyu} and is usually called the 
{\em density of  hyperbolicity for real quadratic maps}. 
This result implies that  there exists a dense set
$H\subset [0,\log 2]$ so that for any $h_0\in H$ there exists 
{\em precisely one} $a\in [0,1]$ with $h_{top}(f_a)=h_0$. 
It follows that $a\mapsto h_{top}(f_a)$ is a devil's staircase, the 
plateaus of which correspond
to intervals of parameters containing a periodic attractor and the subsequent period doubling cascade. 
By density of hyperbolicity, 
such parameters form a dense set, and so there exists no interval of parameters
on which $a\mapsto h_{top}(f_a)$ is strictly increasing. 

Douady \& Hubbard, see \cite{DH}, showed the following monotonicity result:
\begin{quote}
{\bf Bifurcations are monotone in the quadratic family:}
Assume that $(a_-,a_+)$ is a parameter range so that the quadratic family 
$f_a(x)=4ax(1-x)$ has a hyperbolic periodic attractor $p_a$ of period $n$ for each 
$a\in (a_-,a_+)$, then $a \mapsto Df_a^n(p_a)\in (-1,1)$
is differentiable and strictly decreasing.
\end{quote}
The corresponding parameter space for higher degree polynomials
is parametrized by Blaschke products,  see \cite{Mil1} and 
Theorems~\ref{thm:defspace1} and \ref{thm:defspace2} below.
Combining Thurston Rigidity with the previous property 
shows that period doubling cascades are traversed monotonically in the quadratic family.
(We should point out that there are additional results on the transversality of bifurcations
in polynomial families, see \cite{Str1, Epstein, Levin}.)

Let us turn to real cubic maps.  Take $I=[-1,1]$ and cubic maps $f\colon I\to I$ with exactly two 
critical points, both in the interior of $I$. Since this space consists
of two connected components, it makes sense to separate the cases where $f(-1)=-1, f(1)=1$
and where $f(-1)=1, f(1)=-1$.  In the former case, such cubic maps can be written in the form
$f_{\aaa,\bbb}(x)=\aaa x^3+ \bbb x^2+(1-\aaa)x-\bbb$ where
$(\aaa,\bbb)\in \R^2$  are contained in a simply-connected region bounded four  algebraic curves
(this follows as in \cite{Mil}).     It is not hard to show that for a smooth one-parameter family $f_t$ of such cubic maps,   $t\mapsto h_{top}(f_t)$ need
not be monotone, see Remark~\ref{remark:nonmono}. 
 Perhaps this is not too surprising, as the level sets of $(\aaa,\bbb)\mapsto h_{top}(f_{\aaa,\bbb})$ are 
very complicated fractal-like sets.  
Related to this is the result by Kan, Ko\c{c}ak \& Yorke \cite{KKY} that within the H\'enon family 
$F_{\aaa,\bbb}(x,y)=(1-\aaa x^2+\bbb y,x)$,
the entropy of $F_{\aaa,\bbb}$ does not depend monotonically on $\aaa$ for fixed $\bbb$. 

Yet a compelling question is whether, within the space of all real cubic maps, 
those with a given topological entropy form a connected set. 
In the early 1990's Milnor made this question precise, by defining the following space. 
Consider the space $\PB_\shape$ of real polynomials $f$ with 
\begin{enumerate}[topsep=0cm,itemsep=0.05ex,leftmargin=1cm]
\item precisely $b$ distinct  critical points, {\em all of which} are real, non-degenerate and contained in the interior of $I$;
\item  $f(\partial I)\subset \partial I$;
\item with {\em shape} $\epsilon=\epsilon(f)$, where 
$$
\epsilon(f) =\left\{ \begin{array}{ll}  +1 &\mbox{ if $f$ is increasing at the left endpoint of $I$,}\\
     -1  &\mbox{ otherwise.}\end{array}\right.
$$
\end{enumerate}
Note that $\PB_\shape$  consists of polynomials of  degree $d=b+1$. 

Milnor's conjecture essentially asserts that within this space, bifurcations are {\lq}as efficient as possible{\rq}:

\begin{quote} 
{\bf Milnor's Monotonicity of Entropy Conjecture \cite{Mil}:}
For each $\shape\in \{-,+\}$, $b\in \N$ and $h_0\ge 0$, the isentrope
$$
\{f\in \PB_\shape \st h_{top}(f)=h_0\}
$$ 
is connected.
\end{quote}

\begin{figure}
\begin{center}
\unitlength=4.3mm
\begin{picture}(25,16)
\put(1,0){\resizebox{10cm}{7cm}{\includegraphics{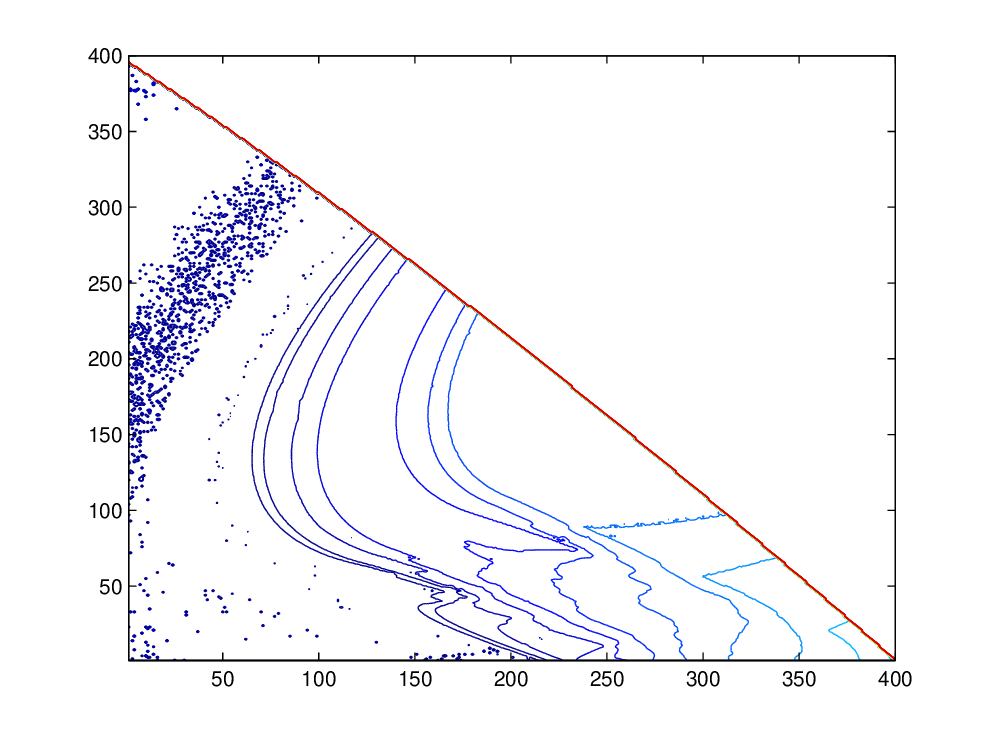}}}
\end{picture}
\end{center}
\caption{\label{fig:cubic_mono} Isentropes for cubic maps
$f_{\aaa,\bbb}(x) = \aaa x^3+\bbb x^2+(1-\aaa)x-\bbb$. The horizontal axis gives $a \in [2.5, 4]$ 
and the vertical axis $\bbb \in [0, \sqrt{4\aaa}-\aaa]$.
The maps $f_{\aaa, -\bbb}$ and $f_{\aaa,\bbb}$ are conjugate, and $f_{\aaa,\bbb}([-1,1]) \not\subset [-1,1]$ for 
$|\bbb| > \sqrt{4\aaa}-\aaa$.}
\end{figure}
 
\medskip
This conjecture was motivated  by numerical experiments, made for the
family of real cubic maps,
of the isentropes and also by considering the {\em {\lq}bones{\rq}} for 
this family.  These bones are curves within the parameter space such that  one critical point is periodic with a specified order type,
 and were introduced by MacKay and Tresser in 
a study of the boundary of chaos for bimodal maps of the interval \cite{MacTr}.
A few years later it was shown in \cite{DGMT} that in the cubic case $b=2$, 
this conjecture follows from another conjecture (density of hyperbolicity for cubic maps).
We should emphasize that although Thurston Rigidity holds for polynomials of any degree, 
this is {\em not} sufficient for proving the monotonicity conjecture for degree $d>2$.
In 2000,  Milnor \& Tresser \cite{MTr} showed that one does not quite need density of hyperbolicity for real cubic maps. 
More precisely, they showed that on some curves in the parameter space, the bimodal family behaves essentially like a one-parameter family of unimodal maps. 
Combining this with an extension of density of hyperbolicity in the quadratic case, due to Heckman \cite{Hec} and  using the Jordan theorem, they were able to conclude the cubic case. 
Using similar techniques as Milnor \& Tresser, 
Radulescu \cite{Radu} proved monotonicity of entropy for
a two-parameter family of quartic polynomials made up of the 
composition of two quadratic maps. 

In this paper we solve this conjecture in full generality: 

\begin{theoremmain}[Milnor's Monotonicity of Entropy Conjecture]
For each $\shape\in \{+,-\}$, $b\in \N$ and $h_0\ge 0$, the isentrope
$$I(h_0)=\{f\in  \PB_\shape \st h_{top}(f)=h_0\}$$
and the set
$$I(h_0^+):= I(h_0)\cap  \mbox{closure}( \{f\in  \PB_\shape \st h_{top}(f)>h_0\})$$ 
are both connected. 
\end{theoremmain}

In particular, the boundary of {\lq}chaos{\rq}, \ie the boundary of the
set of maps in $\PB_\shape$ with positive entropy is connected.
In fact, the proof of the theorem also shows that for each $h_1\le h_0$, 
the set  $$I=\{f\in  \PB_\shape \st h_1\le h_{top}(f)\le h_0\}$$
is connected. 

The set of maps $\PB_\shape \subset \PB$ with shape $\shape$
can be  parametrized by the coefficients of the polynomial, or
more suitably by the critical values of $f$,
see  \cite[Section II.4]{MS} or \cite{MT}. As mentioned, 
one should not expect
that the entropy depends monotonically on any of these parameters.

\begin{remark}
We should emphasise that we prove that the isentrope 
$I(h_0)=\{f\in  \PB_\shape \st h_{top}(f)=h_0\}$, rather than the weaker statement
that the closure of this space is connected. (So we prove connectedness within the space
of maps with non-degenerate critical points, rather than merely in the closure of this space.)
\end{remark}
 
As in Milnor \& Tresser's paper \cite{MTr},  our proof relies on 
stunted sawtooth maps. The other important ingredient is
density of hyperbolicity, but now for real polynomials of arbitrary degree, see \cite{KSS} and \cite{KSS1}.
More precisely, we use an analogue of Thurston Rigidity proved in \cite{KSS} which holds
for all real polynomials with real critical points {\em regardless} of whether the orbits
of the critical points are finite or not.

On the way to proving the Main Theorem, we will also prove the following result (see Theorem~\ref{thm:kneadconn} and Theorem~\ref{thm:Psiproper}):

\begin{theorema}\label{thm:A}
Fix $\shape\in \{+,-\}$, $b\in \N$,  let $f\in \PB_\shape$ and define 
$$\T(f)=\{g\in \PB_\shape \st  g \mbox{ has the same kneading invariants as }f\}.$$
Then $\T(f)$ is connected.
\end{theorema}

\subsection{Related results and some conjectures}
 
In this paper we will also consider the space of so-called admissible stunted sawtooth maps 
 $\SS_*^b$, and show that  isentropes within this space are contractible, see 
Theorem~\ref{Thm:Connected}.
In view of this,  we would like to propose the following
\begin{conjecturen} Any isentrope
$\{f\in  \PB_\shape \st h_{top}(f)=s\}$
is contractible.
\end{conjecturen}

Isentropes in $\PB_\shape$ are extremely complicated. Indeed, we prove in \cite{BSnon} the following

\begin{theoremn} When $b\ge 4$, there are infinitely 
many values for $s\ge 0$ for which $\{f\in  \PB_\shape \st h_{top}(f)=s\}$ is not locally connected. 
\end{theoremn}

In fact, it is not known whether there exists {\em any} value $s\in [0,\log(b+1)]$ so that the corresponding 
isentrope $\{f\in  \PB_\shape \st h_{top}(f)=s\}$
is locally connected. The methods used in the proof of the previous theorem
rely on $b\ge 4$, and it is possible that each isentrope is connected in the cubic case. 

In the survey \cite{Str3} a number of related questions and conjectures are discussed. 
In particular, the following question due to Tresser: 
Consider the space $Pol^d_\epsilon$ of real polynomials $f$ of degree $d$, {\em not necessarily
with all critical points on the real line}, but still with $f(\{\pm 1\})\subset \{\pm 1\}$ and $\epsilon(f)=\epsilon$
as in the definition of $\PB$. 

\begin{conjecture}[Tresser]
Fix $\epsilon\in \{-1,1\}$. Isentropes in $Pol^d_\epsilon$ are connected.
\end{conjecture}

Davoud Cheraghi and the second author have made progress towards this conjecture in 
the context of real unimodal polynomials of degree $4$ with at most one real critical point, 
see \cite{CS}.

\subsection{Acknowledgements}
The authors would like to thank Weixiao Shen who wrote 
Lemma \ref{weixiao2}
and made very helpful comments on earlier versions of the first half of this  paper.
We also would like to thank Charles Tresser, Oleg Kozlovski, Genadi Levin and the referees for their comments.

\section{Strategy of the Proof, organization of the paper and notation}

It is well-known that any multimodal map (with positive topological entropy) 
is entropy-preservingly semi-conjugate to a piecewise monotone map of constant slope,
\cite{Parry, MT}. However, such piecewise affine maps do not
exhibit all possible combinatorial types which exist for polynomials maps.
Instead, one of the ingredients in Milnor \& Tresser's proof  is to consider the space of stunted
sawtooth maps, all obtained from a single sawtooth map $S_0$ 
as in Figure~\ref{fig:sawtooth0}. For example, for each cubic map, there exists a 
stunted sawtooth map $T$ as in this figure  with the same combinatorics. This map $T$ 
is obtained by moving the two plateaus
up or down as appropriate - in a way which is made precise in Section~\ref{sec:Psi}.
In our paper, we will use the space $\SS^b_\shape$ of $b$-modal
stunted sawtooth maps to ``parametrize'' the space $\PB_\shape$ of $b$-modal polynomials.
Indeed, we introduce a map
$$
\ST\colon \PB_\shape \to \SS^b_\shape
$$
which  assigns to $f \in \PB_\shape$ the unique map $T \in  \SS^b_\shape$
which has the same {\lq}kneading invariant{\rq} (\ie symbolic itineraries
of critical points) as $f$.
We discuss the definition of $\ST$ in detail in Section~\ref{sec:Psi}.
An important property of $\ST$ is that $\ST(f)$ and $f$ have the same topological entropy.

\begin{figure}[ht]
\begin{center}
\unitlength=1mm
\begin{picture}(30,35)(-15,-17)
\thicklines
\put(-15,-15){\line(1,0){30}}
\put(-18,-18){$-e$}
\put(-15,-15){\line(0,1){30}}
\put(15,15){\line(-1,0){30}}
\put(15,15){\line(0,-1){30}}
\put(15.5,-18){$e$}
\thinlines
\put(-15,-15){\line(1,4){9.25}}
\put(-5.75,22){\line(1,-4){11}}
\put(15,15){\line(-1,-4){9.25}}
\put(-15.2,-15){\line(1,4){7}}
\put(-8,13){\line(1,0){4.2}}
\put(-3.6,13){\line(1,-4){6}}
\put(2.4,-11){\line(1,0){6}}
\put(8.5,-11){\line(1,4){6.5}}
\thicklines
\put(-9.8,8.2){\line(1,0){7}}
\put(-12,11.3){$T$}
\put(-0.7,6.4){$\tilde T$}
\put(-4.5,20){$S_0$}
\end{picture}
\end{center}
\caption{\label{fig:sawtooth0} Two bimodal stunted sawtooth maps $T$ and $\tilde T$
(drawn in bold lines) constructed from the same sawtooth map $S_0$ (drawn in thin lines).}
\end{figure}
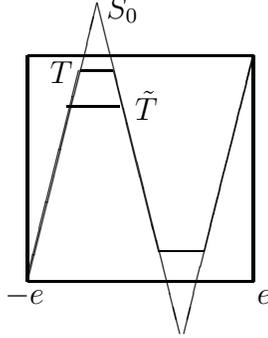

Since we shall fix the shape $\shape$ during the proof, we shall write from now on
mostly $\PB$ and $\SS^b$ instead of $\PB_\shape$ and $\SS^b_\shape$.
One of the crucial benefits of the space $\SS^b$ is that 
outside its plateaus, any map $T\in \SS^b$ agrees with 
the same map $S_0$. This means that all orbits of two stunted
sawtooth maps $T$ and $\tilde T$ agree except on the preimages of their plateaus. 
By decreasing the width of a plateau $Z_i$ (\ie by moving the image $T(Z_i)$
up or down depending on whether $T$ assumes a local maximum or 
minimum at $Z_i$), we create new orbits without destroying any others,
and hence the entropy can only increase. This means that within the space
$\SS^b$, entropy is a monotone function of each critical value separately,
a property which fails for $\PB$, see Remark \ref{remark:nonmono} and \cite{BSnon}. 
Using this idea, Milnor \& Tresser showed the following:

\begin{theorem}[\cite{MTr}] \label{thm:IsentropeSSd}
Isentropes in $\SS^b_\shape$ are connected and contractible. 
\end{theorem}

This result would imply the main theorem if $\ST\colon  \PB_\shape \to \SS^b_\shape$ was a homeomorphism,
but as we shall see that is unfortunately not the case. 

\subsection{Non-bijectivity of $\ST\colon \PB \to \SS^b$}
One of the reasons the map $\ST\colon \PB \to \SS^b$ is non-bijective
is simple to see: if  $f\in \PB$ is hyperbolic (\ie  if each critical point
is in the basin of a hyperbolic periodic attractor), then 
the itinerary of each critical point of $f$ is eventually periodic. 
From the definition of $\ST$ it then follows that the  endpoints of each plateau of $T=\ST(f)$
are also eventually periodic. Since there are uncountably many hyperbolic 
maps in $\PB$ and only countably many such maps $T$,
the map $\ST$ is obviously not injective.
Neither is $\ST$ surjective, see Example~\ref{ex:STdisc}. To overcome this we consider
equivalence classes in $\PB$ and $\SS^b$. 

\subsection{Equivalence classes in $\PB$: partial conjugacy}
The first ingredient aimed at overcoming the fact that  $\ST\colon  \PB \to \SS^b$
is neither injective nor surjective is to define a notion of equivalence classes within these
spaces,  corresponding to  
sets of maps which have the same dynamics except inside the basins of their attractors.
It turns out that we will need slightly different notions within the spaces 
 $\PB$ and $\SS^b$.
For $\PB$, two maps $f$ and $\tilde f$ will be taken to be equivalent if they are
{\em partially conjugate}.  For this to hold, we require that they are
conjugate  (on the real line)  outside the basins $B(f)$ and $B(\tilde f)$ of their periodic
attractors and that they have the same number of critical points
in corresponding components of $B(f)$ and $B(\tilde f)$. In other words, two interval maps $f,\tilde f\colon I\to I$
are partially conjugate, if there exists a homeomorphism $h\colon I\to I$ 
so that $h\circ f=\tilde f\circ h$ holds outside $B(f)$,
so that $h(B(f))=B(\tilde f)$ and so that $h$ maps critical points of $f$ to critical points of $\tilde f$.
Given $f\in \PB$, we define 
$$
\EC(f) \mbox{ to be the set of polynomials $\tilde f \in \PB$ which are partially
conjugate to $f$}.
$$

For more  precise definitions see Section~\ref{sec:defspace}.
Although the space $\mathcal{H}(f)$ (of maps with the same
kneading as $f$) and  $\EC(f)$ are closely related, neither is a subset of the other one, see
Example~\ref{ex:STdisc}.  Extending the rigidity theorems from \cite{KSS1} we obtain

\begin{theorem}[See Theorem~\ref{thm:defspace1}] For any $f\in \PB_\shape$, the set
$\EC(f)$ is connected.
\end{theorem} 
This result is a non-trivial extension of Douady \& Hubbard's result
that hyperbolic components with the space of (complex) quadratic polynomials  are topologically discs.
We emphasize that an important ingredient in the proof 
of this theorem is that all critical points of maps in $\PB$ are real. 
The situation when two real polynomials are conjugate on the real line, 
but have critical points which are outside the real line, is more subtle
and the subject of ongoing research,  see \cite{CS}.

\subsection{Preplateau equivalence in $\SS^b$}
The same definition can also be used to say when  $T,\tilde T\in \SS^b$ are 
partially conjugate. In Lemma~\ref{lem:phpsi} we will see that 
if $f,\tilde f\in \PB$ are partially conjugate then $T=\ST(f),\tilde T=\ST(\tilde f)$ are partially conjugate,
but unfortunately the reverse implication does not hold. This is why we also introduce
another equivalence class within the space $\SS^b$. 
Indeed, in  $\SS^b$ we will define a set $\B(T)$ which deviates slightly from
$B(T)$ and which is based on the preimages of plateaus,  see Section~\ref{sec:preplateaus}.
We then say that $T,\tilde T \in \SS^b$ are equivalent  if $\B(\tilde T) = \B(T)$
and define 
$$
\cell{T}=\{\tilde  T \st \B(\tilde T) = \B(T)\} \,\,  \mbox{ and }\,\, [T]=\mbox{closure}(\cell{T}).
$$
In Theorem~\ref{thm:<T>} we show that $\cell{T}$  and therefore its closure
$[T]$ is connected (in fact, it is a cell). 
From the definitions it follows that all maps within $\EC(f)$ (respectively within $[T]$) 
have the same topological entropy.  

\subsection{The set $\parabolic\subset \PB$ and a useful property of the map $\ST$} 
Unfortunately it is not true that $\ST(\EC(f))\subset [\ST(f)]$ for any $f\in \PB$. 
To address this issue and in order to relate $\EC(f)$ and $[T]$, we  
introduce a subset $\parabolic \subset \PB$,  see Definition~\ref{def_A}, 
which enables us to obtain the following property:

\begin{proposition}[See Proposition~\ref{prop:PsiPH}]
If $f \in \parabolic$ then $\ST(\EC(f))\subset [\ST(f)]$.
Within any $\EC(f)$ we can find special
maps $f_0\in \EC(f)\cap \parabolic$.
\end{proposition}

Because of this proposition we can morally view $\ST$ as a map
which sends equivalences classes (consisting of partially conjugate maps) in $\PB$ to equivalence classes
(consisting of preplateau equivalent maps) in $\SS^b$.

\subsection{Non-surjectivity of $\ST\colon \PB \to \SS^b$ because of wandering pairs}

There is an additional,   more serious way,  
in which $\ST\colon \PB \to \SS^b$ is not surjective, and this has to do with wandering intervals.
An interval $J \subset [-1,1]$ is called {\em wandering} for $f$ if all its iterates 
are disjoint and  $f^n(J)$ does not converge to an attracting periodic orbit as $n\to \infty$. 
It is well-known that polynomial interval maps (indeed $C^2$ interval map with non-flat critical points) 
have no wandering intervals.  This implies that, when $b\ge 3$, there are many stunted 
sawtooth maps $T$ in $\SS^b$ for which there is no
$f\in \PB$ with $T=\ST(f)$. Indeed, take $T$ with the property that 
there exists an interval containing two adjacent plateaus which is eventually mapped into 
a third plateau which is not eventually periodic. Then for any interval map $f$
with $\ST(f)=T$, the interval connecting the corresponding adjacent critical points 
would be wandering. Since a polynomial $f$ does not have wandering intervals,
these adjacent critical points coincide and so that $f$ has only two critical points.  
This in turn implies that $T\ne \ST(f)$.  Therefore by using merely the space $\SS^b$ we could at best 
prove that isentropes within $\PB$ are connected within the larger space of polynomial maps with $\le b$ critical points. 

\subsection{The space of non-degenerate stunted sawtooth maps $\SS^b_*$}
To overcome this problem we define the notion of {\em wandering pairs} of plateaus,
in Definition~\ref{def:wander}, and introduce the space $\SS^b_*$
of stunted sawtooth maps without wandering pairs.
In turns out that $\ST$ maps $\PB$ into $\SS^b_*$. 
The topology of $\SS^b_*$ is much more complicated than that 
of $\SS^b$, which makes it tricky to show that
isentropes in $\SS^b_*$ are contractible.
In $\SS^b$, this is much easier: within $\SS^b$ one can construct a retract
of an isentrope by moving plateaus with relatively great liberty.
To construct a retract for isentropes in $\SS^b_{*}$ we are forced
to move plateaus in exactly the right order and with exactly the right speed.
The description of this occupies most of Section~\ref{sec:connS*},
and leads to

\begin{theorem}[Connectivity of isentropes in  $\SS^b_{\shape,*}$, see
 Theorem~\ref{Thm:Connected}] Isentropes in $\SS^b_{\shape,*}$
are connected and even contractible.
\label{Thm:ConnectedI}  
\end{theorem}

\begin{remark}
The construction in Theorem~\ref{Thm:Connected} of this contraction is quite subtle.
Nevertheless it would be interesting to explore whether  one can 
use the same methodology to construct deformations within the space
$\PB$ (and show that isentropes within that space are contractible).
  \end{remark}

\subsection{The main steps in the proof}

With these notions in place,  we will obtain that $\ST$ is {\lq}almost{\rq} surjective and injective:

\begin{proposition}[$\ST$ is {\lq}almost{\rq} surjective, see Proposition~\ref{prop:realize}]
Take  $T\in \SS^b_*$ and let $\tilde T\in \cell{T}_{\natural}$.
Then there exists a polynomial $f\in \PB\cap \parabolic$
such that $\tilde T=\ST(f)$ and  $T\in [\ST(f)]$.
\end{proposition}

\begin{proposition}[$\ST$ is almost injective, see Proposition~\ref{prop:inj}]
The map $\ST\colon \PB \to \SS^b_*$
is `almost injective' in the sense that
if $f_1,f_2 \in \parabolic$ and $[\ST(f_1)]\cap [\ST(f_2)]\ne \emptyset$,
then $\overline{\EC(f_1)}\cap \overline{\EC(f_2)}\ne \emptyset$.
\end{proposition}

To prove the latter proposition, it is important to analyze how two
subsets $\EC(f)$ and $\EC(\tilde f)$ of the space of polynomials $\PB_\shape$  can intersect.  
It is for this reason that we require a description of what bifurcations occur at the 
boundary of these sets, see
Theorem~\ref{thm:defspace2}. A corresponding description 
for the boundary of $\cell{T}$ within the space  $\SS^b_\shape$
is also needed, and is given in  Theorem~\ref{thm:<T>}.

\begin{proposition}[$\ST$ is almost continuous, see Proposition~\ref{prop:cont}]
$\ST\colon \PB \to  \SS^b_*$ is `almost continuous'
 in the following sense.
Assume that $f_n\to f$ where $f_n,f\in \PB$ and $f_n\in \parabolic$.
Then there exists $T \in \SS^b_*$ so that $\ST(f_n)\to [T]$
and $\ST(f)\in [T]$.
\end{proposition}

Combined this gives the following: 

\begin{theorem}\label{thm:homeolike}
There exists a map $\ST\colon \PB_\shape \to \SS^b_{\shape,*}$
such that
\begin{itemize}[topsep=-1cm,itemsep=0.05ex,leftmargin=0.6cm]
\item $\ST$ is `almost continuous', `almost surjective' and
`almost injective' (as described in the previous three propositions);
\item There exists a connected set $[\ST(f)]\ni \ST(f)$
such that the topological entropy of
any map $T\in [\ST(f)]$ is equal to the topological entropy of
$f$;
\item If $K$ is closed and connected then
$\ST^{-1}(K)=\{f \st [\ST(f)]\cap K\ne \emptyset\}$ is connected.
\end{itemize}
\end{theorem}

Isentropes in  $\SS^b_{\shape,*}$ are contractible (and therefore connected), so
Theorem~\ref{thm:homeolike} implies that  isentropes in $\PB_{\shape,*}$ are connected, 
proving the Main Theorem.

\subsection{Organization of the paper}
Section~\ref{sec:defspace} discusses the notion of partial conjugacy 
and shows that partial conjugacy classes $\EC(f)$ within the space $\PB$ of polynomials 
are connected. It also 
describes  when   different sets $\EC(f)$ and $\EC(\tilde f)$ have common boundary points. 
 This section relies on methods which use complex analysis and results from the
 theory of holomorphic dynamics. In particular, this section it relies on a theorem on rigidity of
 real polynomials. The remainder of the paper only uses real methods. 
Section~\ref{sec:combinatorialmodel} discusses the space $\SS^b$ of stunted sawtooth maps
and properties of the equivalence classes $\cell{T}$. 
Section~\ref{sec:Psi} discusses the map $\ST\colon \PB\to \SS^b$
and its properties.  In Section~\ref{sec:proofofthm} the proof of the Main Theorem is provided. 
The technical result that isentropes in $\SS_*^b$ are connected (and even contractible) is proved in Section~\ref{sec:connS*}.

\subsection{Notation used in this paper}
\begin{itemize}[topsep=-0.2cm,itemsep=0.05ex,leftmargin=0.2cm]
\item[-] $B(f)$ is  the union of the basin of periodic attractors, see  Definition~\ref{def:basin}.
\item[-] $\D$ denotes the open unit disc in $\C$.
\item[-] $I$ is an interval in $\R$.
\item[-] $Z_i$ are (closed) plateaus of the stunted sawtooth map, see 
Section~\ref{sec:sawtooth}. 
\item[-] $\PB$ is the space of real polynomials of degree $b+1$ mapping $I$ (and $\partial I$) into itself,
with precisely $b$ non-degenerate  critical points each of which is contained in the interior of $I$. 
\item[-] $\PB_\shape\subset \PB$ is the space of maps which is increasing (respectively decreasing)
at the left endpoint of $I$ when $\shape=1$ (respectively $\shape=-1$). 
\item[-] $\T(f)$ is the space of maps $g\in \PB$ with the same kneading invariant as $f$, see Theorem A in the introduction.
\item[-]  $\EC(f)$ is the space of maps $g\in \PB_\shape$ which are partially conjugate  to $f$, see Definition~\ref{def:partiallyconjugate}.
\item[-] $\PartHyp(f)$ is the space of Kupka-Smale maps, see Definition~\ref{def:KS}. 
\item[-] $S_0$ is a sawtooth map of modality $b$, see Section~\ref{sec:sawtooth}.
\item[-] $\SS^b$ is the space of stunted sawtooth maps, see 
Section~\ref{sec:sawtooth}.
\item[-] $\B(T)$ is the set of points which are eventually mapped into the interior of a block of plateaus, see equation \eqref{eq:defBT}. 
\item[-] $\langle T \rangle$ is the set of maps with the same $\B(T)$, see equation 
\eqref{<T>}.
\item[-] $[T]$ is the closure of $\langle T \rangle$, see equation \eqref{<T>}.
\item[-] $\min [T]$ and $\cellN{T}$ are certain subsets of the boundary of $\langle T \rangle$, see Definition~\ref{def:naturalbullet}. 
\item[-] $\J:=[Z_i,Z_j]$ is the convex hull of plateaus $Z_i$ and $Z_j$,
see Definition~\ref{def:wander}. 
\item[-]$\SS^b_*$ is the space of non-degenerate maps in 
 $\SS^b$, see Definition~\ref{def:wander}. 
\item[-] $\ST\colon \PB\to \SS^b$ is the map which assigns to a polynomial
a stunted sawtooth map,  see equation \eqref{eq:defST} in Section~\ref{sec:Psi}.
\item[-] $\parabolic$ is a subset of polynomials with parabolic periodic points, see Definition~\ref{def_A}.
\item[-]$M_{n,\shape}$, $M_{n,\shape}^o,M_{n,\shape}^\Sigma$ are spaces of Blaschke products, see Definition~\ref{def:Mn}.
\item[-] $\Ei_t$, $\ei_t$, $\Eih_t$ are
entropy increasing deformations, see Section~\ref{subsec:increase}
and \ref{subsec:eih}.
\item[-] $\ed_t$, $\hed_t$ and $\Edh_t$ are entropy decreasing deformation, see
Sections~\ref{subsec:deltat} and \ref{subsec:careful}.
\item[-] $K_i$ and $\hat K_i$ are periodic intervals related to the $i$-th plateau of $T$, see Section~\ref{subsec:careful}.
\item[-] $\beta_t$ is an entropy preserving deformation, see Section~\ref{subsec:betat}
\item[-] $R_t$ and $r_t$ are retracts, see Sections~\ref{subsec:h0=0},
\ref{subsec:trouble} and \ref{subsec:deltat}. 
\end{itemize}

\label{sec:ingredients}
\section{The partial conjugacy class of maps in $\PB$ is connected}

\label{sec:defspace}

As usual, we say that a polynomial $f$ is {\em hyperbolic} if each of its periodic orbits
is hyperbolic and each of its critical points lies in the basin of a  periodic attractor. 
A well-known result due to Douady \& Hubbard  asserts that each connected component of the 
set $\{c \in \C; \,\, q_c(z)=z^2+c\mbox{ is hyperbolic}\}$ is topologically an open disc 
parametrised by the multiplier of the periodic attractor.
The corresponding case for polynomials of higher degree was considered in \cite{Mil1}.
In \cite{Epstein,Levin}  it was shown that the multipliers of non-repelling periodic points are 
independent parameters.  In this section we will generalize these results to polynomials of higher degree
with the crucial difference that we no longer assume that each critical point
is in the basin of hyperbolic periodic attractors and restrict  to partial conjugacy classes (defined below).
We shall only prove this  generalization for real 
polynomials, because one of the main ingredients we need is a  rigidity result 
which is only known in that context. 

Before stating this generalization we will introduce some terminology. 

\begin{definition}[Basin of an interval map]\label{def:basin}
We say that a periodic orbit $O$ of an interval map $f\colon I\to I$ is {\em attracting}
if its basin $B^s(O)=\{x;f^n(x)\to O\mbox{ as }n\to \infty\}$
contains a (possibly one-sided) neighborhood of $O$. 
Let $B(f)$ be the union of the basins of periodic attractors
of $f$, \ie $B(f)$ consists of all points $x$ so that
$f^n(x)$ tends to a (possibly one-sided) periodic attractor.
Note that if $f$ has a {\em neutral} periodic point (\ie a periodic point which is non-hyperbolic), $B(f)$ need not be open.
When $f$ is a polynomial we also will  consider $f$ as acting on the complex plane, and 
in order to emphasize this we sometimes write $B_\C(f)$ to denote
the basin of $f$ as a subset of $\C$.  
\end{definition}

\begin{definition}[Partially conjugate] \label{def:partiallyconjugate}
We say that two $b$-modal maps $f,g\colon I \to I$
are {\em partially conjugate} if
there is an orientation preserving homeomorphism $h\colon I\to I$
such that
\begin{itemize}[topsep=-0.2cm,itemsep=0.05ex,leftmargin=1cm]
\item $h$ maps $B(f)$ onto $B(g)$;
\item $h$ maps the $i$-th critical point of
$f$ to the $i$-th critical point of $g$;
\item
$h\circ f(x) = g \circ h(x) \mbox{ for all }x\notin B(f)$.
\end{itemize}
We denote by $\EC(f)$ the
set of maps $g\in \PB_\shape$ which are partially conjugate  to $f$. 
\end{definition}

Note that when a critical point
is eventually mapped to the boundary of a component of $B(f)$ that this property persists within $\EC(f)$.
\begin{definition}[Kupka-Smale maps]\label{def:KS}
Let $ \PartHyp$ be the set of $g \in \PB_\shape$  which are 
Kupka-Smale
in the sense that 
\begin{itemize}[topsep=-0.2cm,itemsep=0.05ex,parsep=0cm,leftmargin=1cm]
\item  $g$ has only hyperbolic periodic points and 
\item $g$ has no homoclinic orbits, \ie  no 
critical point of $g$ is mapped to the boundary of a component of
$B(g)$. 
\end{itemize}
\end{definition}
\vskip -0.2cm

The set 
$$
\EC^o(f):=\EC(f)\cap \PartHyp.
$$
generalizes the notion of hyperbolic component for quadratic maps
allowing, for the situation that some critical points are not attracted to periodic attractors provided
the dynamics of such critical points agrees for all maps within $\EC^o(f)$. 
Note that $f\in \PartHyp$ does {\bf not} imply $\EC^o(f)= \EC(f)$ 
because even in this case $\EC(f)$ can contain maps with neutral periodic orbits.

The main result in this section is the following theorem and its more detailed version
Theorem~\ref{thm:defspace2}.

\begin{theorem}[Connectedness of $\EC(f)$]\label{thm:defspace1}
Let $f\in \PB$. 

\begin{itemize}[topsep=-0.2cm,itemsep=0.05ex,parsep=0cm,leftmargin=1cm]
\item  If $f\in \PartHyp$ then  $\EC^o(f)$ is  homeomorphic
to an open ball of dimension equal to the number of critical points in $B(f)$.
\item   $\EC(f)\subset \overline{\EC^o(f)}$ and therefore the set $\EC(f)$ is connected. 
\end{itemize}
\end{theorem}

In fact, we shall also need  Theorem~\ref{thm:defspace2} which states that
for any $f\in \PB$ one can find a continuous family of maps $f_\mu$, $\mu\in [0,1]$
with $f_\mu\in \PartHyp$ for $\mu\in (0,1]$ and  $f_0=f$  with the crucial additional property 
that $f_\mu$ has the same dynamics as $f_0$ outside the basins.
Before stating that theorem more formally, 
let us clarify what the types of non-hyperbolic periodic points can occur for maps within the space $\PB$. 

\begin{lemma}
Let $f\in \PB$. Then each attracting or neutral periodic point of $f$ is real and contains 
a critical point in its basin. Moreover,  if $p$ is a neutral point (say of minimal period $n$) then
it attracts at least from one side and is of one of the following types:
\begin{enumerate}[topsep=0cm,itemsep=0.05ex,leftmargin=1cm]
\item[{\bf (pd)}]\,\, $p$ is \emph{attracting from both sides} with multiplier $-1$
and up to a change of coordinates $f^{2n}$ has the form 
$x\mapsto x-x^3+O(x^4)$ near $p$ . 
\item[{\bf (pf)}] $p$ is \emph{attracting from  both sides}  with multiplier $1$
and  up to a change of coordinates $f^{n}$ has the form 
$x\mapsto x-x^3+O(x^4)$  near $p$. 
\item[{\bf (sn)}] $p$ is  \emph{one-sided attracting} with multiplier $1$ and 
up to a change of coordinates $f^{n}$ has the form 
$x\mapsto x-x^2+O(x^3)$  near $p$. 
\end{enumerate}
\end{lemma}
\begin{proof}
Here we use that maps in $\PB$ are real and only have real critical points. 
This condition implies that if $p$ is a periodic orbit of period $n$
with $Df^n(p)=1$ then it must be attracting one from side (otherwise the attracting
petals will not intersect the real line, but this is impossible since all critical 
points lie in the real line). Using that each attracting forward invariant petal of a neutral periodic point 
contains a critical point, the result follows. Alternatively, one can use the fact that maps $f\in \PB_\shape$ 
have negative Schwarzian derivative, \ie 
$Sf=\left[ f'\cdot f''' -(3/2) (f'')^2\right] /(f')^2<0$. Since this implies that the Schwarzian derivative of $f^n$
and $f^{2n}$ are negative,  see Exercise IV.1.7 in \cite{MS}, the required statement follows.
\end{proof}

\begin{figure}[htp]
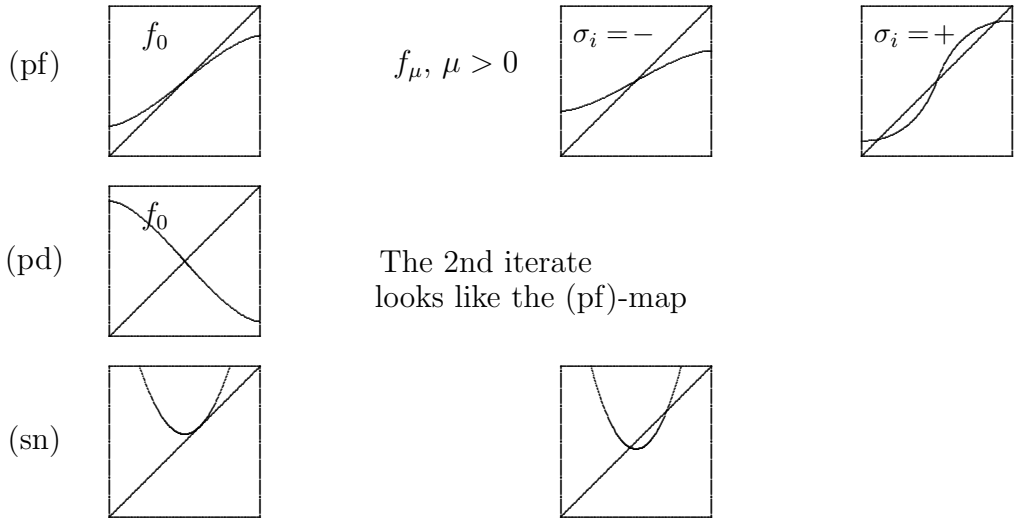
 \hfil
\beginpicture
\dimen0=0.2cm
\setcoordinatesystem units <\dimen0,\dimen0>  point at 60 0
\setplotarea x from 0 to 10, y from 0 to 10
\put {(pf)} at -5 6
\put {$f_0$} at 3 8 \setlinear
\plot 0 0 10 0 10 10 0 10 0 0 10 10 /
\setquadratic 
\plot 0 2 2 2.8 5 5 / 
\plot 10 8 8 7.2 5 5 / 
\setcoordinatesystem units <\dimen0,\dimen0>  point at 30 0
\setplotarea x from 0 to 10, y from 0 to 10
\put {$f_\mu$, $\mu>0$} at -7 6 
\put {{\small $\sigma_i=\! -$}} at 3.5 8 
\setlinear
\plot 0 0 10 0 10 10 0 10 0 0 10 10 /
\setquadratic 
\plot 0 3 2 3.5 5 5 / 
\plot 10 7 8 6.5 5 5 /
\setcoordinatesystem units <\dimen0,\dimen0>  point at 10 0
\setplotarea x from 0 to 10, y from 0 to 10
\put {{\small $\sigma_i=\! +$}} at 3.5 8 
\setlinear
\plot 0 0 10 0 10 10 0 10 0 0 10 10 /
\setquadratic 
\plot 0 1 3 2  5 5 / 
\plot 10 9 7 8 5 5 / 
\setcoordinatesystem units <\dimen0,\dimen0>  point at 60 12
\setplotarea x from 0 to 10, y from 0 to 10
\put {(pd)} at -5 5 
\put {$f_0$} at 3 8
\setlinear
\plot 0 0 10 0 10 10 0 10 0 0 10 10 /
\setquadratic 
\plot 0 9 2  8 5 5 8 2 10 1  / 
\put {The 2nd iterate} at 25 5 
\put {looks like the (pf)-map} at 28 2.5
\setcoordinatesystem units <\dimen0,\dimen0>  point at 60 24
\setplotarea x from 0 to 10, y from 0 to 10
\put {(sn)} at -5 5 
\setlinear
\plot 0 0 10 0 10 10 0 10 0 0 10 10 /
\setquadratic 
\plot 2 10  5 5.5  8 10 / 
\setcoordinatesystem units <\dimen0,\dimen0>  point at 30 24
\setplotarea x from 0 to 10, y from 0 to 10
\setlinear
\plot 0 0 10 0 10 10 0 10 0 0 10 10 /
\setquadratic 
\plot 2 10  5 4.5  8 10 / 
\endpicture
\caption{\label{fig:bifurcations}
{\small A map with a pf (pitchfork), pd (period doubling) respectively sn (saddle-node) fixed point and the corresponding unfoldings constructed in Theorem~\ref{thm:defspace2}.}}
\end{figure}

These case are described in Figure~\ref{fig:bifurcations}.
One of the main ingredients in this paper is the following theorem which shows that 
one can find a deformation $f_\mu, \mu\in [0,1]$ of a any map $f\in \PB\setminus \PartHyp$ 
so that $f_\mu\in \PartHyp$ for $\mu\in (0,1]$ and so that
$f_\mu\in \EC(f_1)$ for each $\mu \in (0,1]$. Therefore the only bifurcations of $f_\mu$ are
that basins can merge or split and that neutral orbits become hyperbolic. 

\begin{theorem}[Bifurcations of maps which are not in $\PartHyp$]\label{thm:defspace2}

Assume that $f\notin \PartHyp$ and let $O_1,\dots,O_k$
be the non-hyperbolic periodic orbits of $f$ and pick a periodic point $p_i\in O_i$ for each $i=1,\dots,k$. 
For each periodic orbit $O_i$, choose a sign $\sigma_i\in \{-,+\}$. 
Then there exists a family  $f_\mu\in \PB_\shape$ and periodic points $p_{i,\mu}$ 
all depending continuously on $\mu\in [0,1]$ so that 
$f_0=f, p_{i,0}=p_i$ and
$$
f_\mu\in \EC^o(f_1)\mbox{ for each }\mu\in (0,1].
$$
If $p_i$ has (minimal) period $n_i$ and 
\begin{enumerate}[topsep=0cm,itemsep=0.05ex,leftmargin=1cm]
\item $p_i$ is attracting from both sides and has multiplier $-1$ (the (pd)-case), then depending on the sign of 
$\sigma_i$ 
a   {\em period doubling} or a {\em period halving bifurcation} occurs as $\mu$
becomes positive;   \ie up to a parameter dependent 
coordinate change $f^{2n_i}_\mu$  has  for each $\mu\in [0,1]$ 
near $p_{i,\mu}$
the form $x\mapsto (1+\sigma_i\mu)x-x^3+\mbox{h.o.t.}$;
\item $p_i$ is attracting from  both sides with multiplier $1$ (the (pf)-case), then depending on the sign of 
$\sigma_i$  a {\em pitch-fork} or a 
{\em reverse pitch-form bifurcation} occurs as $\mu$
becomes positive; \ie  up to a  parameter dependent 
coordinate change $f^{n_i}_\mu$  has for each $\mu\in [0,1]$  near $p_{i,\mu}$ the 
form  $x\mapsto (1+\sigma_i \mu)x-x^3+\mbox{h.o.t.}$;
\item $p_i$ is  one-sided attracting and has multiplier $1$ (the (sn)-case), then one can {\em create a saddle-node pair}
for $f_\mu$ as $\mu$ becomes positive; \ie  up to a  parameter dependent coordinate change $f^{n_i}_\mu$  has for each $\mu\in [0,1]$ 
near $p_{i,\mu}$ 
the  form  $x\mapsto (1+\mu)x-x^2+\mbox{h.o.t}$.
\item If $p_i$ is  one-sided attracting and has multiplier $1$ and there exists a neighborhood $U$ of $p_i$
so that $U\setminus \{p_i\}$ is contained in $B(f)$, then  one can {\em create} or {\em destroy a saddle-node pair}
for $f_\mu$ as $\mu$ becomes positive; \ie  up to a  parameter dependent coordinate change $f^{n_i}_\mu$  has for each $\mu\in [0,1]$ 
near  $p_{i,\mu}$ the  form  $x\mapsto x-x^2+\sigma_i \mu + \mbox{h.o.t}$.
\item Moreover, if a critical point $c$ is eventually mapped into the boundary of a component of $B(f)$, 
then $f_\mu(c)$ is contained in the interior of $B(f_\mu)$ for each $\mu\in (0,1]$. 
\end{enumerate}
\end{theorem}

This theorem asserts that one can find a family of maps $f_\mu\in \PartHyp$
when $\mu\in (0,1]$  so that the two adjacent parabolic petals of $f=f_0$ at a neutral periodic point
$p_i$ of type (pf) and (pd) as in case (1) and (2) correspond for $\mu>0$ to two adjacent hyperbolic basins when $\sigma_i=+$
or to one hyperbolic basin when $\sigma_i=-$ (so the sign of $\sigma_i$ determines whether or not one 
has a reverse period doubling of pitchfork bifurcation).   
In case (4) two basins which touch at a 
a saddle-node orbit are deformed into a map where the basins touch at a repelling orbit (when $\sigma_i=+$)
or merge into one hyperbolic basin (when $\sigma_i=-$). This situation corresponds
to the fixed point with multiplier $=1$ in Figure~\ref{fig:twobasins}.

In order to clarify case (5) in the theorem,
consider the situation  that  $f(c)$ is contained in the boundary of a component $B$ of $B(f)$. 
Then the basin $B_\C(f)$ has two components $B_1,B_2$ which touch at $c$.
These components $B_1,B_2$ can lie symmetrically in the upper and lower half plane or 
to the left and right of $c$. Which case occurs depends on whether $f(c)$ is a left or right end point of $B\cap \R$ 
and whether $f$ has a maximum or a minimum at $c$. A situation with the latter case is shown
in Figure~\ref{fig:twobasins}, see also Figure~\ref{fig:step23} where $U$ plays the role of $B\cap \R$.

\begin{figure}[h!]
\begin{center}
\unitlength=4mm
\begin{picture}(30,7)(0,0)
\put(9.3,-2.5){\resizebox{4cm}{4cm}{\includegraphics{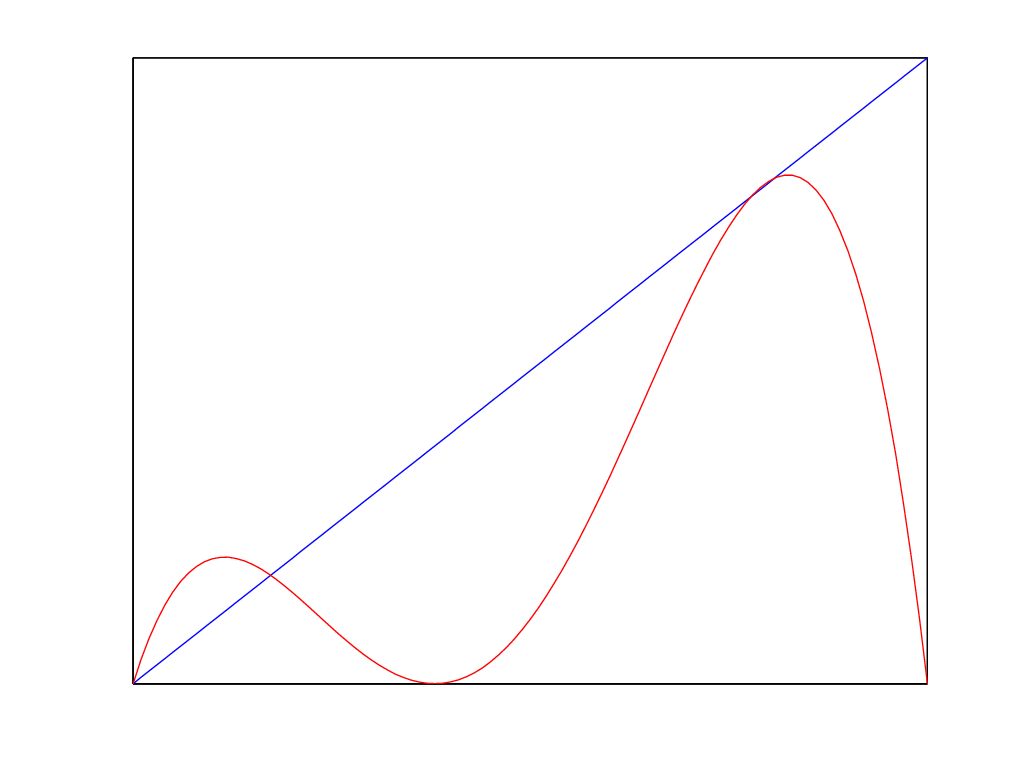}}}
\end{picture}
\end{center}
\vskip 0.3cm
\caption{\label{fig:twobasins}  For this map, there are three
degeneracies: two parabolic fixed points (one with multiplier $-1$ 
and another one with multiplier $1$) and a homoclinic orbit, \ie a  critical point which is mapped
to the boundary of $B(f)$. In this case, both sides of the one-sided attracting fixed point with multiplier $1$
are contained in $B(f)$. Theorem~\ref{thm:defspace2} shows that such a map can be 
embedded in a family of maps $f_\mu$, $\mu\in [0,1]$ so that for $\mu\in (0,1]$ 
the map $f_\mu$ only has hyperbolic periodic points (in the situation shown in the figure, a unique hyperbolic fixed point)
and so that the two basins are merged.}
\end{figure}

\subsection{Underlying rigidity theorems}

Before proving Theorem~\ref{thm:defspace1}, we should remark
that it is related to (and extends) the following result:

\begin{theorem}[Rigidity Theorem, see \cite{KSS}]\label{KSS}
Let $f\in \PB_\shape$. Assume that $f$ has no periodic attractors.
Then $\EC(f)=\{f\}$.
\end{theorem}

For the quadratic case, Theorem~\ref{KSS} was proved  independently
by Lyubich \cite{Lyu} and  Graczyk \& \'Swi\c{a}tek \cite{GS}.
Milnor \& Tresser used this result, or rather a version which
applies to certain cubic maps due to Heckman, a PhD student of \'Swi\c{a}tek (see  \cite{Hec}),
in their proof of the Main Theorem for the cubic case, see \cite{MTr}. 

Let us also note that Theorem~\ref{KSS} is related to {\em density of hyperbolicity}.
We say that an interval map $f$ is {\em hyperbolic} if each critical point of $f$ is in the basin 
of a hyperbolic periodic attractor. Building on Theorem~\ref{KSS} it was shown in \cite{KSS1} that one has
density of hyperbolicity:
each real polynomial  can be approximated by a hyperbolic real polynomial of the same degree. 
In fact,  each $C^\infty$ interval map can be approximated in the $C^\infty$ topology 
by  a hyperbolic $C^\infty$ map, see \cite{KSS1}; it follows for example that 
within generic one-parameter families of interval maps hyperbolic maps are dense, see \cite{Str2}.

The main ingredient in the proof of Theorems~\ref{thm:defspace1} and \ref{thm:defspace2}
is the following:

\begin{theorem}[Generalized Rigidity Theorem, see \cite{KSS}]\label{KSSg}
Let $f,g\in \PB_\shape$ and assume that $f,g$ are partially 
conjugate on the real line. Moreover, assume that for each periodic attractor (or parabolic point)
of $f$ there exists a conformal map $h\colon B^f_\C\to B^g_\C$ from the basin $B^f_\C$ of this periodic attractor 
to the basin $B^g_\C$ of the corresponding periodic attractor of $g$ so that $h\circ f=g\circ h$ on $B^f_\C$.
Then $f=g$.
\end{theorem} 

\begin{proof} 
This follows from the Rigidity Theorem' on page 751 of  \cite{KSS}
which states that
$f$ and $g$ are quasiconformally conjugate
(in this theorem parabolic periodic points are allowed.
By the assumption, we can modify this conjugacy
away from the Julia set to a quasiconformal homeomorphism which agrees with the conformal conjugacy 
outside a small neighborhood of the Julia set and with the original 
conjugacy on the forward orbits of the critical points (and which preserves the real line). 
Using the usual pullback argument, one then obtains a quasiconformally conjugacy which 
is conformal on the Fatou set. 
Since the Julia set of a map in $\PB_\shape$ does not carry an invariant line field, see Theorem 1 in \cite{Shen_C2}, 
it follows that the quasiconformal conjugacy must be conformal.
By the normalization imposed on maps in $\PB_\shape$ it follows that $f=g$.
\end{proof}

\subsection{The set of Blaschke products with real critical points forms a ball}\label{SecBlachke}
To show that $\EC(f)$ is connected, our strategy is to
prove first that $\EC^o(f)$ is  connected.
To this end we shall identify $\EC^o(f)$ with a space
of Blaschke products by means of quasi-conformal surgery.

\begin{definition}\label{def:Mn}
For any $n\ge 1$ and $\shape\in \{+, -\}$, let $M_{n,\shape}$
denote the set of all proper (\ie the inverse of a compact set is compact)
holomorphic maps $A: \D\to\D$ of degree $n$
of the open unit disc $\D$, preserving the real axis, such that
$A$ has $n-1$ distinct real critical points in $(-1, 1)$, and such that
the sign of $A'(-1)$ is $\shape$.
Since $A \in M_{n,\shape}$ maps $\R$ into itself, it can be written as $z\mapsto
\sigma \prod_{i=1}^{n} \frac{z-a_i}{1-\bar{a}_iz}$,
where $\sigma\in \{-1,1\}$ depending on $\shape$ and on the parity of $n$,
and where $\{a_i, 1\le i\le n\}$ is a subset of $\D$ which is symmetric
with respect to the real axis. 
Therefore $M_{n, \shape}$ can be considered as
a subset of $\D^n$
and is thus supplied with the induced topology.
\end{definition}

Not every map $A\in M_{n,\shape}$ has a fixed point, 
but by the Schwarz-Pick lemma any holomorphic map $A\colon \D\to \D$ has at most one fixed point. If 
$A$ has finite degree then it extends to maps $\partial \D$ to itself. If $A\colon \D \to \D$ 
does not have a fixed point in $\D$ then it follows from Denjoy-Wolff that there exists a unique fixed point on 
$\partial \D$ which attracts all points in $\D$.

\begin{definition}
Let $M_{n,\shape}^o$  (respectively $M_{n,\shape}^\Sigma$) be the set of maps $A\in M_{n,\shape}$ with the additional
property that $A(0)=0$ (respectively so that $c_1=0$) where
$-1<c_1<\dots<c_{n-1}<1$ are the critical points of $A$.
\end{definition}

When $n\ge 3$, the assumption that 
$\{a_i, 1\le i\le n\}$ is a subset of $\D$ which is symmetric
with respect to the real axis  does {\em not} imply that all critical points of $A$ are real. 

\begin{lemma}\label{lem:Mconn}
$M_{n,\shape}^o$ and $M_{n,\shape}^\Sigma$ are  
 homeomorphic to an open  Euclidean ball with (real) dimension equal to the number
 of critical points of maps in $M_{n,\shape}$, \ie equal to $n-1$. Moreover, 
$M_{n,\shape}$ is homeomorphic to an open Euclidean ball of (real) dimension $n$.
\end{lemma}
\begin{proof}
Using the same argument as the proof of Lemma 3.1 in \cite{MT} or of Corollary II.4.1 in \cite{MS}, 
one sees that maps in $M_{n,\shape}^o$ can be reparametrized by their critical values.
Since  for any map $A\in M_{n,\shape}^o$ there exists a unique
M\"obius transformation $M$ so that $A\circ M\in M_{n,\shape}^\Sigma$
it follows that the space $M_{n,\shape}^\Sigma$ also has dimension $n-1$.
Since any map $A \in M_{n,\shape}$ can be written in the form $B\circ M$
where $B\in M_{n,\shape}^o$ and $M$ is a M\"obius transformation, 
this implies that the space $M_{n,\shape}$ is homeomorphic to a Euclidean ball of dimension $n$.
Alternatively, this follows from the fact that $M_{n,\shape}^o$ can be parametrized
by their critical points, see  \cite{Za} and that any $A \in M_{n,\shape}$ can be written in the form $M\circ B$
where $B\in M_{n,\shape}^o$ and $M$ is a M\"obius transformation.
\end{proof}

\subsection{The  set $\EC^o(f)$ is homeomorphic to a ball}\label{SecPHo}

Take $f\in \PartHyp$ and let us associate spaces $\mathcal{M}(f)$ and $\mathcal{M}^o(f)$
to $\EC^o(f)$. For this we will consider $f\in \PB\cap \PartHyp$ as a map
acting on the complex plane, and define $B_\C(f)$
as the set of points in the complex plane
whose iterates converge to periodic attractors (or parabolic points)
of $f$. Let
$U_1, U_2, \dots, U_m$ be the components of $B_\C(f)\subset \C$
which contain critical points, and let $n_i$ be the number
of critical points in $U_i$. For each $i$ let $s_i$ be the
minimal positive integer
such that $f^{s_i}(U_i)=U_{i'}$ for some $1\le i'\le m$.
Note that it is conceivable that some components $U_i$ are backward iterates
of the immediate basin of the periodic attractor.
Let us consider the space $\mathcal{M}(f)=\prod_{i=1}^m M_{n_i, \shape_i}$,
where $\shape_i$ denotes the sign of $(f^{s_i})'$ at the left endpoint
of $B_i := U_i \cap \R$. (If this endpoint is a critical point,
then $\shape_i$ is the sign of
the second derivative at this point.)
An element $\bfA=(A_1, \ldots, A_m)\in \mathcal{M}(f)$
will be viewed as a dynamical system
on the disjoint union of $m$ copies of the unit disk,
$$
\bfA: \bigcup_{i=1}^m \D_i\to \bigcup_{i=1}^m \D_i,\ \mbox{ where }\D_i=\D\times \{i\}
$$
such that
$\bfA(z, i)=(A_i(z), i')$, where $i'$ is as above.
Let us say that $\bfA\sim\tilde\bfA$ if they are conjugate to each other
via a component-preserving conformal map
$\phi: \bigcup \D_i\to\bigcup \D_i$ such that
for each $1\le i\le m$, $\phi|\D_i$ is a real symmetric
(\ie $\phi(z) = \overline{\phi(\overline z)}$) map
whose restriction to the real line preserves the orientation.

Let $\mathcal{M}^o(f)$ denote the subset of $\mathcal{M}(f)$ consisting of
maps $\bfA=(A_1, A_2, \ldots, A_m)$ with the following property: if
$\bfA^k$  maps $\D_{i}$ onto itself, then $\bfA^k$ has a fixed
point in $\D_i$. In other words, if $U_{i_1}, U_{i_2}, \cdots, U_{i_k}$
is a cycle of attracting basins of $f$, then we require that
$A_{i_k}\circ \dots \circ A_{i_1}$ has a fixed point.
This means that up to a M\"obius transformation
we can assume that the periodic points in $U_i$ correspond to $0$.
It follows that
$\mathcal{M}^o(f)/\sim \ =\prod_{i=1}^m M^{\delta_i}_{n_i,\eps_i}$ where 
$M^{\delta_i}_{n_i,\eps_i}=M^o_{n_i,\eps_i}$ when $U_i$ contains a  periodic attractor
and $M^{\delta_i}_{n_i,\eps_i}=M^o_{n_i,\eps_i}$ otherwise.

Let us define a map
\begin{equation}\Theta: \EC^o(f) \to \mathcal{M}^o(f) / \sim 
\label{eq:theta0}\end{equation}
as follows.
For $g\in \EC^o(f)$, let $U_i(g)\subset \C$, $i=1,2,\ldots, m,$ be the
components of $B_\C(g)$ containing critical points corresponding to the sets $U_i$ from above. For each $i$, let $\phi_i: U_i(g)\to \D$ be some
real-symmetric conformal map whose restriction to the real axis is
orientation-preserving, and
let 
\begin{equation}
A_i(g)=\phi_{i'}\circ g^{s_i}\circ \phi_i^{-1}.
\label{eq:A}
\end{equation} Then define
\begin{equation}\Theta(g)=[(A_1(g), A_2(g), \ldots, A_m(g))],
\label{eq:theta}\end{equation}
where $[A]$ denotes the equivalence class of $A$.
Note that the space $\mathcal{M}^o(g)$ associated to any
map in $g\in \EC^o(f)$ is the same and so this definition makes sense.

\begin{lemma}\label{lem:thetamap}\label{weixiao2}
The map $\Theta$ defines a homeomorphism between $\EC^o(f)$ and
${\mathcal{M}}^o(f) / \sim$.
In particular, $\EC^o(f)$  is homeomorphic
to an open ball of dimension equal to the number of critical points in $B(f)$.
\end{lemma}
\begin{proofof}{Lemma~\ref{lem:thetamap}}
Since the sets $U_i(g)$ move continuously (in the Carath\'edory topology)
with respect to $g\in \EC^o(f)$ 
the map $\Theta$ is continuous. Here we use the continuous dependence
of the Riemann mapping from $U_i(g)$ to $\D$ as the simply connected domain 
$U_i(g)$ moves continuously with $g$,  see the discussion in Section 5.1 in \cite{McM}. 
By the Rigidity Theorem~\ref{KSS},
$\Theta$ is injective. Indeed, if $\Theta(g)=\Theta(\tilde{g})$ then $g$ and $\tilde g$
 are topologically  conjugate on $\R$, and moreover they are conformally conjugate near
the corresponding periodic attractors. Therefore $g$ and
$\tilde{g}$ are affinely conjugate.

Because a continuous bijective map between open subsets of Euclidean spaces
is a homeomorphism (due to Brouwer's invariance of domain theorem),
it remains to prove that $\Theta$ is surjective.
Let $\bfA=(A_1, A_2, \ldots, A_m)$ be an element in $\mathcal{M}^o(f)$.
Our aim is to construct a map $g\in \EC^o(f)$ so that $\Theta(g)=[\bfA]$.
To do this, one applies quasi-conformal surgery techniques in a standard fashion.
Let us therefore be brief, and refer to the exposition given in Theorem VIII.2.1 of \cite{CG}
for details.
Choose $f_0\in \EC^o(f)$ and let $U_i$, $s_i$, $n_i$, $\eps_i$ be the
objects associated to $f_0$ as above.
Let $\phi_i\colon U_i\to \D$
be a real-symmetric conformal map sending the periodic attractor in $U_i$ to $0$. Then
$\phi_{i'}\circ f_0^{s_i}\circ \phi^{-1}_i \colon \D\to \D$ is a map $A_i^o$ in $M_{n_i,\eps_i}^o$
and $\bfA^o=(A_1^o,A_2^o,\dots,A_m^o)\in \mathcal{M}^o(f)$.   Define a new smooth covering
map $\tilde A_i\colon \D\to \D$ as follows.  Take discs
$\Delta(r_i)\subset \D$ with $r_i<1$ sufficiently close to $1$
so that $\cup_{i} \Delta(r_i)\times \{i\}$ is mapped into itself by
$\bfA$ and also by $\bfA^o$. Let
$\mbox{Ann}_i=((A_i^o)^{-1}\Delta(r_i))\setminus \Delta(r_i)$
so that $\mbox{Ann}_i$ is  a fundamental annulus of $A_i^o$.
Choose $\tilde A_i\colon \D\to \D$ so that it  agrees with  $A_i$ on $\Delta(r_i)$ and with
$A_i^o$ on $\D\setminus (\Delta(r_i)\cup \mbox{Ann}_i)$,
and so that it is a smooth  covering map on the fundamental annulus $\mbox{Ann}_i$.
Next define a smooth map $\tilde f$ which agrees with $f$ outside $\cup U_i$
and which is equal to $f^{-(s-1)}\circ  \phi^{-1}_i \circ \tilde A_i\circ  \phi_i$ on $U_i$,
where $f^{-(s-1)}$ stands for inverse of the conformal map $f^{s-1}\colon f(U_i) \to U_i'$.
The smooth map $\tilde f$  agrees with $f$ outside $\cup U_i$ and is conformal 
outside the annuli $\phi^{-1}_i ( \mbox{Ann}_i)$.
Since the $\tilde f$  orbit of each point only hits at most  once the fundamental annuli
we can choose an invariant ellipse field which agrees with the standard 
linefield in $\phi_i^{-1}(\Delta(r_i))$ and on the complement of $B(f)$.
Using the Measurable Riemann Mapping Theorem, we obtain a $K$-quasiconformal
homeomorphism $h$ so that $g:=\tilde h\circ \tilde f \circ \tilde h^{-1}$
is again holomorphic and therefore the required polynomial of the same degree as $f$.
Since $g$ and $f$ are conjugate outside $B(f)$, we have $g\in \EC^o(f)$.  

Let us now show that $\Theta(g)=[A]$. Let
$U^g_i:=h(U_i)$ be the components of $B(g)$.
Writing
$$H_i:=\phi_i \circ h^{-1}\colon h(U_i)\to  \D\mbox{ and }
H_{i'}:=\phi_{i'} \circ h^{-1}\colon h(U_{i'})\to  \D, $$ 
we have that $H_{i'} \circ  g^{s_i}\circ H_i^{-1}\colon \D\to \D$
agrees with the Blaschke product 
$A_i$ on $H_i(\Delta(r_i))$. Moreover, 
$H_i^{-1}$  is conformal on this set. Since the forwards iterates of critical points $A$ are
contained in this set, by pulling back via the dynamics, 
one obtains a sequence of $K$-quasiconformal maps $H_{i,n}\colon h(U_i)\to \D,
H_{i',n}\colon h(U_{i'})\to \D$  which are conformal on larger and larger
subsets of $h(U_i)$ and $h(U_{i'})$ respectively, and so that
 $H_{i',n} \circ  g^{s_i}\circ H_{i,n}^{-1}\colon \D\to \D$ agrees with 
 $A$ on corresponding increasing subsets of $\D$.
By taking limits, one obtains conformal maps $\hat H_i\colon h(U_i)\to \D,
\hat H_{i'}\colon h(U_{i'})\to \D$  so that 
$\hat H_{i'} \circ  g^{s_i}\circ H_{i}^{-1}=A_i$.
\end{proofof}

\subsection{The proof of Theorems~\ref{thm:defspace1} and \ref{thm:defspace2}}
Take $f\in \PB_\epsilon$. Assertion (i) in Theorem~\ref{thm:defspace1} was proved in the previous lemma. 
So let us prove that $\EC(f)\subset \overline{{\EC}^o(f)}$ and that there exists a family of maps as 
in Theorem~\ref{thm:defspace2}. 

In Steps 1-4 we find a polynomial map $P\in \PB_\shape$ which will essentially 
play the role of $f_1$.  To find such a polynomial, 
we will first approximate $f$ by a suitable continuous map $g$ whose dynamics is
the same as that of $f$ except {\lq}on the basins of periodic attractors{\rq}.

\noindent
{\bf Step 1.} In this step we find a family of piecewise smooth interval maps $g_t$, $t\in [0,1]$
with $g_0=f$  which undergoes  the required bifurcations at $t=0$, as $t$ becomes positive,  
for each of the periodic attractors as in case (1)-(4) of the assumption of 
Theorem~\ref{thm:defspace2}.
Here we ensure that $g_t$ agrees with $f$ outside
a small neighborhood in $B(f)\cap \R$ of the neutral periodic points $p_1,\dots,p_k$. 
Depending on the sign of $\sigma_i$ 
in the assumption of Theorem~\ref{thm:defspace2}, 
we choose $g_t$ so that as $t$ increases, the neutral point $p_i$ undergoes
a period-doubling or period-halving bifurcation in case (1),  a pitchfork
or a reverse pitchfork bifurcation in case (2) and a saddle-node or a reverse saddle-node bifurcation 
in case (4).  If $p_i$ is as in  case (3),  we choose $g_t$ so that it is merely piecewise smooth at $p_i$,
and $g^{n_i}_t(x)=p_i+(1+t)(x-p_i)-(x-p_i)^2+O(x-p_i)^3$ for $x$ in a one-sided attracting neighborhood 
of $p_i$. 

\begin{figure}[htp]
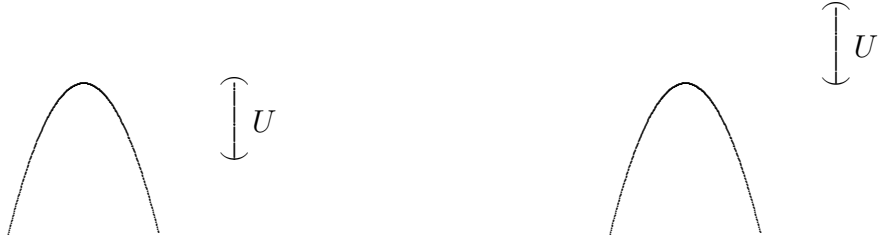
 \hfil
\beginpicture
\dimen0=0.2cm
\setcoordinatesystem units <\dimen0,\dimen0>  point at 0 0
\setplotarea x from 0 to 10, y from 0 to 10
\setlinear
\setquadratic 
\plot 0 0 5 10 10 0 / 
\setlinear \plot 15 10 15 15 / 
\put {$\smile$} at 15 10 
\put {$\frown$} at 15 15 
\put {$U$} at 17 12.5 
\setcoordinatesystem units <\dimen0,\dimen0>  point at 40 0
\setplotarea x from 0 to 10, y from 0 to 10
\setlinear
\setquadratic 
\plot 0 0 5 10 10 0 / 
\setlinear \plot 15 5 15 10 / 
\put {$\smile$} at 15 5 
\put {$\frown$} at 15 10 
\put {$U$} at 17 7.5 
\endpicture
\caption{\label{fig:step23}
{\small A critical point which is mapped to the boundary of a component of the basin of
a periodic attractor. The cases discussed in step 2 and 3 of the proof of Theorem 3.6 are shown on the left
respectively right.}}
\end{figure}

\noindent
{\bf Step 2.} First assume that $c$ is a  critical point as in case (5) so that $c$ is mapped to the 
boundary of a component $U$ of $B(f)$
and $f^{-1}(U)\cup \{c\}$ contains a (real)  neighborhood of $c$, see Figure~\ref{fig:step23} on the right.
Then we can choose the family $g_t$
  so that $g_t(c)$ is in the interior of $B(f)$ for each $t\in (0,1]$. 
 Let us denote $\hat B(f)=\cup \overline{J}$ where the union runs over all connected components $J$ of
 $B(f)$ and $\overline{J}$ is the closure of $J$.  Note that  
 the only difference between $g_t$ and $f$ is that some basins are merged or split in two
 and therefore $\hat B(g_t)= \hat B(f)$ for each $t\in [0,1]$.

\noindent
{\bf Step 3.} 
Pick $t_0>0$ small and let $\hat g=g_{t_0}$.   Next we need to take special care of 
case (5) in the situation when there exists a critical point $c$ so that $f(c)$ is mapped
into the boundary of a component $U$ of $B(f)$ so that  $f^{-1}(U)\cup \{c\}$ contains no (real) 
neighborhood of $c$, as is  drawn in Figure~\ref{fig:step23} on the right.
In this case choose a semi-conjugacy  $h\colon I\to I$ (\ie 
$h$ is continuous, monotone and surjective)   and a continuous $b$-modal interval map $g$ 
so that $h\circ g=\hat g\circ h$ where
$h^{-1}(x)$ is a point, except if $x\in \cup_{n\ge 0} \hat g^{-n}(c)$ for any $c$
as in the previous sentence. Therefore we {\lq}glue-in{\rq} intervals in the backward
orbit of $c$ allowing us to {\lq}move{\rq} $\hat g(c)$ into the interior of  $B(\hat g)$.
We can do this so that $g(c)$ now is mapped into the interior of the component of $B(g)$ corresponding 
to $U$. It is possible that a critical point $c'$ is eventually mapped to $c$.
In this case $\cup_{n\ge 0} \hat g^{-n}(c)$ contains $c'$, 
and we proceed in the same way. Also note that there exists an order preserving
homeomorphism $h_0$ of $I\setminus \hat B(f)$ to $I\setminus \hat B(g)$ 
so that $h_0\circ f=g\circ h_0$ on $I\setminus \hat B(f)$. 

\noindent
{\bf Step 4.} 
Note that $g$ has no wandering intervals and also no inessential periodic attractors.
Hence, by the Fullness Theorem II.4.1 in \cite{MS}, there exists a polynomial $P\in \PB_\shape$ which is topologically conjugate to  
$g$ and so that each of its periodic orbits is hyperbolic and  so that none of its critical point is 
mapped to the boundary of a component of the basin of $P$.
Note that  $f$ and $P$ are conjugate outside
their basins: there exists an order preserving
homeomorphism $h_1$ of $I\setminus \hat B(f)$ to $I\setminus \hat B(P)$ 
so that $h_0\circ f=P\circ h_0$ on $I\setminus \hat B(f)$. 
In other words, $f$ and $P$ are conjugate outside their basin of attractors
and components of $B_\C(f)$ and $B_\C(P)$ correspond to each other in the following manner:
\begin{enumerate}[topsep=0cm,itemsep=0.05ex,leftmargin=1cm]
\item Each component   of $B_\C(f)$ containing a hyperbolic periodic point 
corresponds to a unique component of $B_\C(P)$.
\item Each component  of $B_\C(f)$ of  the basin of
a neutral periodic point $p_i$ of $f$ (\ie of a petal) corresponds to a component of the basin of 
a hyperbolic periodic point for $P$. However,  two touching adjacent basins (petals)
of $B_\C(f)$ may correspond to one component of $B_\C(P)$ (so these petals are merged).
This happens in case (1), (2) and (4) when the corresponding sign $\sigma_i$ is negative.
\item Each critical point of $f$ which is eventually mapped into the boundary of the basin of $f$,
corresponds to a critical point of $P$ which is contained in the interior of $B_\C(P)\cap \R$.
\end{enumerate}
We will use $P$ to construct the required family $f_\mu$,
but we do not claim that the polynomial  $P$ is close to $f$, even in the $C^0$ topology.

Let $U_1,\dots,U_m$ be the components of $B_\C(P)$
which contain critical points. We will use the homeomorphism 
$\Theta\colon \EC^o(P)\to {\mathcal{M}}^o(P) / \sim$
defined in  Subsection~\ref{SecPHo} to construct a family of 
polynomials through $f$. To do this, we construct in Step 5
a family of polynomials $F_\mu$ which may have a much higher degree than $P$.
This family $F_\mu$ will be used to obtain a one-parameter family of Blaschke products, \ie 
a one-parameter family in ${\mathcal{M}}^o(P) / \sim$. 

\noindent
{\bf Step 5.}
Let $U_1^f,\dots,U_k^f$ be the components of $B_\C(f)$ which contain critical points
and let $U_{k+1}^f,\dots,U_{k'}^f$ be the components $U$ of $B_\C(f)$ for which 
there exists a critical point $c$ so that $f(c)\in \partial (U\cap \R)$. 
Note that $k'=m$ by the construction of $P$.
Let $V_1^f,\dots,V_n^f$ be the components of $B_\C(f)$ containing the 
forward iterates of the sets $U_1^f,\dots,U_{m}^f$. Let $X=\cup_i \partial (V_i^f\cap \R)$.
Moreover, for each critical point $c$ of $f$, let $n(c)$ be the 
smallest integer for which there exits a component $U$ of $B_\C(f)$
so that $f^{n(c)}(c)\in  \partial (U\cap \R)$ and if there exists no such integer let $n(c)=0$. 
Let $Y=\cup_c \{f^i(c); 0\le i\le n(c)-2\}$ 
where the union is taken over all critical points $c$ of $f$. 

Next consider a real polynomial map $Q$ which is zero on the set $X\cup Y$, so that $Q'$ is zero on $\Crit(f)$
and let  $F_\mu=f+\mu Q$. 
Note that the degree of $Q$ and therefore of  $F_\mu$ might be much larger than that of $f$.
By the choice of $Q$ each periodic point of $f$ on the boundary of $B_\C(f)$ is still a periodic point for 
$F_\mu$ and each critical point of $f$ is still a critical point of $F_\mu$. 
By  Theorem~VI.1.2 in \cite{CG} one can choose $Q$ so that as $\mu\in [0,1]$ becomes positive, 
$F_\mu$ undergoes all the bifurcations of neutral periodic points
and homoclinic orbits required in  cases (1)-(5) of the theorem.
In particular, each of the components $U_1,\dots,U_m$ of $B_\C(P)$ corresponds in a unique way to a component
$U_{1,\mu},\dots,U_{m,\mu}$ of $B_\C(F_\mu)$ when $\mu\in (0,1]$ is small. 
Note that $F_\mu$ can have many more attractors than $P$ and $F_\mu$ may not necessarily be conjugate to $f$
and we do not claim that  $U_{i,0}=U_i$ either.
Each of the attractors in $U_{1,\mu},\dots,U_{m,\mu}$ is hyperbolic
and $U_{i,\mu}$ depend continuously on $\mu\in (0,\mu_0]$ for $\mu_0>0$ small. 
As below  Lemma~\ref{lem:Mconn}, we have that $F_\mu^{s_i}(U_{i,\mu})=U_{i',\mu}$
and we can associate to each of the maps $F_\mu^{s_i}\colon U_{i,\mu}\to U_{i',\mu}$
a Blaschke product $A_i$ as in \eqref{eq:A}. In this way we obtain 
$$[(A_{1,\mu},\dots,A_{2,\mu})]\in \mathcal{M}^o(P) / \sim$$  as  in Lemma~\ref{lem:Mconn}.
Note that $[(A_{1,\mu},\dots,A_{2,\mu})]$ depends continuously on $\mu\in (0,\mu_0]$, 
because the domains $U_{i,\mu}$  vary continuously.

{\bf Step 6.} 
Next consider the map $$\Theta: \EC^o(P) \to \mathcal{M}^o(P) / \sim $$
from equation \eqref{eq:theta0} and \eqref{eq:theta}
and define 
$$f_\mu = \Theta^{-1} [(A_{1,\mu},\dots,A_{m,\mu}].$$
Therefore $f_\mu$ is obtained from $P$ by gluing  in the 
components $U_i$ Blaschke products which are obtained from the family
$F_\mu$. By definition $f_\mu\in \PB_\shape$ and 
by  construction, when $\mu>0$,  all periodic orbits of  $f_\mu$ are hyperbolic
 and no critical point point of $f_\mu$ is eventually mapped in the boundary 
 of a component of $f_\mu$. Moreover, $f_\mu\in \EC^o(f_{\mu_0})=\EC^o(P)$ for each $\mu\in (0,\mu_0]$.

\medskip
{\bf Step 7.}  Let us now show that 
$f_\mu$ tends to $f$ as $\mu\to 0$ (in the sense that the coefficients
of $f_\mu$ converge to those of $f$). 
To see this, let $V_{1,f_\mu},\dots, V_{n,f_\mu}$ be the components of $B(f_\mu)$
which contain forward iterates of critical points of $f_\mu$ and let $V_{1,\mu},\dots,V_{n,\mu}$
be the corresponding components of $B(F_\mu)$. By construction
there exists a family of conformal homeomorphisms
$$h_\mu\colon 
V_{1,f_\mu}\cup \dots \cup  V_{n,f_\mu} \to V_{1,\mu}\cup \dots \cup V_{n,\mu}$$
so that $h_\mu \circ f_\mu = F_\mu \circ h_\mu$ on this set.
Moreover, $h_\mu$ depends continuously on $\mu\in (0,\mu_0]$.
Moreover, even though some of these components pinch (at places where $f$ has a parabolic 
periodic point), the conformal homeomorphisms $h_\mu$ have a well-defined
conformal {\lq}limit{\rq} 
$$h\colon  V_{1,f_0}\cup \dots \cup  V_{n,f_0} \to V_{1}\cup \dots \cup V_{n}$$
so that $h \circ f_0 = F_0 \circ h$. It follows that $f_0$ and $F_0=f$
are conformally conjugate on the basin of periodic attractors. 
Hence by the  Generalized Rigidity Theorem~\ref{KSSg} we  obtain that 
$f_0=f$ and that $f_\mu\to f$ as $\mu\to 0$.

This completes the proof of  Theorems~\ref{thm:defspace1} and \ref{thm:defspace2}.
\medskip

Note that we do not state that the Julia set of $F_\mu$ is related to that of $f$.
This enables us to avoid using the techniques employed in \cite{Hai1, Hai2}.
Moreover, in general, it is not obvious how to deform a map with
an attracting and a repelling orbit to one with
a parabolic orbit, or vice versa to deform a map
with a parabolic point to a `subhyperbolic'
map  {\em in such a way  that the Julia set remains topologically
the same}, see \cite{Hai1,Hai2}. We are not concerned with this question.

\section{Partial conjugacy classes within the space $\SS^b$ of stunted sawtooth maps}\label{sec:combinatorialmodel}

\subsection{\boldmath Definition of the space of stunted sawtooth maps $\SS^b$.
\unboldmath}\label{sec:sawtooth}
Fix the number of
turning points $b$ and the shape $\shape$ of the polynomials
we will consider. From now
on we will drop the symbol $\shape$.
Following \cite{MTr},  it will be useful to introduce
a space of piecewise linear maps with $b$ (possibly touching) plateaus.
Fix the slope $\lambda = b+2$ and let $e = b \lambda/(\lambda - 1)$.
An elementary calculation shows that there exists 
a unique $b$-modal piecewise linear map $S_0$ (a {\lq}sawtooth map{\rq}) 
with shape $\epsilon$ and $b$ turning points 
 $c_1, \dots, c_b$ at $-b+1,-b+3,\dots,b-3,b-1$ with

%
\medskip
\unitlength=1mm
\begin{picture}(30,35)(-120,-15)
\put(-115,15){$\bullet$  $b+1$ intervals of monotonicity}
\put(-115,5){\quad \quad $I_0=[-e,c_1], I_1=[c_1,c_2],\dots,I_b=[c_{b},e]$;}
\put(-115,-5){$\bullet$  slope $\pm \lambda$ and extremal values 
$\pm \lambda$;}
\put(-115,-15){$\bullet$ and such that $S_0(\{-e,e\})\subset \{-e,e\}$.}
\thicklines
\put(-15,-15){\line(1,0){30}}
\put(-18,-18){$-e$}
\put(-15,-15){\line(0,1){30}}
\put(15,15){\line(-1,0){30}}
\put(15,15){\line(0,-1){30}}
\put(15.5,-18){$e$}
\put(17,3){$S_0$}
\thinlines
\put(-15,-15){\line(1,5){7}}
\put(-8,20){\line(1,-5){8}}
\put(0,-20){\line(1,5){8}}
\put(15,-15){\line(-1,5){7}}
\thinlines
\put(-15,-15){\line(1,1){30}}
\end{picture}

\bigskip

The space of $\SS^b$ of {\em stunted sawtooth} maps consists
of continuous maps $T$ with plateaus $Z_{i,T}$, $i=1,\dots,b$, which are obtained
from $S_0$ and satisfying

\unitlength=1mm
\begin{picture}(60,35)(-120,-15)
\put(-115,15){$\bullet$  $Z_{i,T}$ is a closed symmetric interval around $c_i$;}
\put(-115,5){$\bullet$  $T$ and $S_0$ agree outside $\bigcup_i Z_{i,T}$;}
\put(-115,-5){$\bullet$  $T|Z_{i,T}$ is constant and $T(Z_{i,T}) \in [-e,e]$;}
\put(-115,-15){$\bullet$  $Z_{i,T}$ have pairwise disjoint interiors.}
\thicklines
\put(-15,-15){\line(1,0){30}}
\put(-18,-18){$-e$}
\put(-15,-15){\line(0,1){30}}
\put(15,15){\line(-1,0){30}}
\put(15,15){\line(0,-1){30}}
\put(15.5,-18){$e$}
\thinlines
\put(-15,-15){\line(1,5){7}}
\put(-8,20){\line(1,-5){8}}
\put(0,-20){\line(1,5){8}}
\put(15,-15){\line(-1,5){7}}
\thinlines
\put(-15,-15){\line(1,1){30}}
\put(-8,20){\line(1,-5){8}}
\put(0,-20){\line(1,5){8}}
\put(15,-15){\line(-1,5){7}} 
\thinlines 
\put(-2,6){{\small $\zeta_2$}}
\put(0,10){\line(0,1){4.5}}
\put(0,4){\line(0,-1){13.5}}
\put(-10,1){{\small $\zeta_1$}}
\put(-8,8){\line(0,1){4}}
\put(-8,-3){\line(0,-1){11.5}}
\put(6,-4){{\small $\zeta_3$}}
\put(7,1){\line(0,1){3}}
\put(7,-7){\line(0,-1){7.5}}
\thicklines
\put(-15,-15){\line(1,5){5.5}} \put(-9.5,12.5){\line(1,0){3}}
\put(-6.5,12.5){\line(1,-5){4.5}}
  \put(11,5){\line(-1,0){6}}
\put(2,-10){\line(1,5){3}}
\put(-2,-10){\line(1,0){4}}
    \put(11,5){\line(1,-5){4}}
    \put(17.5,3.5){$T$}
\end{picture}

Maps in $\SS^b$ are allowed to have touching plateaus (\ie plateaus with one point in common). 
We allow  plateaus to touch
because, as we will see later on,  such maps $T$ correspond to polynomial maps
for which several critical points lie in one component of the basin of an attracting
periodic point. 

Note that if $T\in \SS^b$ has touching plateaus, then the union of these touching plateaus
is called a {\em block of plateaus}. If $T$ has touching plateaus, then it is constant on
at least one lap of $S_0$. In this case $T$ is $b$-modal only in a degenerate
sense. 

It is convenient to use the $b$ `signed' extremal values $\zeta \in [-e,e]^b$
to parametrize $\SS^b$:
\[
\zeta_i = \left\{ \begin{array}{ll}
T(Z_{i,T}) & \mbox{ if } S_0 \mbox{ assumes a maximum at }c_i,\\
- T(Z_{i,T}) & \mbox{ if } S_0 \mbox{ assumes a minimum at }c_i.
\end{array} \right.
\]
Sometimes we will denote by $T_\zeta$ the map $T$ with parameters
$\zeta = (\zeta_1, \dots, \zeta_b)$.  Note that decreasing $\zeta_i$ 
results in widening the corresponding plateau of $T_\zeta$
and that $\zeta_i+\zeta_{i+1}$ is equal to the length
of the convex hull of $T(Z_i)$ and $T(Z_{i+1})$. Hence
$$
\zeta_i \ge -\zeta_{i+1} \quad  \text{ for  } i = 1, \dots, b-1,
$$
with equality when the plateaus $Z_i$ and $Z_{i+1}$ touch.
Let us denote by $[Z_i,Z_{i+1}]$ the convex hull of the plateaus $Z_i$ and $Z_{i+1}$.
Thus we can identify $\SS^b$ with
\[
\{ \zeta = (\zeta_1, \dots, \zeta_b) \st \zeta_i \in [-e,e],
\ \zeta_i \ge -\zeta_{i+1} \}.
\]
We define $T<\tilde T$
if for the corresponding parameters $\zeta_i\le \tilde \zeta_i$ for all $i=1,\dots,b$ 
with at least one inequality.

\begin{proposition}\label{ex:nonmon}
The map $\zeta = (\zeta_1, \dots, \zeta_b) \to h_{top}(T_\zeta)$ 
is non-decreasing in each coordinate.
\end{proposition}

\begin{remark}\label{remark:nonmono}
In \cite{BSnon} we prove that the analogous statement is false for $\PB$ for $b \ge 2$.
That is, if the cubic family is parametrized by its critical values $a,b$,
then the map $(a,b)\mapsto h_{top}(f_{a,b})$ is not monotone in each of its parameters separately. 
\end{remark}

A consequence of this proposition is that
$\{T \in \SS^b \st h_{top}(T) = s\}$ is contractible, see Theorem 6.1 in \cite{MTr}.

\begin{proofof}{Proposition~\ref{ex:nonmon}}
Increasing a parameter $\zeta_i$ makes a plateau narrower, and affects none of
the orbits that never enter $Z_i$. Therefore only new orbits are created, and none destroyed. Hence entropy is non-decreasing in each $\zeta_i$.
\end{proofof}

\subsection{The \basin\, and the basin of a map $T$}\label{sec:preplateaus}
We define the basin $B(T)$ of a map $T\in \SS^b$ 
exactly as before, see Definition~\ref{def:basin}. 
Since maps $T$ have plateaus, we also introduce a related notion:
we define the {\em \basin} $\B(T)$ of a map $T$ to
be the set of points $x$ which eventually map into the {\em
interior} of the union of the plateaus of $T$, \ie
\begin{equation}
\B(T)= \bigcup_{k \ge 0} T^{-k}( \interior ( \cup_{i=1}^b Z_{i,T} ) ).
\label{eq:defBT}
\end{equation}
Because we allow the possibility of plateaus touching each other, we
take the interior of the union rather than the union of the
interiors. 
We say that a component $W$ of $\B(T)$ is {\em periodic} of {\em period} $s$
if $T^s(W)\subset \overline{W}$. A periodic point $p$ of $T$ is called
{\em hyperbolic} if its orbit enters the interior of a plateau of $T$.
The following elementary lemma explains how the sets
$\B(T)$ and $B(T)$ are related.

\vskip 0.4cm

\begin{lemma}\label{lem:descriptionWT}
Let $T\in \SS^b$. Then $\B(T)$ is open and dense. Moreover, for $\B'$ and $\B''$ are components of $\B(T)$,
\begin{enumerate}
\item if $T^n(\B')$ intersects a boundary point  $y$ of $\B''$, then 
$T^n(\B')=\{y\}$;
\item
if $T^n(\B') \cap \B'' \neq \emptyset$, then $T^n(\B') \subset \B''$;
\item if $W'$ is periodic of period $s$, then $T^s(\partial W')\subset \partial W'$
and either $T^s(W')\subset W'$ or $T^s|W'$ is constant;
\item 
$W'$ is either eventually mapped into a periodic component $W''$
of $\B(T)$ (with $T^s(W'')\subset W''$ for some $s$) or there exists $n$ so that $T^n(W')$ is equal to a point;
\item\label{item:basin}
 if $T^s(W')\subset W'$, then $W'$ contains precisely one periodic point $p$ 
(so that some forward iterate is contained in the interior of a plateau)
and $T^{ks}(x)\to p$ for every point in $x\in W'$;
\item\label{item:basin2} $B(T)$ is equal to the set of points which are eventually mapped into 
a periodic plateau. Moreover,  for each component $B$ of $B(T)$ 
\begin{enumerate}
\item there exists a sequence of touching components 
$W_i$ of $\B(T)$ so that $\bigcup_{i\in \mathcal{I}} W_i\subset B \subset \bigcup_{i\in \mathcal{I}} 
\overline{W}_i$,
where $\mathcal{I}$ is an at most  countable index set; 
\item if each periodic orbit of $T$ is hyperbolic 
(\ie disjoint from $\partial(\cup_i Z_i)$), then $B$ is equal to a component of $\B(T)$; 
\end{enumerate}
\item if $W'$ and $W''$ touch and one of them intersects a component $B$ of the basin of a periodic attractor,
then both of them are contained in $B$. 
\end{enumerate}
\end{lemma}
\begin{proof} Openness of $\B(T)$ follows from the definition.
Since the complement of $\B(T)$ is forward invariant and
$T$ is expanding on this complement, the set $\B(T)$ is dense.
To prove (1), take $x\in \B'$ so that $y:=T^n(x)\in \partial \B''$. Note that we can
assume that $n$ is {\lq}minimal{\rq}, \ie there exists no $0<k<n$ such that 
$T^k(x)$ is in the boundary of a component of $\B(T)$. 
Since $y\in \partial W''$ no iterate of $y$ is mapped into the interior of a plateau. 
Since $x\in W$ and $y=T^k(x)$ there exists $0\le l<n$ so that $f^l(x)$ is mapped in
the interior of a plateau. It follows that the interior $H$ of the component of $T^{-n}(y)$
containing $x$ is non-empty.
Since $x\in \B(T)$ and so $x$ is eventually mapped into the $\interior(\cup_{i=1}^b Z_{i,T})$, it follows 
that  $H\subset \B(T)$ (here we use the minimality of $n$). Let $x'$ be an endpoint of $H$. Then $T^n$ is not locally 
constant near $x'$ and so $x',T(x'),\dots,T^{n-1}(x')
\notin  \interior(\cup_{i=1}^b Z_{i,T})$. Since $T^n(x')=y\notin \B(T)$ 
and therefore $y,T(y),\dots \notin \interior(\cup_{i=1}^b Z_{i,T})$,
it follows that $T^k(x')\notin \interior(\cup_{i=1}^b Z_{i,T})$ for all $k\ge 0$
and therefore  $x'\notin \B(T)$. Hence $H=W'$ and so $T^n(W')=\{y\}$,
proving Assertion (1).  This implies (2), (3) and (4) because each component of $\B(T)$ is open,
because $\B(T)$ is backward invariant and because $T$ has only finitely many plateaus. 
Note that if $T^s(W') \subset W'$ then by Assertion (1) the one-sided slope of $T^s$ at the
endpoints of $W$ is $>1$.  Therefore
$T^s$ has a fixed point $p\in W'$ at which the map is locally constant.
If $T^{2s}|W'$ has another fixed point, then $T^{2s}|W'$ also has a repelling fixed point
which is impossible since $W' \subset \B(T)$.   Since $T^s$ 
is locally constant at $p$, there exists an interval
neighborhood  $U_0$ of $p$ so that 
$T^s(U_0)\subset U_0$ and so that $T^{n}(x)\to p$ for each $x\in U_0$. 
If we denote by $U_n$ the component of $T^{-sn}(U_0)$ containing $U_0$,
we have $U_{n+1}\supset U_n\supset \dots \supset U_0$. Therefore $T^s$ maps $U=\bigcup U_n$ into itself
and $T^s(\partial U)\subset \partial U$.  Since the only fixed point of $T^{2s}|J$ 
is $p$, it follows that $U=W$. This proves assertion (5).  
To prove (6) note that each attracting periodic orbit of $T$ necessarily
intersects  $\cup_{i=1}^b Z_{i,T}$ because $T$ and the unstunted sawtooth map $S_0$ agree outside this set.
If this periodic orbit intersects $\interior(\cup_{i=1}^b Z_{i,T})$, then 
each component $B$ of the basin of this periodic orbit coincides with a component
$W'$ of $\B(T)$. On the other hand, 
if this periodic orbit does {\em not} intersect the interior of $\cup_{i=1}^b Z_{i,T}$,
then the backward orbit of this periodic orbit is not contained in $\B(T)$
and then components of $B(T)$ are contained in $\bigcup \overline{W}_i$ 
where $W_i$ are adjacent components of $\B(T)$. This situation is clarified in Example~\ref{ex:seesawunimodal} 
below. The final statement holds because if $W'$ and $W''$ have a boundary point in common, 
then  $T^i(W\cup W')$ is a single point for some $i>0$. 
\end{proof}

\subsection{\boldmath The sets $\cell{T}$, $[T]$ and $\cellN{T}$.\unboldmath}
\label{sec:propertiesT}

As in Definition~\ref{def:partiallyconjugate} we say that $T$ and $\tilde T$ are {\em partially conjugate} if there exists an orientation preserving homeomorphism which maps $Z_i$ to $\tilde Z_i$, which maps $B(T)$
to $B(\tilde T)$ and which conjugates $T$ and $\tilde T$ outside these sets.  \
We define $\EC(T)$ to be the set of
$\tilde T\in \SS^b$ which are partially conjugate to $T$. Since such maps have plateaus, we also define
\begin{equation}\label{<T>}
\cell{T}
=\{\tilde T\in \SS^b \st \B(\tilde T)=\B(T)\} \,\, \mbox{ and } \,\,
[T]=\mbox{closure}(\cell{T} ).
\end{equation}
Note that if $\tilde T\in \cell{T}$ then $\cell{\tilde T}=\cell{T}$. 

Of course, $\cell{T}$ and $\EC(T)$ are closely related:
\begin{lemma}\label{lem:cellph} 
If each periodic orbit of $T$ is hyperbolic, then 
$\cell{T}\subset \EC(T)\subset [T]$. 
\end{lemma}

Example~\ref{ex:seesawunimodal} shows that the assumption that
all periodic orbits of $T$ are hyperbolic is required.

\begin{proof}
By Lemma~\ref{lem:descriptionWT}\eqref{item:basin2}(b),  
if each periodic point of $T$ is hyperbolic, then each component of $B(T)$
is a component of $\B(T)$. It follows that if
$\B(\tilde T)=\B(T)$ then the basins of $\tilde T$ and $T$ are the same,
and therefore $\tilde T$ and $T$ are partially conjugate.  
Now assume that $\tilde T\in \EC(T)$. Since $\tilde T$ and $T$ agree outside plateaus, and periodic points are dense outside the basins, 
the partial conjugacy outside $B(\tilde T)$ and $B(T)$ has to be the identity map.
It follows that if $\tilde T$ also has only hyperbolic periodic orbits, 
then $B(\tilde T)$ is also open and $\B(T)=\B(\tilde T)$.
If $\tilde T$ has one or more non-hyperbolic periodic orbits, then by widening the corresponding 
plateaus one obtains a sequence of maps $\hat T_n\in \cell{T}\cap \EC(T)$ with 
$\hat T_n\to \tilde T$. This implies the lemma. 
\end{proof}

Below we shall show that $\cell{T}$ is contained in a hyperplane $V_T$, and that each 
map  $[T]\setminus \cell{T}$ either has touching plateaus or 
an orbit of one of the following special types:

\begin{definition}[Homoclinic orbit]
We say that $T$ has a homoclinic orbit (hc), if some iterate of a plateau is mapped
to the boundary of a component of  $\B(T)$. 
\end{definition} 

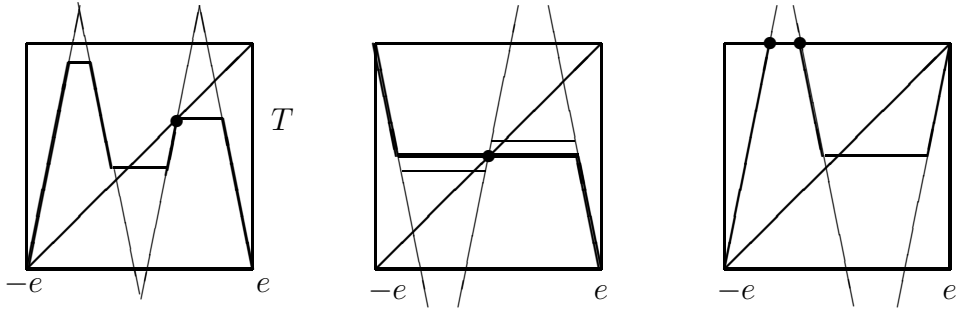
\begin{figure}[h!]
\begin{center}
\unitlength=1mm
\begin{picture}(30,30)(0,-15)
\thicklines
\put(-15,-15){\line(1,0){30}}
\put(-18,-18){$-e$}
\put(-15,-15){\line(0,1){30}}
\put(15,15){\line(-1,0){30}}
\put(15,15){\line(0,-1){30}}
\put(15.5,-18){$e$}
\thicklines
\put(-6.5,12.5){\line(1,-5){2.8}}
\put(-6.6,12.5){\line(1,-5){2.8}}
\put(15,-15){\line(-1,5){4}}
\put(15.1,-15){\line(-1,5){4}}
\put(-15,-15){\line(1,1){30}}
\put(-15,-15){\line(1,5){5.5}} 
\put(-14.8,-15){\line(1,5){5.5}} 
\put(-9.5,12.5){\line(1,0){3}}
\thinlines
\put (-9.5,12.){\line(1,5){1.6}}
\put (-6.6,12.5){\line(-1,5){1.6}}
\thicklines
  \put(11,5){\line(-1,0){6}}
\thinlines 
\put (11,5){\line(-1,5){3}}
\put (5,5){\line(1,5){3}}
\thicklines
  \put(3.5,-2.5){\line(1,5){1.6}}
    \put(3.7,-1.8){\line(1,5){1}}
\put(-3.5,-1.5){\line(1,0){7}}
\thinlines 
\put (-3.3,-1.8){\line(1,-5){3.3}}
\put (3.5,-2.5){\line(-1,-5){3.3}}
\thicklines
    \put(17.5,3.5){$T$}
    \put(3.9,3.7){$\bullet$}
  \end{picture}
  \begin{picture}(30,35)(-15,-15)   
\thicklines
\put(-15,-15){\line(1,0){30}}
\put(-15,-15){\line(1,1){30}}
\put(-15,-15){\line(0,1){30}}
\put(15,15){\line(-1,0){30}}
\put(15,15){\line(0,-1){30}}
\put(-16,-19){$-e$}\put(14,-19){$e$}
\thicklines
\put(-15,15){\line(1,-5){3}}
\put(-15.3,15){\line(1,-5){3}}
\put(15,-15){\line(-1,5){3}}
\put(14.7,-15){\line(-1,5){3}}
\put(-12,0){\line(1,0){24}}
\put(-12,0.2){\line(1,0){24}}
\thinlines
\put(-12,0){\line(1,-5){4}}
\put(0,0){\line(-1,-5){4}}
\put(0,0){\line(1,5){4}}
\put(15,-15){\line(-1,5){7}}
\thinlines 
\put(0.4,2){\line(1,0){11}}
\put(-0.4,-2){\line(-1,0){11}}
\put(-1,-1){$\bullet$}
\end{picture}
\begin{picture}(30,35)(-30,-15) 
\thicklines
\put(-15,-15){\line(1,0){30}}
\put(-15,-15){\line(1,1){30}}
\put(-15,-15){\line(0,1){30}}
\put(15,15){\line(-1,0){30}}
\put(15,15){\line(0,-1){30}}
\put(-16,-19){$-e$}\put(14,-19){$e$}
\thicklines
\put(-15,-15){\line(1,5){6}}
\put(15,15){\line(-1,-5){3}}
\put(-9.8,15){\line(1,0){5.8}}
\put(-5,15){\line(1,-5){3}}
\put(-1.6,0.1){\line(1,0){13.5}}
\thinlines
\put(-15,-15){\line(1,5){7}}
\put(-1.8,0.1){\line(-1,5){4}}
\put(-1.8,0.1){\line(1,-5){4}}
\put(15,15){\line(-1,-5){7}}
\put(-10,14){$\bullet$}
\put(-6,14){$\bullet$}
\end{picture}
\end{center}
\caption{\label{fig:saddle-node} Left: a saddle-node merging two basins. Middle:
a map with 
 two adjacent plateaus with a pitch-fork fixed point; in this case, there are nearby maps 
 with two attracting periodic points of the same period. Right: a map with a homoclinic orbit.}
\end{figure}

\begin{figure}[h!]
\begin{center}
\unitlength=1mm
\begin{picture}(30,30)(0,-15)
\thicklines
\put(-15,-15){\line(1,0){30}}
\put(-18,-18){$-e$}
\put(-15,-15){\line(0,1){30}}
\put(-15,-15){\line(1,1){30}}
\put(15,15){\line(-1,0){30}}
\put(15,15){\line(0,-1){30}}
\put(-15,-15){\line(1,3){8}}
\put(15,-15){\line(-1,3){8}}
\put(7.8,7.7){\line(-1,0){15.5}}
\put(15.5,-18){$e$}
\thicklines
    \put(17.5,3.5){$T$}
    \put(6.5,6.6){$\bullet$}
  \end{picture}
  \begin{picture}(30,30)(-30,-15)
\thicklines
\put(-15,-15){\line(1,0){30}}
\put(-18,-18){$-e$}
\put(-15,-15){\line(0,1){30}}
\put(15,15){\line(-1,0){30}}
\put(15,15){\line(0,-1){30}}
\put(15.5,-18){$e$}
\thicklines
\put(-6.5,12.5){\line(1,-5){3}}
\put(15,-15){\line(-1,5){4}}
\put(-15,-15){\line(1,1){30}}
\put(-15,-15){\line(1,5){5.5}} 
\put(-15,-14){\line(1,5){2.15}}
\thinlines 
\put (-9.5,12.5){\line(1,5){1.6}}
\put (-6.5,12.5){\line(-1,5){1.6}}
\thicklines
\thinlines 
\put (11,5){\line(-1,5){3}}
\put (5,5){\line(1,5){3}}
\thicklines
  \put(3.5,-2.5){\line(1,5){1.6}}
\thinlines 
\put (-3.3,-2.5){\line(1,-5){3.3}}
\put (-3.3,-2.5){\line(1,-5){3.3}}
\put (3.7,-2.5){\line(-1,-5){3.3}}
\thicklines
    \put(17.5,3.5){$T$}
    \put(-4.1,-4.1){$\bullet$}
    \put(-3.3,-3.3){\line(-1,0){9.5}}
        \put(-3.3,-3.2){\line(-1,0){9.5}}
      \put(-3.3,-3.3){\line(1,0){7}}
      \put(-3.3,-3.2){\line(1,0){7}}
      \put(3.7,-3.3){\line(1,5){2.5}}
      \put(15,-15){\line(-1,5){4.8}}
      \put(5.5,9){\line(1,0){5}}
  \end{picture}
\end{center}
\caption{\label{fig:period doubling} Left: a period doubling fixed point; nearby maps
have a periodic point of period two. Right: another period doubling fixed point, this time
between two plateaus.}
\end{figure}
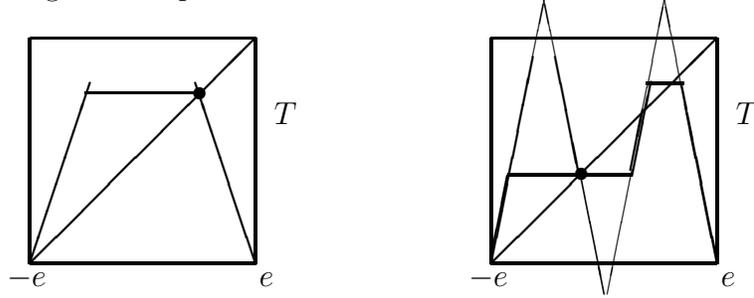

\begin{definition}[Hyperbolic/Saddle-node/Period Doubling/Pitchfork]\label{def:hypT}
Assume that $p$ is an attracting periodic point of $T$ (of minimal period $s$)
and therefore the orbit of $p$ enters a plateau of $T$. 
We say that $p$ is  {\em hyperbolic} if  its orbit enters the {\em interior} of a plateau of $T$.
Otherwise,  we say $p$ is {\em non-hyperbolic} and a 
\begin{enumerate}[topsep=-0.2cm,itemsep=0.05ex,leftmargin=1.2cm]
 \item {\em saddle-node} (sn) if $p$ is in the boundary of a plateau of $T^s$
 and $S_0^s$ is orientation preserving near $p$;
 if there exists a neighborhood $U$ of $p$ so that
 both components of $U\setminus \{p\}$ are in the immediate basin
 of a periodic attractor, then we say that $p$ is a {\em saddle-node merging two basins}, see the left panel of Figure~\ref{fig:saddle-node};
 \item  {\em period-doubling} (pd) if $p$ is in a plateau of $T^s$
 and $S^s_0$ is orientation reversing near $p$, see Figure~\ref{fig:period doubling}; 
 \item {\em pitch-fork} (pf) if $p$ is in the interior of a plateau of $T^s$ and 
 $S^s_0$ is orientation preserving near $p$, see the middle panel of Figure~\ref{fig:saddle-node}.
 \end{enumerate}
\end{definition}
In the latter case, $p$ is either the common boundary of two touching plateaus of $T$, 
or the orbit of $p$ hits at least twice a boundary point of a plateau.

\begin{definition}($\min [T]$ and $\cellN{T}$) \label{def:naturalbullet} 
Given $T \in \SS^b$ we define the sets 
$$
\min [T] = \{ T' \in [T] \st \mbox{there is no }
T'' \in [T] \mbox{ with }  T'' < T'\}.
$$
Here $<$ is the partial ordering
on the space $\SS^b$ defined above Proposition~\ref{ex:nonmon}. 
To deal with the situation that
plateaus of $T$ touch, we also define
$$
\cellN{T}  = 
\left\{  T' \in \min [T] \st  
\begin{array}{l} 
  \mbox{no plateau $Z'$ of $T'$ is eventually mapped  }   \\
  \mbox{into } \cup_{i=1}^b \interior(Z'_i) \mbox{ by some positive iterate.} 
 \end{array}
 \right\} 
$$
\end{definition}

Clearly 
$$\emptyset \ne
\cellN{T} \subset \min [T] \subset [T].
$$
We illustrate these definitions in 
Examples~\ref{ex:seesawunimodal}-\ref{ex:wandering}, 
which helps in obtaining a general description of the set $\cell{T}$. 
We should emphasize that some parts of the  boundary of $\cell{T}$ 
are contained in $\cell{T}$ whereas others are not. 
Figures~\ref{fig:seesawunimodal}, \ref{fig:adjacent_plateaus} 
and \ref{fig:<T>} give explicit descriptions of  $\cell{T}$ in a number of situations.

\begin{example}\label{ex:seesawunimodal} 
Figure~\ref{fig:seesawunimodal}
illustrates the definition of $\cell{T}$ in the {\em unimodal case}.
Let $p\in (-e,e)$ be the orientation reversing fixed point of $S_0$ and $q\in (0,e)$ 
the periodic point of $S_0$ of period two. 
For $\zeta\in [-e,p)$, we have 
$\B(T_\zeta)=(-e,e)$. For $\zeta\in [p,q)$, \, $\B(T_\zeta)$ 
consists of a countable number of adjacent intervals: $\B(T_\zeta)=(-e,e)\setminus \bigcup_{n\ge 0}T^{-n}(p)$.
When $\zeta=q$,\, $\B(T_\zeta)$ consists of a countable number of 
adjacent intervals:  $\B(T_\zeta)=(-e,e)\setminus \bigcup_{n\ge 0}(T^{-n}(p)\cup T^{-n}(q))$. 
In particular, 
$$
\cell{T_{-e}}= [-e,p)\mbox{ , }\EC(T_{-e})=\{-e\}\mbox{ , } \min[T_{-e}]=\cellN{T_{-e}}=\{-e\},
$$
where $[-e,p)$ stands for the set of maps $T_\zeta$ with $\zeta\in [-e,p)$. Moreover, 
$$
\cell{T_p}=[p,q)\mbox{ , }\EC(T_p)=(-e,p]\mbox{ , } \min [T_p]=\cellN{T}=\{p\}.
$$
This example shows that the assumption in Lemma~\ref{lem:cellph} 
that $T$ has only hyperbolic periodic orbits is essential.
\end{example}

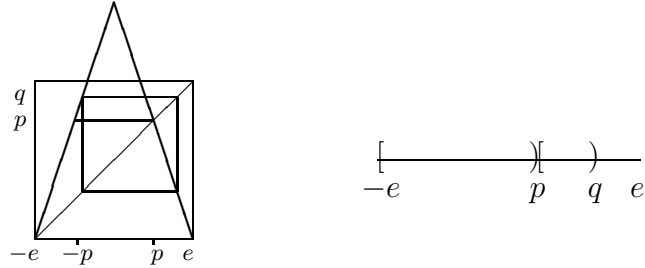
\begin{figure}[h!]
\begin{center}
\unitlength=0.7mm
\begin{picture}(30,40)(20,-16)
\thinlines
\put(-15,-15){\line(1,0){30}}
\put(-15,-15){\line(0,1){30}}
\put(15,15){\line(-1,0){30}}
\put(15,15){\line(0,-1){30}}
\put(-15,-15){\line(1,1){30}}
\put(-20,-19){\scriptsize $-e$} \put(13,-19){\scriptsize $e$}
\put(7.5,-15){\line(0,-1){1}} 
 \put(7,-19){\scriptsize $p$}
\put(7.5,7.5){\line(-1,0){15}}\put(7.5,7.7){\line(-1,0){15}}
\put(-6,12){\line(1,0){18}}
\put(-6,12){\line(1,0){12}}
\put(-19,11.5){\scriptsize $q$}
\put(-19,6.5){\scriptsize $p$}
\put(-6,12){\line(0,-1){18}}
\put(-6,-6){\line(1,0){18}}
\put(12,-6){\line(0,1){18}}
\put(-7,-15){\line(0,-1){1}}\put(-10,-19){\scriptsize $-p$}
\thicklines
\put(-15,-15){\line(1,3){15}}
\put(0,30){\line(1,-3){15}}
\thinlines
\put(50,0){\line(1,0){50}}
\put(49.7,-1){[}
\put(78.7,-1){)}
\put(79,-7){$p$}
\put(80,-1){[}
\put(90,-1){)}
\put(90,-7){$q$}
\put(47,-7){$-e$}
\put(98,-7){$e$}
\end{picture}
\end{center}
\caption{\label{fig:seesawunimodal} The set $\cell{T}$ in the unimodal case, see 
Example~\ref{ex:seesawunimodal}.}
\end{figure}

\begin{example}\label{ex:adjacent_plateaus} In Figure~\ref{fig:adjacent_plateaus}
we illustrate the definition of $\cell{T}$ in the case  when there is a component $\B_i=(a,b)$
of the  $\B(T)$  containing two plateaus with the corresponding set $\cell{T}$
drawn in parameter space on the right. The map $T$  depicted in the figure,
is the first return map to $(a,b)$ and 
has the property that  $\B(T)\cap (a,b)=(a,b)$.
The set $\cell{T}$ for this map is shown in Figure~\ref{fig:adjacent_plateaus}
on the right, and is equal to the union of the two open triangles 
with the open interval $\Delta$ connecting the points marked $0$ and $2$
(corresponding to maps $T_0$ and $T_2$). To see this for $T_0$, note that 
$\B(T_0)\cap (a,b)$ is equal to $(a,b)\setminus Q$ where $Q$ is a countable set
made up of backward iterates of the left boundary point of the left plateau.
Therefore $\cell{T}$ is neither open nor closed. 
Note that $\min [T] = \overline\Delta$ 
and $\cellN{T}=\{T_0, T_1, T_2\}$.  
Furthermore, $\cell{T_2} = {T_2}$ since taking $\zeta_2>0$
(while $\zeta_1$ is left unchanged) results
in the right endpoint of $Z_1$ no longer belonging to $\B(T)$.
The open interval in Figure~\ref{fig:adjacent_plateaus} connecting
$2$ to $3$ is a single $\cell{\tilde T}$, whereas the line segment
connecting $1$ to $3$ consists of countably many different cells 
$\cell{\tilde T}$
consisting of half-open line segments with endpoints corresponding to maps 
for which one plateau is mapped into the boundary of the other plateau.
\end{example}

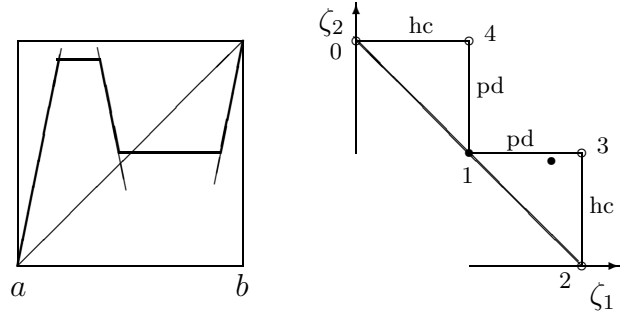
\begin{figure}[h!]
\begin{center}
\unitlength=1mm
\begin{picture}(30,35)(8,-17)
\thinlines
\put(-15,-15){\line(1,0){30}}
\put(-15,-15){\line(1,1){30}}
\put(-15,-15){\line(0,1){30}}
\put(15,15){\line(-1,0){30}}
\put(15,15){\line(0,-1){30}}
\put(-16,-19){$a$}\put(14,-19){$b$}
\thicklines
\put(-15,-15){\line(1,5){5.5}}
\put(15,15){\line(-1,-5){3}}
\put(-9.8,12.5){\line(1,0){5.8}}
\put(-4,12.5){\line(1,-5){2.5}}
\put(-1.6,0.1){\line(1,0){13.5}}
\thinlines
\put(-1.6,0.1){\line(1,-5){1}}
\put(-1.6,0.1){\line(-1,5){2.8}}
\put(15,15){\line(-1,-5){3.8}}
\put(-15,-15){\line(1,5){5.8}}
\thinlines
\put(45,-15){\vector(1,0){20}}\put(61,-20){$\zeta_1$}
\put(30,0){\vector(0,1){20}}\put(25,16.5){$\zeta_2$}
\thinlines
\put(60,-15){\line(-1,1){30}}
\put(59.6,-14.4){\line(-1,1){29.9}}
\put(60,-15){\line(0,1){15}}
\put(60,0){\line(-1,0){15}}
\put(45,0){\line(0,1){15}}
\put(45,15){\line(-1,0){15}}
\put(37,16){\scriptsize hc}\put(61,-8){\scriptsize hc}
\put(46,8){\scriptsize pd}\put(50,1){\scriptsize pd}
\put(56,-1){\circle*{1}}
\put(60,0){\circle{1}} \put(62,0){\scriptsize$3$}
\put(60,-15){\circle{1}} \put(57,-18){\scriptsize$2$}
\put(45,15){\circle{1}} \put(47,15){\scriptsize$4$}
\put(30,15){\circle{1}} 
\put(26.5,12.5){\scriptsize$0$}
\put(45,0){\circle*{1}} \put(44,-4){\scriptsize$1$}
\end{picture}
\end{center}
\caption{\label{fig:adjacent_plateaus}  The set $\cell{T}$ in the case where two plateaus lie in one component of the basin, 
see Example~\ref{ex:adjacent_plateaus}. The {\tiny \textbullet} on the right indicates the
parameter of the map $T$ on the left. 
}
\end{figure}

\begin{example}\label{ex:bimodal}
Figure~\ref{fig:<T>} illustrates the definition of $\cell{T}$ in the 
bimodal case when there exists a periodic component $W$ of $\B(T)$ of 
period $s_1+s_2$ so that $W$ and the component $W'$ of $\B(T)$ 
containing $T^{s_1}(W)$  both contain a plateau. 
In this case $\min[T]=\cellN{T}=\{ T_1 \}$ where $T_1$
is the map corresponding to $1$ in the figure. 
Note that only the left and bottom boundary
is contained in $\cell{T}$ (not including the endpoints of these lines).
\end{example}

\begin{example}\label{ex:wandering}
In this example we show why we consider $\min[T]$; it is possible that
 (with the analogous definition) $\min \cell{T}=\emptyset$.
Consider $T\in \SS^b$ so that there exists an interval $[Z_i,Z_{i+1}]$
which is mapped into another plateau $Z_k$ with $T(Z_i)\subset \partial Z_k$, with $T(Z_{i+1})$ contained in the interior of $Z_k$ and so that no iterate of 
$Z_k$ is contained in a plateau. 
Then $\partial Z_k$ is not contained $\B(T)$. This 
implies that the $i$-th projection of $\B(T)$ is a point while the $i+1$-th 
projection of $\B(T)$ is an open interval. Therefore $\min\cell{T}=\emptyset$. 
(However, $\min[T]\ne \emptyset$; it consists of maps
 for which $Z_i$ and $Z_{i+1}$ touch.
 \end{example}

\begin{figure}
\begin{center}
\unitlength=4.3mm
\begin{picture}(35,13)(-2.5,1.9)
\put(0,2){\circle*{0.2}} \put(-0.7,1.5){\scriptsize$1$}
\put(10,2){\circle{0.2}} \put(9.5,1.3){\scriptsize$2$}
\put(10,5){\circle{0.2}} \put(10.3,4.7){\scriptsize$3$}
\put(9,5){\circle{0.2}} \put(8.3,4.5){\scriptsize$4$}
\put(9,9){\circle{0.2}} \put(9.3,8.8){\scriptsize$5$}
\put(7,9){\circle{0.2}} \put(6.3,8.6){\scriptsize$6$}
\put(0,12){\circle{0.2}} 
\thicklines
\put(0,2){\line(0,1){9.8}}
\put(0.07,2){\line(0,1){9.8}}
\put(0,2){\line(1,0){9.8}}
\put(0,1.93){\line(1,0){9.8}}
\thinlines
\put(0,12){\line(1,0){3}}
\put(3,12){\line(0,-1){1}}
\put(3,11){\line(1,0){4}}
\put(7,11){\line(0,-1){2}}
\put(7,9){\line(1,0){2}}
\put(9,9){\line(0,-1){4}}
\put(9,5){\line(1,0){1}}
\put(10,5){\line(0,-1){3}}
\put(10,2){\vector(1,0){1.5}} \put(10.6,1.1){$\zeta_1$}
\put(0,12){\vector(0,1){2}} \put(-1,13.3){$\zeta_2$}
 \put(0.2, 7){\scriptsize sn} \put(5, 2.2){\scriptsize sn}
 \put(1.5, 12.1){\scriptsize hc} \put(10.1, 3.1){\scriptsize hc}
 \put(3.1, 11.4){\scriptsize pd} \put(9.1, 5.2){\scriptsize pd}
 \put(5, 11.2){\scriptsize pd} \put(9, 7){\scriptsize pd}
 \put(7.1, 10.2){\scriptsize sn} \put(7.6, 9.1){\scriptsize sn}
  \put(6.5,8.2){\scriptsize$pf$}
\thinlines
\put(14.3,1.2){\footnotesize$\B_1$}
\put(18.3,1.2){\footnotesize$\B_2$}
\put(23.3,1.2){\footnotesize$\B_1$}
\put(27.3,1.2){\footnotesize$\B_2$}
\put(12.3,11.5){\scriptsize 1}
\put(13,10){\resizebox{1.5cm}{1.5cm}{\includegraphics{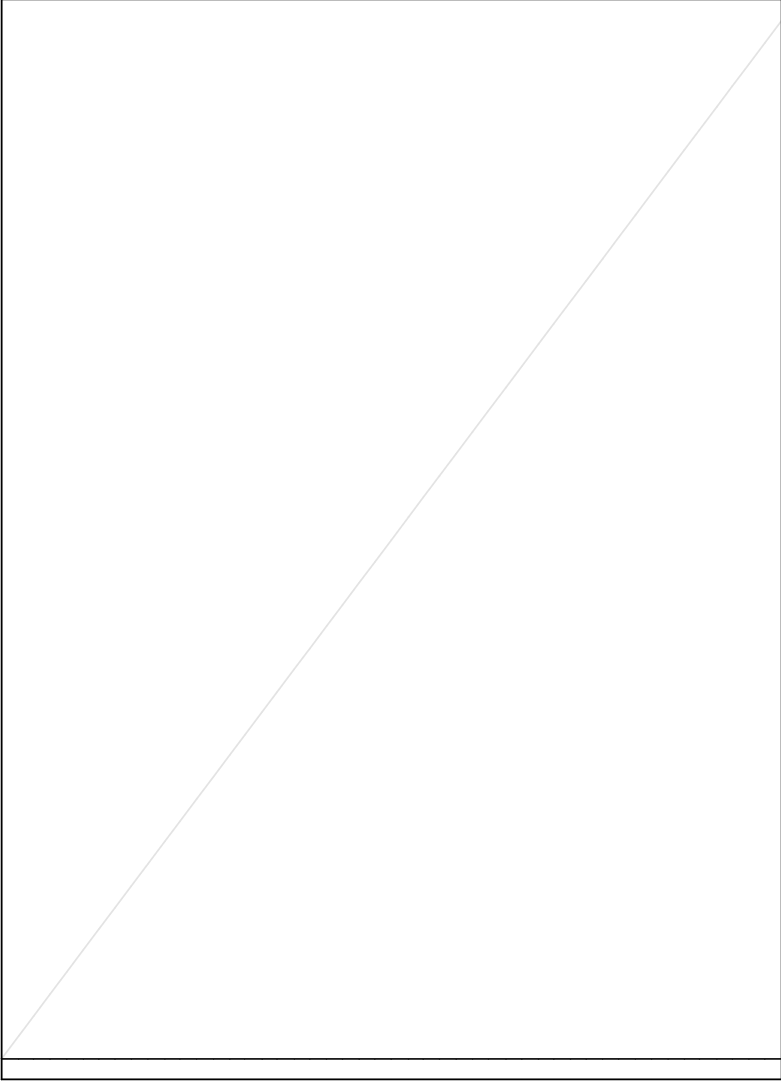}}}
\put(17,10){\resizebox{1.5cm}{1.5cm}{\includegraphics{ssm0_0}}}
\put(12.1,7.5){\scriptsize 2}
\put(13,6){\resizebox{1.5cm}{1.5cm}{\includegraphics{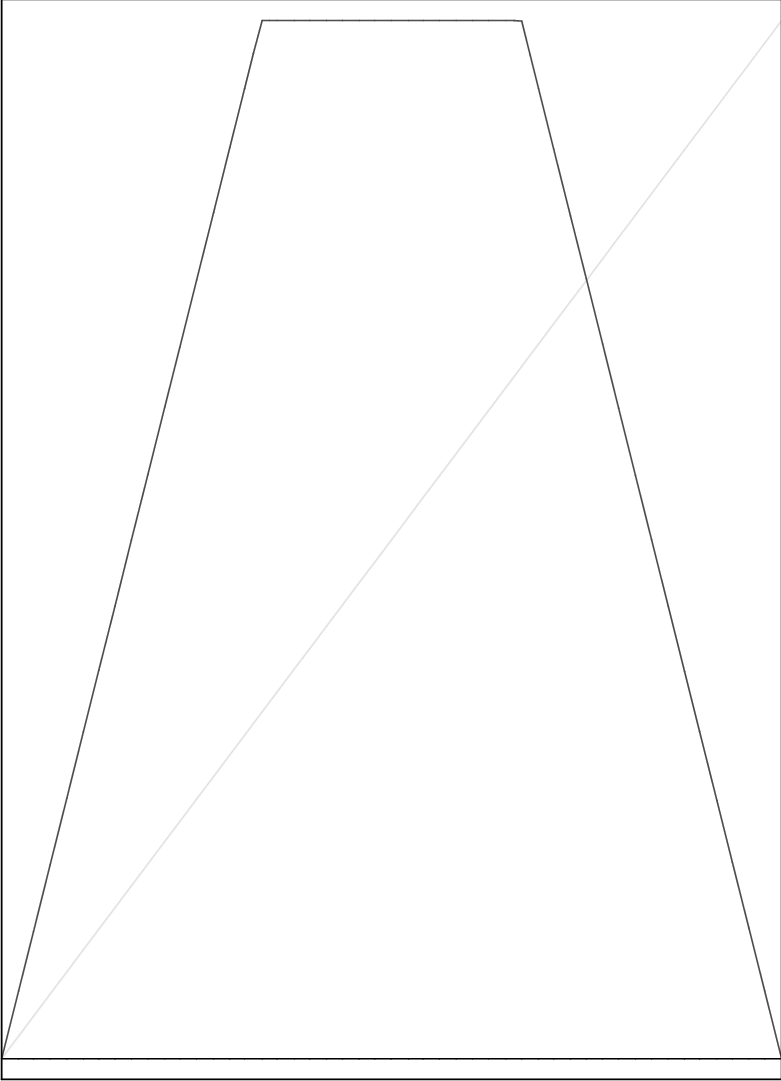}}}
\put(17,6){\resizebox{1.5cm}{1.5cm}{\includegraphics{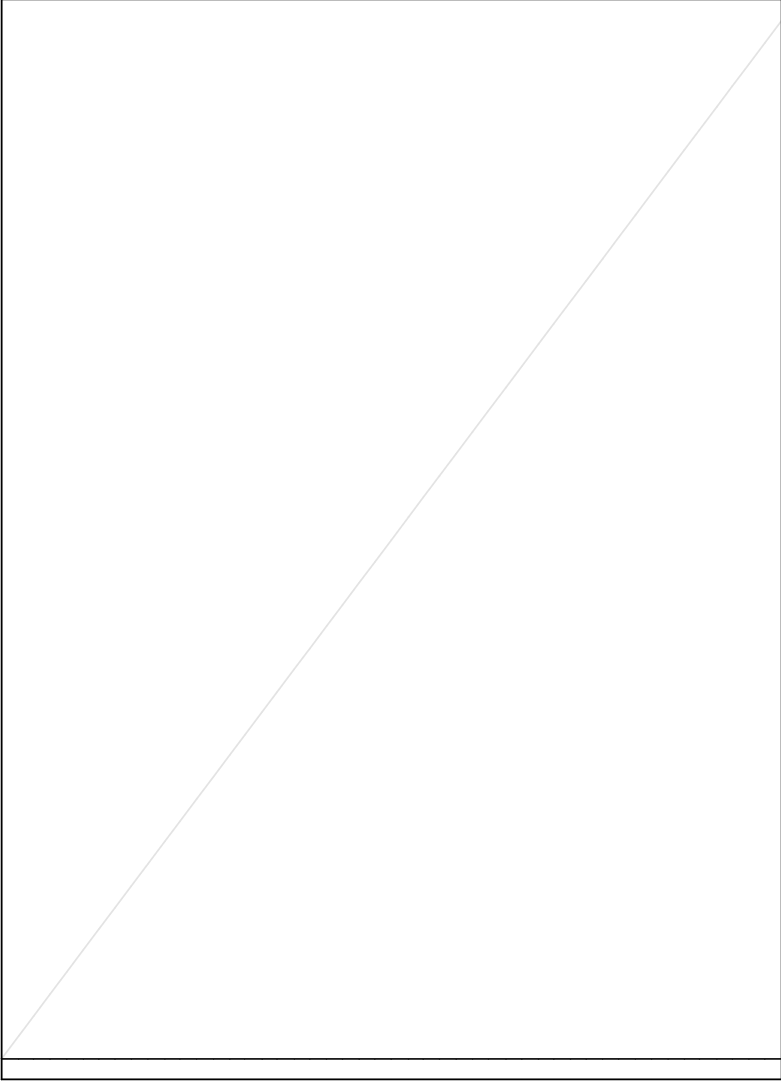}}}
\put(12.3,3.5){\scriptsize 3}
\put(13,2){\resizebox{1.5cm}{1.5cm}{\includegraphics{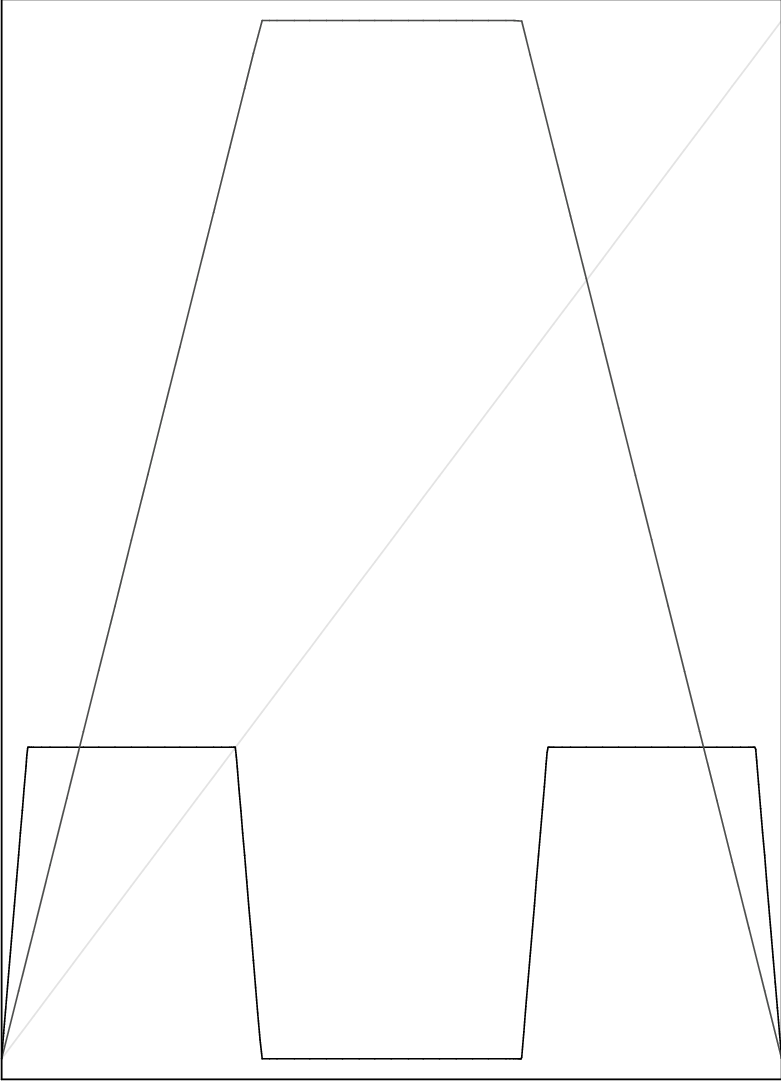}}}
\put(17,2){\resizebox{1.5cm}{1.5cm}{\includegraphics{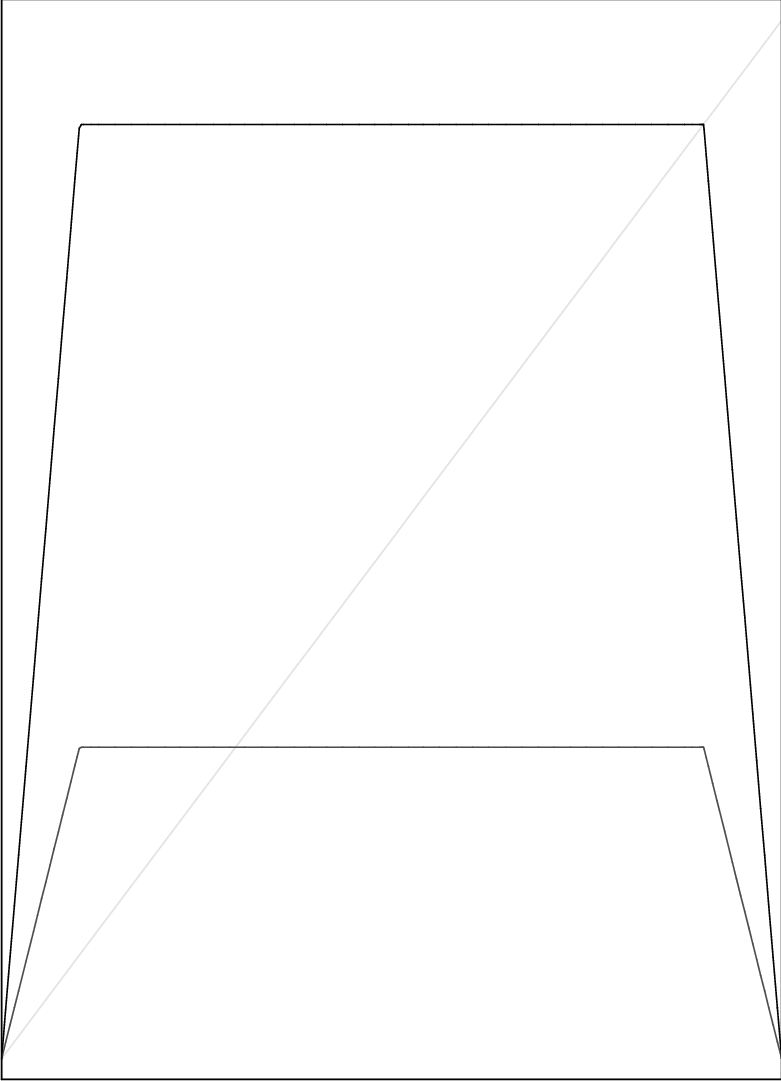}}}
\put(21.3,11.5){\scriptsize 4}
\put(22,10){\resizebox{1.5cm}{1.5cm}{\includegraphics{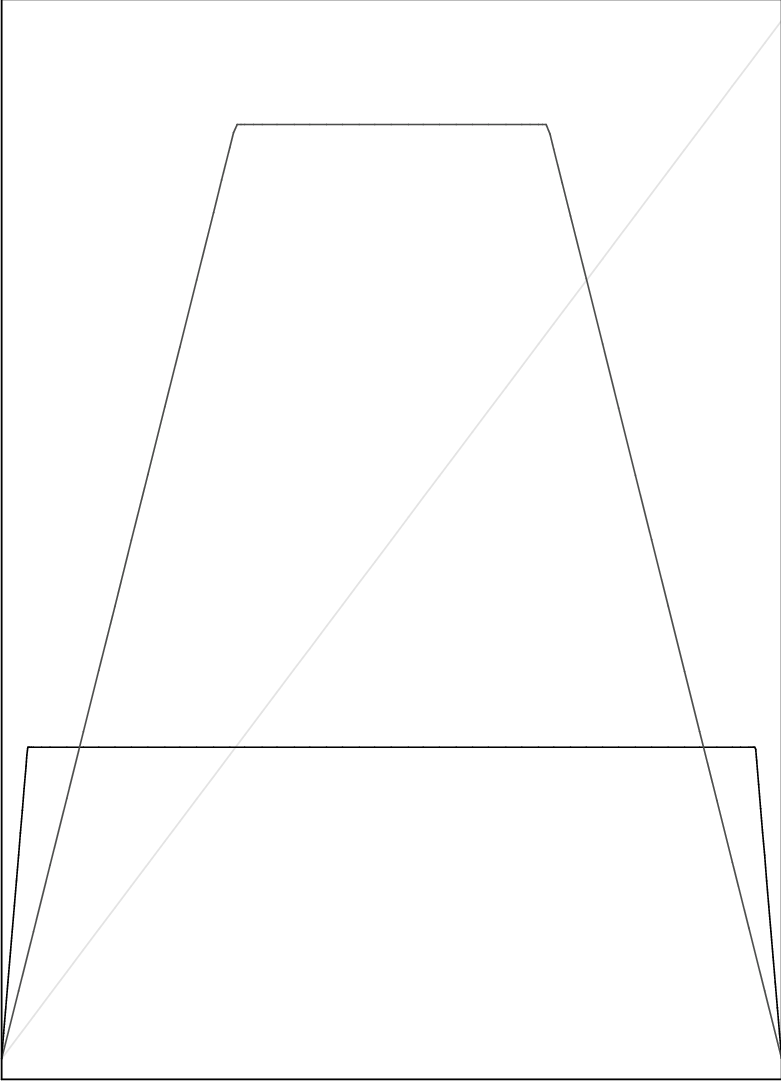}}}
\put(26,10){\resizebox{1.5cm}{1.5cm}{\includegraphics{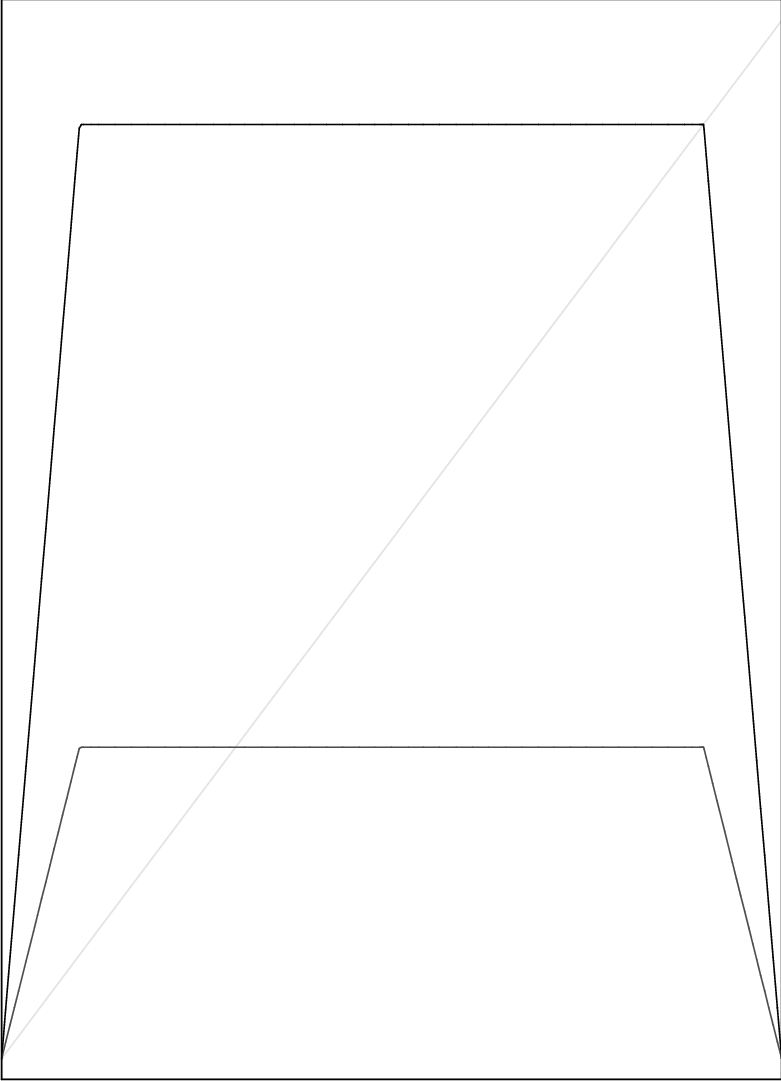}}}
\put(21.3,7.5){\scriptsize 5}
\put(22,6){\resizebox{1.5cm}{1.5cm}{\includegraphics{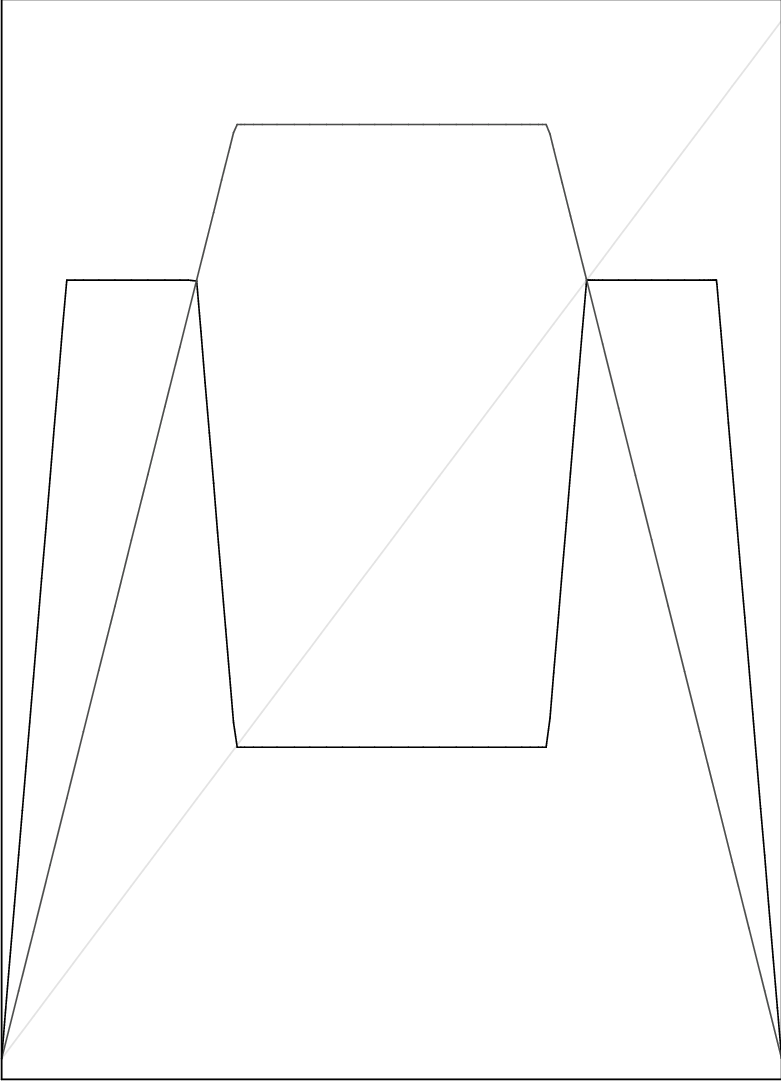}}}
\put(26,6){\resizebox{1.5cm}{1.5cm}{\includegraphics{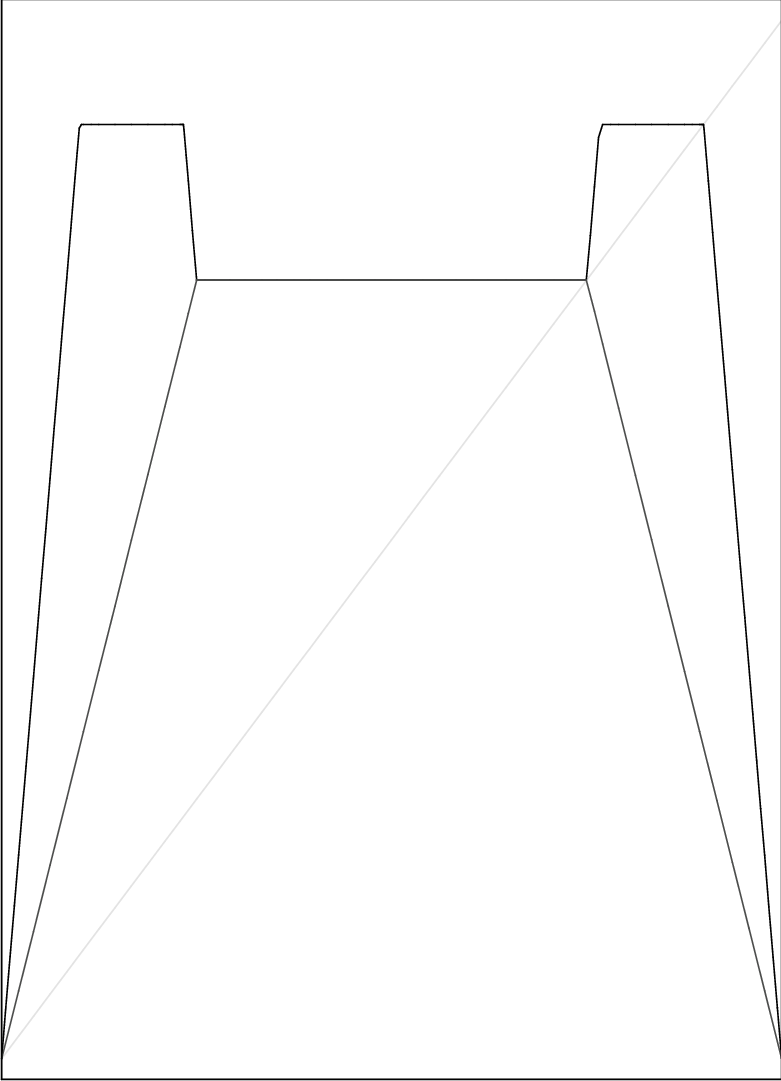}}}
\put(21.3,3.5){\scriptsize 6}
\put(22,2){\resizebox{1.5cm}{1.5cm}{\includegraphics{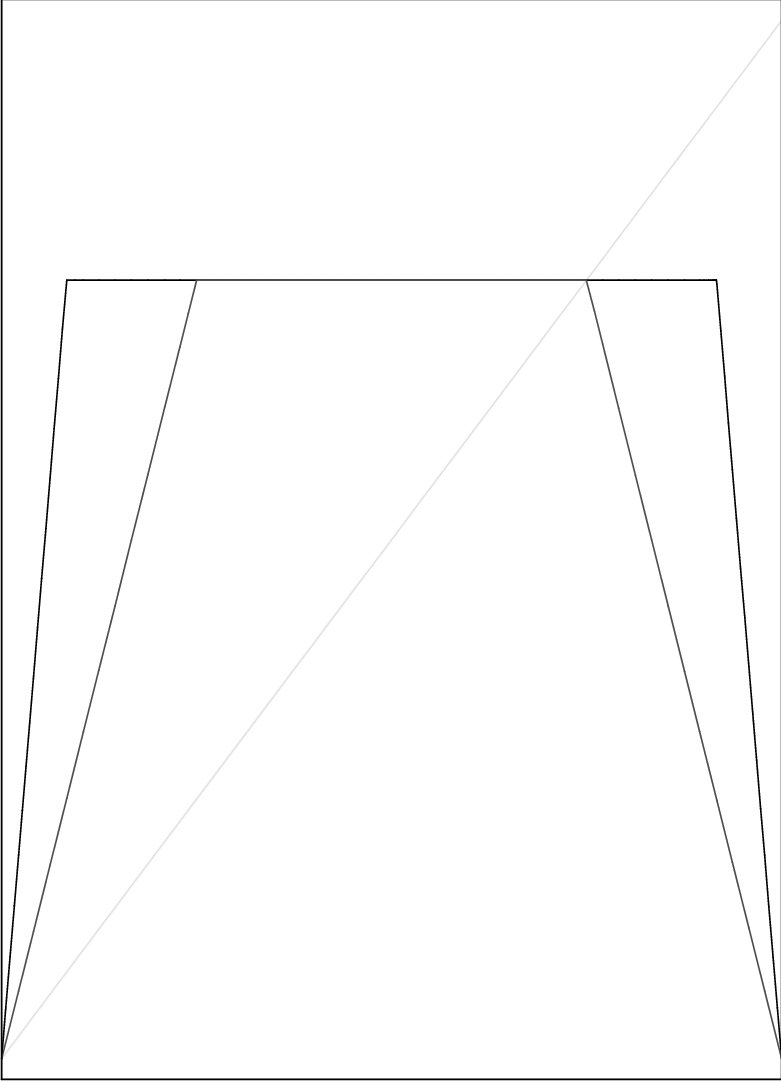}}}
\put(26,2){\resizebox{1.5cm}{1.5cm}{\includegraphics{ssm0p75_0p75}}}
\end{picture}
\end{center}
\caption{\label{fig:<T>}
The set $\cell{T}$ in the bimodal case when an attractor contains two plateaus
in its basin, see Example~\ref{ex:bimodal} is depicted on the left. 
For each of the 6 indicated parameters on the left,  
we draw on  the right side, the corresponding graphs of $T^{s_1}$ and
return map $T^{s_2} \circ T^{s_1}$ to $\B_1$ 
(respectively $T^{s_2}$ and the return map $T^{s_1} \circ T^{s_2}$ to $\B_2$).
Here the graph with the larger slope corresponds to the first return map.
Of course the diagonal has only meaning for the maps
$T^{s_2}\circ T^{s_1}$ and $T^{s_1}\circ T^{s_2}$. The part of the boundary 
of the polygon which is contained in $\cell{T}$ is marked by  (sn). 
}
\end{figure}

\subsection{Wandering pairs and the space $\SS_*^b$}\label{sec:SS*}
The space $\SS^b$ contains maps that are of no use to us because they 
possess wandering intervals, a phenomenon that does not occur in $\PB$.
For this reason, we define a subset $\SS_*^b$ to be used as a more faithful
parameter space of $\PB$ than $\SS^b$.

An important property of $\SS^b_*$, which will be used in the proof of the Main Theorem,
is that isentropes $\{T \in \SS^b_{\shape,*} \st h_{top}(T) = h\}$ are contractible in $ \SS^b_*$.
The proof of this is deferred to Section~\ref{sec:connS*}.
In this section, we will define $\SS^b_*$ and show some of its basic properties.

\begin{definition}\label{def:wander}
A {\em pair} of plateaus $(Z_i,Z_j)$ is called {\em wandering} if  there exists $n \ge 1$
 such that $T^n(\J)$ is a point, where
 $\J:=[Z_i,Z_j]$ is the convex hull of $Z_i$ and $Z_j$.
We say that $T$ is {\em non-degenerate} if for every wandering
pair $(Z_i, Z_j)$, the corresponding interval $\J$ 
belongs to the closure of a component of the basin of 
a periodic plateau. 
Let $\SS_*^b$ denote the set of non-degenerate maps 
$T \in \SS^b$.
\end{definition}

\begin{remark} If $T \in \SS^b_*$ then $[T] \subset \SS^b_*$.\end{remark}

\begin{remark}
Blocks of touching plateaus occur
only at the boundary of the parameter space, and bimodal
maps with touching plateaus have an attracting fixed point.
For this reason, wandering pairs don't occur in Milnor \& Tresser's
paper \cite{MTr}, and $\SS_*^b = \SS^b$ for $b \le 2$.
\end{remark}

\begin{remark}
For $b \ge 3$, $\SS^b_*$ is clearly not closed, but it is not open either.
Indeed, take $T$ with a pair of adjacent plateaus $Z_i,Z_{i+1}$
so that $[Z_i,Z_{i+1}]$ is mapped into the interior of a plateau 
$Z_j$ and $T(U) \cap Z_j \subset \partial Z_j$ for some small open neighborhood $U$ of $Z_j$. Since $Z_j$ is in the basin of an attracting fixed
point, such a map $T$ exists in $\SS^b_*$.  Moreover, there exist maps $\tilde T$ arbitrary close to $T$ so that
$\tilde Z_j$ is no longer contained in the basin of a periodic attractor (just choose $\tilde \zeta_j<\zeta_j$ appropriately). 
It follows that there exists maps $\tilde T\notin \SS^b_*$ arbitrarily close to $T$.
\end{remark}

Nevertheless, the space $\SS^b_*$ has the following useful property:

\begin{lemma}\label{lem:s*open}
(a) Take $T\in \SS^b$ and an interval $U$ 
 which is not eventually mapped into a plateau and is also  
not contained in the basin of a periodic attractor. Then there exists $n>m$ so that
$T^n(U)\cap T^m(U)\ne \emptyset$ and $\liminf_{j}|T^j(U)|>0$.

(b) Take $T\in \SS^b_*$ and consider adjacent plateaus $Z_i,Z_{i+1}$ so that
the convex hull $[Z_i,Z_{i+1}]$ is not contained in the closure of a component of the basin of a periodic attractor.
Then for each map $\tilde T\in \SS^b$ sufficiently close to $T$, the 
adjacent plateaus $\tilde Z_i,\tilde Z_{i+1}$ do not form a wandering pair either.
\end{lemma}

\begin{proof}
Take $T\in \SS^b_*$ and interval $U$ as in statement (a).
Since $T$ is expanding outside its plateaus, 
there exists a sequence $n_i \to \infty$ so that $T^{n_i}(U)$ intersects one of the plateaus
of $T$ (but is not contained in a plateau). It follows  $T^{n_i}(U)$ contains a neighborhood
 of one of the endpoints of a plateau for infinitely many $i$. 
This means that there exist $n>m$ so that
$T^{n}(U)$ and $T^{m}(U)$ intersect, and therefore $T^{n+(j+1)k}(U)\cap  T^{n+jk}(U)\ne \emptyset$  
for $k:=n-m$ and all $j\ge 0$. It follows that $V=\cup_{j\ge 0}T^{n+j k}(U)$ is an interval and $T^k(V)\subset V$.
So either $V$ contains a repelling fixed point of $T^k|V$ or the interval $V$ is the finite union 
of basins of basins of plateaus, separated by one-sided fixed points of $T^k|V$.
Since $U$ is not eventually mapped into the basin of a periodic attractor,
 $V$ contains in its interior a  periodic point $p$ of period $k$ 
 which is not attracting from both sides. Hence $U$ contains in its interior 
 a point $x$ so that $T^l(x)=p$ for some $l$.   The conclusion of statement (a) of the lemma follows. 

Next consider  $\J :=[Z_i,Z_{i+1}]$ as in statement (b). 
By definition of $\SS^b_*$, the interval $U$ is never mapped into another plateau. Hence, by the proof of statement (a), 
some iterate of $\J$ contains a repelling periodic point $p$ of $T$ in its interior, where
$p$ is either (i) repelling from both sides or (ii) $p$ is attracting from one side and separates the immediate basin of 
two adjacent fixed points of $T^k$.
Therefore, if (i) holds and $\tilde T$ is sufficiently close to $T$, then $p$ is still contained in 
the interior of some iterate of $[\tilde Z_i,\tilde Z_{i+1}]$ and $p$ is still repelling for $\tilde T$. 
If (ii) holds then the two basins could merge, but will still contain the iterate of $U$.
Hence in any case $[\tilde Z_i,\tilde Z_{i+1}]$  also does not form a wandering pair for $\tilde T$.
\end{proof}

\subsection{\boldmath Further properties of the sets $\cell{T}$, $[T]$ and $\cellN{T}$.\unboldmath}
\label{sec:propertiesTfurther}

Given $T\in \SS^b$, we define the following equivalence class $\sim_T$ on $\{1,2,\dots,b\}$:
 $i\sim_T j$ if and only  if $Z_i$ and $Z_j$ are both in the immediate basin of the same periodic attractor.
 (We do not require that $Z_i$ and $Z_j$ are in the same component of the immediate basin.)
 Let $J_1,J_2,\dots,J_s\subset \{1,\dots,b\}$ be the corresponding equivalence classes.
For $J=\{i_1,\dots,i_j\}\subset \{1,\dots,b\}$  define
$\pi_{J}\colon \SS^b \to \R^{\# J}$ be the projection
of $\zeta=(\zeta_1,\dots,\zeta_b)\in \SS^b$ to  $(\zeta_{i_1},\dots,\zeta_{i_j})\in \R^{\# J}$. 

\begin{theorem}\label{thm:<T>}
There exists an affine space $V_T$ 
so that, with respect to the coordinates $(\zeta_1,\dots,\zeta_b)\in \R^b$, 
$\cell{T}$ is a connected polygonal region in $V_T$.
Moreover, the following properties hold:
\begin{enumerate}
\item If $[T_1]=[T_2]$ then $\cell{T_1}=\cell{T_2}$.
\item\label{order} If $T_1<T_2<T_3$ and $T_1,T_3\in \cell{T}$ then 
$T_2\in \cell{T}$;
\item \label{prod}
 $\cell{T}$ has a product structure: if we take
$J_1,\dots,J_s\subset \{1,\dots,b\}$ as above the statement of the theorem, then there exist connected 
(polygonal) sets $A_i\subset \R^{\# J_i}$, $i=1,\dots,s$ so that 
$$
\cell{T}=\{\zeta; \, \pi_{J_i}(\zeta)\in A_i \mbox{ for each }i=1,\dots,s\}.
$$
If $\# J_i =1$, then the corresponding set $A_i$ is an interval $[a_i,b_i)$ or a point $\{a_i\}$
whereas if $\# J_1>1$, then $A_i$ has a polygonal shape (see Figures~\ref{fig:adjacent_plateaus} and \ref{fig:<T>} for representative examples). 
\item\label{inf}\label{exist}
For each $T\in \SS^b_*$, 
there exists $T' \in  \SS^b_*$ so that $T'\in \cellN{T'}$ and 
$T\in [T']$.
\item\label{phintersect} If $[T_1] \ne [T_2]$ and $[T_1] \cap [T_2]\ne \emptyset$,
then there exists $T' \in \cellN{T'}$ with $T'\in \min ([T_1]\cap [T_2])$.
\item \label{bifurcation}
If $\tilde T\in [T]\setminus \cell{T}$, then one of the following properties holds:
\begin{itemize}
\item $\tilde T$ has a saddle-node merging two basins;
\item $\tilde T$ has a period doubling orbit;
\item $\tilde T$ has a pitch-fork orbit;
\item $\tilde T$ has a homoclinic orbit. 
\end{itemize}
\end{enumerate}
\end{theorem}
\begin{proof} 
Let $T=T_\zeta$ be the map corresponding to $\zeta=(\zeta_1,\dots,\zeta_b)$,
let $\tilde \zeta=(\tilde \zeta_1,\dots,\tilde \zeta_b)$, 
and assume that $\tilde T=T_{\tilde \zeta}$ is so that $\cell{\tilde T}=\cell{T}$, \ie $\B(\tilde T)=\B(T)$.
Consider a plateau $Z_k$ of $T$ and in order to be definite assume that 
$S_0$ has a maximum in $Z_k$.  
Let $W_k$ be the component of $\B(T)$ intersecting $Z_k$.
We denote the $k$-th plateau of $\tilde T$ by $\tilde Z_k$. 

{\bf Step 1:} First consider the case that $T(Z_k)$  is {\em not} contained in the closure of a component of
$\B(T)$. Let us show that in this case $\tilde \zeta_k=\zeta_k$.
Indeed,  $W_k$ is equal to the interior of $Z_k$
(or equal to the interior of the union of all plateaus which touch $Z_k$). 
If $\tilde \zeta_k<\zeta_k$ then $\tilde Z_k$ strictly 
contains $W_k$ and in particular the corresponding component of $\B(\tilde T)$ 
strictly contains $W_k$, contradicting $\cell{\tilde T}=\cell{T}$.  Hence $\tilde \zeta_k\ge \zeta_k$. 
By assumption, there is a sequence of points converging from the right  to $T(Z_k)$
which are not in $\B(T)$. Hence if $\tilde \zeta_k> \zeta_k$ then the component of $\B(T_{\tilde \zeta})$
intersecting $\tilde Z_k$ is equal to $\tilde Z_k$ and so is strictly
inside $Z_k$, contradicting 
$\B(\tilde T)=\B(T)$. It follows that in this case $\tilde \zeta_k=\zeta_k$.
Note that the freedom of choice of $\zeta_k$ does {\bf not} 
depend on any of the other coordinates.

{\bf Step 2.} Next consider the case $W_k$ is not periodic
and $T(Z_k)$ is contained in the closure of a component $W=(a,b)$ of $\B(T)$. 

{\bf Step 2a.} Let us first assume that $T(Z_k)\subset [a,b)$. 
If $T(Z_k)=\{a\}$ then $W_k$ is equal to the interior of $Z_k$ and otherwise
$W_k$ is equal to the component of $T^{-1}(a,b)$ containing $Z_k$.
In either case, $W_k$ is equal to the component of $S_0^{-1}(a)$ containing
$Z_k$. If $\tilde T(\tilde Z_k)\ge b$ then there exists
$x\in Z_k\setminus \tilde Z_k$ so that $\tilde T(x)=S_0(x)=b\notin W$ and hence $x\notin \B(T)$,  
contradicting that $\B(\tilde T)=\B(T)$. 
If $\tilde T(\tilde Z_k)<a$, then the interior of $\tilde Z_k$ strictly contains $W_k$ which is not possible either. 
On the other hand, since $W_k$ is not periodic, changing $\tilde \zeta_k\in [a,b)$ 
does not change the component of $\B(\tilde T)$ containing $\tilde T(\tilde Z_k)$. 
It follows that in this case a necessary and sufficient condition on  $\tilde \zeta_k$ for 
$\B(\tilde T)=\B(T)$ to hold is that  $\tilde \zeta_k\in [a,b)$ 
and $\tilde T\in \SS_*^b$.

{\bf Step 2b.} If $W_k$ is not periodic
and $T(Z_k)=\{b\}$ where $W=(a,b)$ is a component of 
$\B(T)$ then $W_k$ is the interior of $Z_k$ and as in Step 1 we get  $\tilde \zeta_k=\zeta_k$. 
Again, the freedom of choice of $\zeta_k$ does {\bf not} depend on any of the other coordinates.

{\bf Step 3.} Now consider the case that $W_k$ is periodic, \ie  $T^s(W_k)\subset \overline{W_k}$ for some $s>0$.
By Lemma~\ref{lem:descriptionWT}
either $T^s(W_k)\subset \partial W_k$ or $T^s\colon W_k\to W_k$. In the latter case $T^s\colon W_k\to W_k$
has a unique fixed point $p\in W_k$ (and $T^j(p)\in  \interior ( \cup_{i=1}^b Z_{k,T})$ for some $0\le j<s$)
and $T^s(\partial W_k)\subset \partial W_k$.  Because of  
Lemma~\ref{lem:descriptionWT}\eqref{item:basin}-\eqref{item:basin2}, the component of $\B(T_{\tilde \zeta})$ 
intersecting $\tilde Z_k$ is equal to $W_k$ if and only if 

(i) $\tilde T^s(W_k)\subset \overline{W_k}$;

(ii) $\tilde T^s(W_k)\cap \partial W_k\ne \emptyset$ implies that $T^s|W_k$ is constant, and 

(iii) $\tilde T^{2s}|W_k$ has at most one fixed point (which is attracting).
 
Let $(a,b)$ be the component of $\B(T)$ which contains $T(W_k)$ in its closure.
Properties (i), (ii) and (iii) persist while {\em decreasing} $\tilde \zeta_k\in [a,b)$, by 
Lemma~\ref{lem:descriptionWT}\eqref{item:basin}-\eqref{item:basin2}. 
While {\em increasing} $\tilde \zeta_k$ in $[a,b)$ these properties are preserved
until  $\partial \tilde Z_k$ first hits a fixed point $q$ of $S^{2s}_0$.
If $S^j_0(q)\in \interior ( \cup_{i=1}^b Z_{i,T})$  for at least one $0\le j<s$, 
then one can continue to increase $\tilde \zeta_k$ until  one  $\partial \tilde Z_k$ hits another fixed point $q$ of $S^{2s}_0$.
In this way, we can keep increasing $\zeta_k$ (\ie shrink the width of the plateau), until either $\partial Z_i$
contains a fixed point $q'$ of $S^{2s}_0$ for which $S^j_0(q')\notin \interior ( \cup_{i=1}^b Z_{i,T})$ for all $0\le j<s$
or until we no longer have  $\tilde T^s(Z_k)\subset W_k$.
 
If follows that, fixing all $\zeta_j$ with $j\ne k$, the
set $\cell{T}$ is equal to a line segment of the form  $a_k\le \zeta_k < b_k$.
Here the left boundary $a_k$ does {\bf not} depend on the choice for the other $\zeta_j$'s,
because one can decrease $\tilde \zeta_k\in [a_k,b_k)$ and the latter interval
does not depend on $\zeta_j$, $j\ne k$. 
However,  the right hand boundary $b_k$ in general {\bf will} depend on parameters
$\zeta_j$, $j\ne k$ for which $Z_j$ is contained in (the closure of) one of the components of $\B(T)$
containing $W_k,T(W_k),\dots,T^{s-1}(W_k)$. The location is determined by the position of the fixed points of 
$S^{2s}_0$. 

Thus we have proved that $\cell{T}$ is contained in  a polygonal 
region in some hyperplane $V_T$ and that properties (1), (2) and (3) in the
theorem are satisfied. 

{\bf Step 4.} We claim that we can choose  $T'\in \cellN{T}$ 
(by only changing $T$ inside the basin of periodic attractors) so that  $T\in [T']$
and $T'\in \cellN{T'}$. 
Indeed,  if the number of plateaus within a component of the basin  is even, then 
we choose $T'$ analogously to the map $T_1$ in Example~\ref{ex:adjacent_plateaus} 
and when the number of plateaus is odd then all these plateaus of $T'$ necessarily touch and 
the attracting periodic point is at the boundary point of the union of the touching plateaus.
This implies that, for the coordinates corresponding to attracting plateaus, 
$T\in [T']$ and $T'\in \cellN{T'}$.
If $Z_i$ is contained in the basin, but not in the immediate basin of a
periodic attractor, then for any map $\hat T\in \cellN{T}$ the corresponding plateau
$\hat Z_i$ is either mapped into the boundary of this component, or $Z_i$ touches
with one of its neighbouring plateaus. From this 
description it follows that $T\in [T']$ and $T'\in \cellN{T'}$. Since $T$ and $T'$
only differ on the basin of periodic attractors, we will still have $T'\in \SS^b_*$. 

{\bf Step 5.} To prove Property \eqref{phintersect}, 
assume that $[T_1] \ne [T_2]$ and $[T_1] \cap [T_2]\ne \emptyset$.
Property \eqref{order} gives that  $\min([T_1] \cap [T_2])\subset [T_1] \cap [T_2]$.
It is possible that $T\in \min([T_1] \cap [T_2])$ has  
two or more touching plateaus $Z_i,Z_{i+1}$ which are  mapped by some iterate $T^s$
into the interior of another plateau 
$Z_j$.  In this case choose a continuous deformation  $T_t$ with $T_0=T$ so that for each 
such pair of touching plateaus,  $T^s_t([Z_i,Z_{i+1}])\subset Z_j$ 
for all $t\in [0,1]$ (leaving all other plateaus unchanged). 
Then $T':=T_1\in \min([T_1] \cap [T_2])$ has the required properties.

{\bf Step 6.} Property \eqref{bifurcation} holds because $T'\in [T]\setminus \cell{T}$ 
implies that at least one periodic orbit is in the boundary of a plateau. By Property 
\eqref{order} this periodic orbit cannot be a saddle-node, unless it corresponds
to a situation where two basins are merged.
\end{proof}

\section{\boldmath The map $\ST\colon\PB \to \SS^b$. \unboldmath}
\label{sec:Psi}

Let us review some basic kneading theory, see  \cite{MT}, and
also \cite{MS, MTr}.
Given an interval $I$ and piecewise monotone
$b$-modal map $f\colon I\to I$ with turning points
$c_1<\dots<c_b$ in the interior of $I$, 
one can associate
to each point $x\in I$ an {\em itinerary} $\ii_f(x)$ consisting
of a sequence $(i_0,i_1,\dots)$
of symbols  from the alphabet $\{I_0,c_1,I_1,c_2,\dots,c_b,I_{b}\}$.
Here $I_0,\dots,I_b$ are the components of $I\setminus \{c_1,\dots,c_b\}$
ordered from left to right. 
It is well-known that $x\mapsto \ii_f(x)$ is monotone
w.r.t.\ the signed lexicographic ordering and that therefore
the $i$-th {\em kneading sequence}
$$
\nu_i:=\lim_{x\downarrow c_i}\ii_f(x)
$$
is well-defined. Note that the sequence $\nu_i$ does not contain any
of the symbols $c_1,\dots,c_b$. Let $\sigma$ be the left shift
on the space of symbol sequences.
The kneading invariant $\nu(f)$ of $f$ is
defined as
$$
\nu(f):=(\nu_1,\dots,\nu_{b}).
$$
Any kneading invariant which is realized by some piecewise monotone
$b$-modal map is called {\em admissible}.

To each  map $f\in \PB$ one can associate {\em uniquely}
a stunted sawtooth map as follows.
Let $\nu(f)=(\nu_1,\dots,\nu_b)$  be the
kneading invariant of $f$,
and let $s_i$ be the unique point in the $(i+1)$-th lap $I_i$
of $S_0$  such that
\begin{equation}\label{eq:kneadingdef}
\lim_{y\downarrow s_i} \ii_{S_0}(y)=\nu_i := \lim_{x\downarrow c_i} \ii_f(x).
\end{equation}
Let $Z_i$ be the symmetric interval around the $i$-th
turning point of $S$ with right endpoint $s_i$.
Let us define 
\begin{equation}
\ST\colon\PB \to \SS^b, \quad f \mapsto \ST(f), \label{eq:defST}
\end{equation}
by associating to $f$ the unique stunted sawtooth map $\ST(f)$
which agrees with $S_0$ outside $\cup_{i=1}^b Z_i$ and
which is constant on $Z_i$ with value $S_0(s_i)$.

\subsection{Some good properties of $\ST$}

The main reason for working with the map $\ST$ is that it allows
us to work with the Euclidean space $\SS^b$ rather than with the space of kneading invariants.

\begin{lemma} \label{lem:phpsi}
The map $\ST\colon \PB \to \SS^b$ 
\begin{enumerate}[topsep=-0.2cm,itemsep=0.05ex,leftmargin=1.2cm]
\item is well-defined;
\item the kneading invariant of $f$ and $T:=\ST(f)$ are the same in the sense that
$\lim_{y\downarrow Z_i}i_T(y)=\nu_i$.
\item $f$ and $\ST(f)$ have the same topological entropy;
\item $\ST(\PB_\shape) \subset \SS^b_{\shape,*}$. 
\end{enumerate}
\end{lemma} 

\begin{proof}
Since $S_0$ allows every sequence in $\{I_0,\dots,I_b\}^{\N}$ as itinerary, 
we can always find a stunted version $T$ with the required kneading sequences.
In fact, because $S_0$ is expanding and so distinct points have different
itineraries, the stunted version is unique, so $\ST$ is well-defined.
It also follows that the orbits of the boundary points of $Z_i$ under 
$T$ and the sawtooth map $S_0$ agree and therefore statement (2) holds.
Entropy is fully determined by kneading sequences, so $\ST$ preserves entropy.
For the last statement, assume by contradiction that $\ST(f)\in \SS^b\setminus \SS^b_*$.
Since $f$ has no wandering intervals, $\ST(f)\in \SS^b\setminus \SS^b_*$
implies that there exists an interval  connecting two adjacent critical points, so that 
the $n$-th iterate of this interval is another critical point. Clearly this is impossible. 
\end{proof}

\subsection{Some bad properties of $\ST$}
The next example shows that $f\mapsto \ST(f)$ is neither continuous, nor injective
nor surjective, but later on we shall see that this map 
is {\lq}almost{\rq} continuous, injective and surjective.  
It also shows that $\ST(f)=\ST(\tilde f)$ does {\em not} imply that $f,\tilde f$ are partially conjugate.

\begin{example}\label{ex:STdisc} Consider the family $f_\lambda(x)=\lambda
x(1-x)$ and let $T_\zeta$ be as in Example~\ref{ex:seesawunimodal}.  
Then there are parameters $0<\lambda_1'=2<\lambda_1=3<\lambda_2'<\lambda_2<\lambda_3'$
so that $\lambda_1,\lambda_2$ are the first  two  
period doubling parameters, and $\lambda_1',\lambda_2',\lambda_3'$ are the first three parameters
at which the critical point of $f_\lambda$ is periodic.  Then 
$$\ST(f_\lambda)=\left\{\begin{array}{ll}
T_{-e}&\mbox{\,\,for }\lambda\in [0,\lambda_1'],\\[1mm]
T_{p}& \mbox{\,\,for }\lambda\in (\lambda_1',\lambda_2'],\\[1mm]
T_{q}& \mbox{\,\,for }\lambda\in (\lambda_2',\lambda_3'],
\end{array}\right.  $$ 
see Figure~\ref{fig:uni}. As in Example~\ref{ex:seesawunimodal}, $T_p$ and $T_q$ are the maps for which
 the right endpoint of the plateau is a fixed point  and has period two respectively.
 The discontinuities of $f\mapsto \ST(f)$
occur when the critical point becomes periodic, rather than when a period doubling
bifurcation occurs. In particular, for $\delta>0$ small, 
$\ST(f_{\lambda_1+\delta})=\ST(f_{\lambda_1-\delta})=T_p$ 
but $f_{\lambda_1+\delta}$ and $f_{\lambda_1-\delta}$ are not partially conjugate.
On the other hand,   for each $\lambda \in [0,\lambda_1]$,\,\, $f_\lambda\in \EC(f_0)$
and  $\ST(f_{\lambda})\in \{T_{-e},T_p\}\subset  [T_{-e}]=[\ST(f_{\lambda_1'})]$; note that 
$T_{-e}$ and $T_p$ are partially conjugate.
\end{example}

\begin{figure}[h]
\begin{center}
\unitlength=7.8mm
\begin{picture}(30,5.5)(-1,0.7)
\thinlines \put(2,4.3){\line(1,0){12}}
\put(2,4.9){0}\put(2,4.25){[}
\put(3.7,4.9){\small $\lambda'_1=2$}\put(4,4.25){]}
\put(4.07,4.25){(}
\put(3.2,4.9){\small $1$}
\put(6.4,4.9){\small $\lambda_1$ p.d.} \put(6.6,4.05){*}
\put(8.6,4.9){\small $\lambda'_2$}\put(8.6,4.25){]}
\put(8.65,4.25){(} \put(13.25,4.25){]}\put(13.25,4.9){\small $\lambda'_3$}
\put(10.4,4.9){\small $\lambda_2$ p.d.} \put(10.6,4.05){*}
\put(2.1, 3.9){\vector(0, -1){2}}
\put(4.1, 3.9){\vector(-1, -1){1.8}}
\put(4.3, 3.9){\vector(0, -1){2}}
\put(8.4, 3.9){\vector(-2, -1){3.5}}
\put(2,1){$T_0$}\put(4.2,1){$T_p$}
\put(8.9, 3.9){\vector(0, -1){2}}
\put(13.2, 3.9){\vector(-2, -1){3.5}}
\put(8.6,0.8){$T_q$}
\put(2,5.5){\small $\overbrace{\hspace{0.9cm} }$}
\put(2.2,6){{$\EC$}}
\put(3.35,5.5){\small $\overbrace{\hspace{2.45cm} }$}
\put(4.5,6){{$\EC$}}
\put(4.2,4){\small $\underbrace{\hspace{3.45cm} }$}
\put(6.2,3.2){{$\T$}}
\put(6.8,5.5){\small $\overbrace{\hspace{2.9cm} }$}
\put(8.8,4){\small $\underbrace{\hspace{3.45cm} }$}
\put(8.3,6){{$\EC$}}
\put(10.7,3.2){{$\T$}}
\put(10.75,5.5){\small $\overbrace{\hspace{2.9cm} }$}
\put(12.3,6){{$\EC$}}
\put(13,3){$\ST$}
\put (6.65,4.4){\line(1,0){2.0}}
\put (2,4.4){\line(1,0){2.1}}
\put (10.7,4.4){\line(1,0){2.7}}
\end{picture}
\end{center}
\caption{\label{fig:uni} The map $\lambda \mapsto \ST(f_\lambda)$ is discontinuous at parameters $\lambda'_i$ where the critical point
of $f_\lambda$ is periodic, see Example~\ref{ex:STdisc}.
These parameters are alternated with period doubling parameters $\lambda_i$,
and at $\lambda = 1$, the stability of $0$ changes, so $\EC$ changes too.
The equivalence classes $\EC$ (of partially conjugate maps) and $\T$ (of maps
with the same kneading invariant) are also shown. 
The segments with an additional 
line correspond to maps  in the set $\parabolic$ defined in Section~\ref{subsec:parabolic}.}
\end{figure}
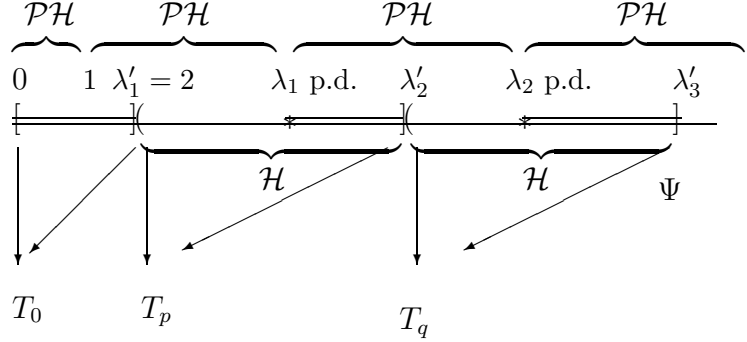

\subsection{\boldmath The definition of $\parabolic$ \unboldmath}
\label{subsec:parabolic}

\begin{definition}\label{def_A}
The set $\parabolic$ is the collection of polynomials $f\in \PB$
such that for each component $B$  of its basin $B(f)$ containing a critical point
the following holds:
\begin{enumerate}[topsep=-0.2cm,itemsep=0.05ex,leftmargin=1.2cm]
\item if $f(\partial B)$ consists of a single point, 
then each points in $f(B)$ has the same itinerary. 
\item if $f(\partial B)$ consists of two points,
then  (i)  each point in $f(B')$ has the same itinerary
where $B'$ be the convex hull of the critical points in $B$
 and (ii) if  $B$  contains an attracting periodic point $p$, then $p$ is in the interior of $B'$.
\end{enumerate}
\end{definition}

Note that $f(\partial B)$ consists of a single point if and only if the interior of $B$ contains an odd number of critical points.
The reason for introducing $\parabolic$ is:

\begin{proposition}\label{prop:PsiPH}
If $f \in \parabolic$ then $\ST(\EC(f)) \subset [\ST(f)]$. 
\end{proposition}

Before proving this proposition we will 
motivate the definition of $\parabolic$ by considering three examples.

\begin{example}\label{example:explainA*}
 If $f_\lambda(x)=\lambda x(1-x)$ is a quadratic map
with an attracting fixed point, then $\ST(f_\lambda)$ is a stunted sawtooth map
which is either equal to the constant map $T_0$ corresponding to the parameter $\zeta=-e$
or to the map $T_p$ which has plateau $[-p,p]$, see
Example \ref{ex:seesawunimodal} and  Figure~\ref{fig:seesawunimodal}. Now 
$T_p\in [T_{-e}]$ while $T_{-e}\notin [T_p]$.
Moreover $\EC(f_{3/2})=\{f_\lambda; \lambda \in (1,3]\}$
and  $\ST(\EC(f_{3/2}))\subset  \{T_{-e},T_p\}\subset [T_{-e}]$.
Note that $f_\lambda\in \parabolic$ when $\lambda\in [0,2]$ and then 
$\ST(f)=T_{-e}$  and so $\ST(\EC(f))\subset [\ST(f)]$.
If $f_\lambda\notin \parabolic$ then this inclusion does {\bf not} hold,
so the assumption that $f \in \parabolic$ is essential for Proposition~\ref{prop:PsiPH} to hold.
Note that it also {\bf not} true that $\ST(\EC(f))\subset \EC(\ST(f))$;
Example~\ref{ex:seesawunimodal} (see also Figure~\ref{fig:uni})
is the simplest counter-example.
\end{example}

The next example, shows why in the definition of $\parabolic$ we add condition 2(ii)
if $f(B)$ consists of two points,   

\begin{example}\label{example:explainA2*}
Assume that $f$ is a cubic map with an attracting fixed point which
attracts both critical points (say with the left critical point a maximum).
Then $\ST(f)$ is equal to one of the following five maps
$T_0, T_1, T_2, T_3, T_4$
determined by $(\zeta_1,\zeta_2)$ equal to $(-e,e), (0,0), (e,-e)$, $(e,0)$ or  $(0,e)$ 
as indicated in Example~\ref{ex:adjacent_plateaus}
and  Figure~\ref{fig:adjacent_plateaus} (when taking $a=-e$ and $b=e$).
Note that  $T_0, T_1, \dots, T_4 \in [T_i]$ when $i=1$ but not when $i=0,2,3,4$.
Also note that $\ST(\EC(f))\subset \{T_0,T_1,T_2\} \subset [T_1]$.
If $f\in \parabolic$ then $\ST(f)=T_1$
and so $\ST(\EC(f))\subset [\ST(f)]$. 
\end{example}

\begin{example}
Finally consider the example of a bimodal map such that there exist  
$s_1, s_2 \in \N$ so that $f^{s_1}(B_1)\subset B_2$ and 
$f^{s_2}(B_2)\subset B_1$ where $B_i$ are distinct  components of 
$B(f)$  and so that $B_1$ and $B_2$ both contain exactly one critical point.
Then the situation is as in Example~\ref{ex:adjacent_plateaus} and 
$\ST(f)$ is contained in the polygon drawn in Figure~\ref{fig:<T>}. 
Maps in $\{ \ST(\tilde f) \st \tilde f \in \EC(f) \}$
correspond to the maps indicated by $1, 2, 4, 6$ in Figure~\ref{fig:<T>}
and their symmetric counterpart under reflexion in
the diagonal of the $(\zeta_1,\zeta_2)$-space.
There are six such maps, all belonging to $\partial [\ST(f)]$.
The fact that $f\in \parabolic$ ensures that $\ST(f)$ corresponds to the lower corner of the region,
\ie the map $T$ denoted by $1$ in Figure~\ref{fig:<T>}.
This map has the property that $[T]$ is equal to this polygon
(this is false for maps denoted by  {\bf hc} and {\bf pd}).
\end{example}

\begin{proofof}{Proposition~\ref{prop:PsiPH}}
Let $T = \ST(f)$ for an arbitrary $f \in \PB$.
The definition of $\ST$ ensures that the orbit of $T(Z_i)$ under $T$ 
is the same as 
an orbit of $S_0$, and therefore no plateau can be mapped into the interior
of a plateau by $T$. 
On the other hand, if $T^k(Z_i) \in \partial Z_j$ for some minimal $k$,
then a small changes of $\zeta_i$ can move $T^k(Z_i)$ into the interior of $Z_j$.
Therefore $T \in \partial \cell{T}$, and in particular,
$\ST(f) \in \cellN{\ST(f)}$ for $f \in \parabolic$.

Take $\tilde f \in \EC(f)$. 
For each $i$ so that $c_i$ is not in the basin of a periodic attractor, $c_i$ and $\tilde c_i$ have the same kneading invariant 
and so the $i$-th component of $\ST(f)$ and $\ST(\tilde f)$ agree.
 
Now let us consider critical points in the  basin of a periodic attractor.
Although kneading sequences within $\EC(f)$
are not constant, all itineraries in the basin of a periodic attractor
(in the limit sense of \eqref{eq:kneadingdef}) are (pre)periodic
to the same periodic sequence in $\{ I_0, \dots , I_b\}^{\N}$.
Since every periodic attractor of $\tilde f$ has a critical point in its immediate
basin, there is $\tilde T \in \SS^b$ that realizes the corresponding
periodic itineraries by orbits that intersect the interior
of a plateau. Additionally, $\tilde T$ can be chosen such that
all kneading sequences of $c_i \in B(f)$ are indeed achieved by 
points in the interior of plateaus.
For this $\tilde T$ and the semiconjugacy $h$ between 
$\tilde T$ and $\tilde f$, we have $h^{-1}(B(\tilde f)) \subset \B(\tilde T)$.
Also $\B(\tilde T) = \B(T)$, so $\tilde T \in \cell{T}$.
Naturally, the periodic attractors of $\ST(\tilde f)$ lie
on the boundary of plateaus, but that still means that $\ST(\tilde f) \in [\tilde T] = [T]$. Since $\tilde f \in \EC(f)$ is arbitrary,  
$\ST(\EC(f)) \subset [\ST(f)]$.
\end{proofof}

\subsection{\boldmath Almost surjectivity of $\ST$. \unboldmath}
\label{subsec:surjectivity}

In Proposition~\ref{prop:realize} below, we shall prove that $\ST$ is almost surjective.
In order to prove this we need  a result from
\cite{MS} concerning full families.  Let us say
that a periodic attractor of a piecewise monotone interval map
$g\colon I\to I$ is  {\em essential} if it contains a turning point in its immediate basin.
We say that $g$ has {\em no wandering intervals},
if each interval $J$ for which $J,g(J),g^2(J),\dots$ are
all pairwise disjoint necessarily intersects the basin of some
periodic attractor. If $g$ has no wandering interval,
then each interval $J$ for which $g^n|J$ is a homeomorphism
for all $n$ is necessarily contained in the basin of periodic
attractor, see \cite{MS}.
It is well-known, see Theorem IV.A in \cite{MS}, that maps in $\PB$ do not have
wandering intervals and that all their attractors are essential.

\begin{theorem}[Fullness of Families]\label{thm:fullness}
Each piecewise monotone map $g$ with $b$ turning points
is topologically conjugate to a polynomial in $\PB$, provided the following two properties are met:
\begin{enumerate}[topsep=-0.2cm,itemsep=0.05ex,leftmargin=1.2cm]
\item $g$ has no wandering intervals
and no inessential attractors;
\item each periodic turning point is an attractor (this is
automatically satisfied if $g$ is $C^1$).
\end{enumerate}
Moreover, assume that $g$ has an attracting periodic point, then one can find $f,\tilde f\in \PB$ 
which are both topologically conjugate to $g$ 
so that the corresponding attracting periodic point is hyperbolic for $f$ and
parabolic for $\tilde f$.
\end{theorem}
\begin{proof}  The first part of this theorem is Theorem II.4.1 in \cite{MS}. 
The second part requires a slightly modifying  the proof
in \cite{MS} on page 124-125.
\end{proof}

The next proposition gives the required surjectivity: 

\begin{proposition}[$\ST$ is almost surjective]\label{prop:realize}
For each  $T\in \SS^b_*$ there exists a polynomial $f\in \PB\cap \parabolic$
such that $T\in [\ST(f)]$. 
\end{proposition}

\begin{proofof}{Proposition~\ref{prop:realize}} 
By Theorem~\ref{thm:<T>}\eqref{exist}, for each $T \in \SS^b_*$ there exists $T' \in \SS^b_*$ so that
 $T' \in \cellN{T'}$ and $T\in [T']$. Therefore,
if we can prove that there exists a polynomial $f\in \PB\cap \parabolic$
such that $T'=\ST(f)$, then the proposition follows as well.

Since $T$ is not piecewise monotone (because of its plateaus), we cannot apply Fullness Theorem~\ref{thm:fullness} directly. 
In order to obtain a piecewise monotone map, we first
replace $T$ on each component $B$ of its basin which contains plateaus,
by an affinely scaled copy of a map $L_q$ as in Figure~\ref{fig:fourshapes} of the appropriate type. 
Here $q$ is  the number of plateaus in $B$.
Let us call the resulting map $T''$. 
We can choose $L_q$ so that $T'$ and $T''$ have the same kneading invariants
(and hence $\cellN{T'} =\cellN{T''}$),
 and so that 
if $T \in [T'']$ has an attracting periodic point 
in the common boundary point of two plateaus, then $T''$ has an attracting
periodic point between the two corresponding turning points. 

\begin{figure}[h]
\begin{center}
\unitlength=7.8mm
\begin{picture}(20,4.7)(1,0.7)
\thinlines \put(1,1){\line(1,0){4}}\put(1,1){\line(0,1){4}}
\put(1,5){\line(1,0){4}}\put(5,1){\line(0,1){4}} 
\put(2.5, 3.5){\tiny $q=3$}
\put(6,1){\line(1,0){4}}\put(6,1){\line(0,1){4}}
\put(6,5){\line(1,0){4}}\put(10,1){\line(0,1){4}} 
\put(7.5, 3.5){\tiny $q=3$}
\put(11,1){\line(1,0){4}}\put(11,1){\line(0,1){4}}
\put(11,5){\line(1,0){4}}\put(15,1){\line(0,1){4}} 
\put(12.5, 3.5){\tiny $q=4$}
\put(16,1){\line(1,0){4}}\put(16,1){\line(0,1){4}}
\put(16,5){\line(1,0){4}}\put(20,1){\line(0,1){4}} 
\put(17.5, 3.5){\tiny $q=4$}
\thicklines
\put(1,1){\line(5,1){1.2}}\put(5,1){\line(-5,1){1.2}}
\put(2.2,1.23){\line(5,-1){0.8}}\put(3,1.09){\line(5,1){0.8}}
\put(6,5){\line(5,-1){1.2}}\put(10,5){\line(-5,-1){1.2}}
\put(7.2,4.75){\line(5,1){0.8}}\put(8,4.91){\line(5,-1){0.8}}
\put(11,1){\line(1,5){0.42}}\put(15,5){\line(-1,-5){0.416}}
\put(11.4,3.1){\line(5,-1){1.1}} \put(12.5,2.9){\line(5,1){1.06}}
\put(14.56,2.9){\line(-5,1){1.0}}
\put(16,5){\line(1,-5){0.42}}\put(20,1){\line(-1,5){0.416}}
\put(16.45,2.9){\line(5,1){1.1}} \put(17.55,3.1){\line(5,-1){1.06}}
\put(19.58,3.08){\line(-5,-1){1.}}
\end{picture}
\end{center}
\caption{\label{fig:fourshapes} The map $L_q$ where $q$ is the 
number of touching plateaus in the component $B$ of the basin of $T\in \cellN{T''}$.
We choose $L_q$ so that it is continuous, piecewise affine and so that 
the slope between its turning points is at most $1/4$.}
\end{figure}
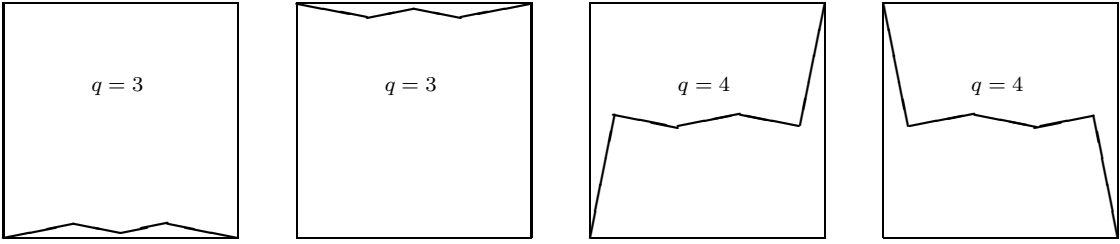

Since $T''$ may still have plateaus, we next define $x\sim y$ for $x,y\in [-e,e]$, if there exists $n\ge 0$ so that
${T''}^n$ maps the convex hull $[x,y]$ into one of the (remaining) plateaus of $T''$. 
Collapse each such interval $[x,y]$ to a point and let $T'''$ be the corresponding map.
From the definition it follows that $T'''$ is continuous and since 
$T\in \SS^b_*$,
it follows that $T'''$ is $b$-modal and has no wandering intervals. 
It also has no inessential attractors (since $T$ does not have these either). 
Hence we can apply Theorem~\ref{thm:fullness} 
 and there exists $f\in \PB$ that is topologically conjugate to $T'''$, 
and therefore have the same kneading invariants. 
The choice of the $L_q$'s corresponds exactly with the definition
of $\parabolic$, so indeed $f \in \parabolic$. 
\end{proofof}

\subsection{\boldmath Almost injectivity of $\ST$. \unboldmath}
\medskip

\begin{proposition}[$\ST$ is almost injective]\label{prop:inj}
The map $\ST\colon \PB \to \SS^b_*$
is `almost injective' in the sense that
if $f,\tilde f \in \parabolic$ and $[\ST(f)]\cap [\ST(\tilde f)]\ne \emptyset$,
then at least one
of $\EC(f) \cap \overline{\EC(\tilde f)}$ and 
$\overline{\EC(f)}\cap\EC(\tilde f)$ is non-empty.
\end{proposition}

\begin{remark}
Here the assumption that $f,\tilde f\in \parabolic$ is needed. Indeed, using the notation from
 Example~\ref{ex:STdisc}, take $f=f_{\lambda_1'+\epsilon}$ and 
$\tilde f=f_{\lambda_3'}$.  Then $\ST(f)=T_p$ and $\ST(\tilde f)=T_q$, and therefore
$[\ST(f)]\cap [\ST(\tilde f)]\ne \emptyset$ while $\overline{\EC(f)}\cap \overline{\EC(\tilde f)} = \emptyset$.
\end{remark}

\begin{proof} If $\ST(f)=\ST(\tilde f)$ then  $f,\tilde f\in \PB$ have the same kneading invariants.
Hence there exists an order preserving bijection 
$h\colon \cup_{c} \cup_{n\in \Z} f^n(c) \to \cup_{\tilde c}\cup_{n\in \Z}\tilde f^n(\tilde c)$
(where the outer union runs over the critical points $c$ of $f$ and $\tilde c$ of $\tilde f$), so that $h\circ f=\tilde f\circ h$.
Here $h$ maps each critical points of $f$ to the corresponding critical point of $\tilde f$.
It follows that if $f$ has no periodic attractors, then $\tilde f$ also has no periodic attractors
and so $f$ and $\tilde f$ are topologically conjugate. Rigidity Theorem~\ref{KSS} then gives that $f=\tilde f$.

If $f$ has a periodic attractor $p$, then define $H \owns p$ to be the largest interval 
such that $f^s(H) \subset H$ and $f^s|H$ preserves orientation.
That is, $s$ is either the period of $p$ if $p$ is orientation preserving and twice the period
otherwise.
Since $f \in \PB$, by taking an iterate of $p$ is necessary, we can assume that $H$ contains a 
critical point $c$, and $( f^{ks}(c) )_{k \ge 1}$ is a monotone sequence in $H$ converging to $p$.
In this case $\tilde f$ also has a periodic attractor in a corresponding interval $\tilde H$. 
However, it is possible that $f^s|H$ has a unique attracting fixed point, while $\tilde f^s|\tilde H$ has two 
attracting fixed points, or vice versa.  It follows that either $\EC(f)=\EC(\tilde f)$ or 
 there exists a map $g\in \PB$ with at least one parabolic periodic point 
such that $g \in \EC(f) \cap \overline{\EC(\tilde f)}$ or $g \in \overline{\EC(f)}\cap\EC(\tilde f)$, see  the last part of 
Theorem~\ref{thm:fullness}. 
Since $\EC(f)$ and $\EC(\tilde f)$ are connected, in particular it follows that the space $\T(f)$ from Theorem~\ref{thm:A} is connected.

Assume now that $[\ST(f)]\cap [\ST(\tilde f)]\ne \emptyset$ and $[\ST(f)]\ne [\ST(\tilde f)]$.
 By Theorem~\ref{thm:<T>}(6), there exists a map  $T_*\in \cellN{T_*}$ so that 
$T_*\in \min([\ST(f)] \cap [\ST(\tilde f)])$. According to Theorem~\ref{thm:<T>}\eqref{bifurcation}
there are four possibilities:
\begin{enumerate}[topsep=-0.2cm,itemsep=0.05ex,leftmargin=0.8cm]
\item $T_*$ has a saddle-node merging two basins (or splitting one into two);
\item $T_*$ has a period doubling orbit;
\item $T_*$ has a pitch-fork orbit;
\item $T_*$ has a homoclinic orbit.
\end{enumerate}
Proposition~\ref{prop:realize} produces a map $f_* \in \PB$
such that $\ST(f_*) = T_*$ and with the corresponding bifurcations.

\noindent
{\bf Claim:} In all these cases, $f_* \in \EC(f) \cap \overline{\EC(\tilde f)}$ or $f_* \in \overline{\EC(f)}\cap\EC(\tilde f)$.
  
\noindent
{\bf Proof of Claim.} This  follows from Theorem~\ref{thm:defspace2}.
Indeed, let $i=1,\dots,b$ and assume that the coordinate $\zeta_i(\ST(f))<\zeta_i(\ST(\tilde f))$. 
  Since $f\in \parabolic$,  by Proposition~\ref{prop:PsiPH},  $\ST(\EC(f))\subset [\ST(f)]$.

{\bf Claim:} One can find a 
 continuous path $f_t$ in $\PB(f)$, $t\in [-1,0]$ so that  $f_t\in \EC^o(f)$ for $t\in (-1,0)$,
 $f_{-1}=f$, $f_0=f_*$. 

\noindent
{\bf Proof of Claim:} This follows from Theorem~\ref{thm:defspace2}. 
Indeed, consider $T=\ST(f)$ and for each of its non-hyperbolic basins, 
consider what type of bifurcations $T_*\in [T]$ undergoes: merging or splitting components of touching basins, see (1)-(4) above.  Then choose the sign $\sigma_i$ for each neutral orbit of $f_t$ 
so that Theorem~\ref{thm:defspace2} ensures that $f_t$    undergoes the corresponding bifurcation as $t\uparrow 0$:
it splits into two components by a saddle-node, a period doubling bifurcation, 
 a pitch-fork bifurcation or a critical point moves from the boundary into the interior of one of the components of the basin. 

Similarly, since $\tilde f\in \parabolic$ we have $\zeta_i(\ST(\tilde f))=\zeta_i(T_*)=\zeta_i(\ST(f_*))$
(in fact  $\zeta_{i'}(\ST(\tilde f))=\zeta_{i'}(T_*)=\zeta_{i'}(\ST(f_*))$ holds for each $i'\in \{1,\dots,b\}$
when $c_{i'}$ is a critical point attracted to the same periodic orbit as $c_i$.)
 Applying Theorem~\ref{thm:defspace2} again, one can find a continuous path 
$f_t$ in $\PB(f)$, $t\in [0,1]$ so that $f_t\in \EC^o(\tilde f)$ for $t\in (0,1)$, $f_0=f_*$ and $f_1=\tilde f$. 
 Since $\zeta_i(\ST(\tilde f))=\zeta_i(T_*)=\pi_i(\ST(f_*))$, we can even make sure that
 $\zeta_i(\ST(f_t))=\zeta_i(T_*)=\zeta_i(\ST(f_*))$ for each $t\in [0,1]$ (so the kneading invariant 
 of the $i$-th critical point does not change as $t$ varies in $[0,1]$). 
 If $\zeta_i(\ST(f))>\zeta_i(\ST(\tilde f))$, these bifurcations occur in reverse.
 If $\zeta_i(\ST(f))=\zeta_i(\ST(\tilde f))$, then the $i$-th kneading invariant of $f$ and $\tilde f$
 are the same, and the argument from the beginning of the proof applies for this critical point.
Since one can apply this argument for each periodic attractor simultaneously,
 the proposition follows.
\end{proof}

The previous proof in particular showed: 

\begin{theorem}\label{thm:kneadconn}
Given a map $f\in \PB$, the set $\T(f)$ of maps $\tilde f\in \PB$ with the same
kneading invariant as $f$ forms a connected set. Moreover, if $\tilde f\in \T(f)$
then $\overline{\EC(f)}\cap \overline{\EC(\tilde f)}\ne \emptyset$.
\end{theorem}

\subsection{\boldmath Almost continuity of $\ST$. \unboldmath}

\begin{proposition}[$\ST$ is almost continuous]\label{prop:cont}
$\ST\colon \PB \to \SS^b_*$ is `almost continuous' in the following sense.
Assume that $f_n\to f$ where $f_n,f\in \PB$ and $f_n\in \parabolic$. Then there exists $T_* \in \SS^b_*$
so that any limit of $T_n\in [\ST(f_n)]$ is contained in $[T_*]$ and so that $\ST(f)\in [T_*]$. 
\end{proposition}

\begin{remark}\label{remark:cont}  It is not necessarily true that the
limit of $T_n\in [\ST(f_n)]$ is contained in $[\ST(f)]$. 
Indeed, let $f$ be a cubic map so that its left critical point
is a fixed point (and is a maximum).  Then $\ST(f)$ is equal to the 
map $T_0$ from Figure~\ref{fig:adjacent_plateaus}. Since
$[T_0]=\{T_0\}$ whereas for maps $f_n\to f$ with $f_n\in \parabolic$ one has that
$[\ST(f_n)]$ is equal to the union of the two triangles in the figure. Note that $f\notin \parabolic$. 
\end{remark}

\begin{proofof}{Proposition~\ref{prop:cont}}
Take $T_n=\ST(f_n)$ and $T=\ST(f)$. 
By taking a subsequence, we can assume that $T_n$ converges to some map $\tilde T$.
  Let $Z_i, Z_{i,n}$ be the plateaus associated to 
$T$ and $T_n$ respectively. Note that $\cell{T}$ has a product structure, 
see Theorem~\ref{thm:<T>}\eqref{prod} and let $\pi_i$ and $\pi_J$
be the projections as defined in that theorem. 
 If $c_i$ is not eventually mapped onto another critical
point, then $\nu_i(f_n)\to \nu_i(f)$ as $n\to \infty$
(in the usual topology on sequence spaces) and so
$Z_{i,n}\to Z_{i}$ as $n\to \infty$. 
That is, $\pi_i(\ST(f_n))\to \pi_i(\ST(f))$. 
In this case define $\zeta_i:=\zeta_i(\ST(f))$.
If  $c_i$ is mapped onto another critical point, say $f^k(c_i)=c_j$
but is not in the immediate basin of a periodic attractor,  then $T^k(Z_i)\subset \partial Z_j$ 
and for a sufficiently small neighborhood $U_i$ of $Z_i$ 
one has $T^k(U_i)\cap Z_j\subset \partial Z_j$. 
That is, if $T$ has a maximum (minimum) at $Z_i$ then $T(Z_i)$ is the left (respectively right) endpoint
of a component of $\B(T)$.  Since $f_n\to f$ and $c_i$ is not in the basin of a periodic attractor,
one has that $\tilde T^k(\tilde Z_i)\subset \partial \tilde Z_j$
where $\tilde T^k(\tilde Z_i)$ is possibly the {\lq}other{\rq} endpoint of $\tilde Z_j$.
It follows that $\pi_i([\ST(f)])$ is equal to an interval of the form
$[a_i,b_i]$ and that $\pi_i(\ST(f_n))$ converges to an endpoint of this interval.
In this case define $\zeta_i:=a_i$. 
If $c_i$ is periodic, then $f_n$ also has an attracting periodic point 
$p_n$ near $c_i$. Let $B_n$ be the component of the immediate basin 
containing $p_n$ and let $B$ be the component of the immediate basin
containing $c_i$ and we define $\zeta_i:=\zeta_i(\ST(f))$. 
If $f_n(\partial B_n)$ consists of one point, then $f_n\in \parabolic$
implies that the itinerary w.r.t. $f_n$ of each critical point in $B_n$ agrees with 
the itinerary w.r.t. $f$ of the corresponding critical point in $B$.
In this case $\pi_i(\ST(f_n))=\pi_i(\ST(f))$. If  $f_n(\partial B_n)$ consists
of two points, then this no longer needs to be the case, see 
Example~\ref{example:explainA2*} and Remark~\ref{remark:cont}.
In this case $\pi_i([\ST(f_n)])$ corresponds to a set as in 
Figure~\ref{fig:adjacent_plateaus}
and Figure~\ref{fig:<T>}
and $\pi_i(\ST(f))$ is in the closure of this set. Since  $f_n\in \parabolic$
it follows that $\ST(f_n)$ is equal to the point marked 1 in these figures.
Since $\ST(f_n)\in  \cell{\ST(f_n)}_\flat$, we get that
 $[\ST(f_n)]$ is equal to the closure to this set
 and we set $\zeta_i=\ST(f_n)$.
If $c_i$ is mapped to another critical point and 
in the immediate basin of a periodic attractor, then the same argument goes through. 
The map $T_*$ for which $\zeta_t(T_*)=\zeta_i$ for $i=1,\dots,d$ 
and where $\zeta_i$ is chosen as above,  has the required properties.
\end{proofof}

\section{Proof of the Main Theorem
 }\label{sec:proofofthm}

In this section we shall prove the Main Theorem, assuming Theorem~\ref{Thm:ConnectedI}  
(which will be proved as Theorem~\ref{Thm:Connected} in the next section). First let us prove the following

\begin{theorem}[{\lq}Fibers{\rq} of $\ST$ are connected]
\label{thm:Psiproper} Assume that $K$ is a closed and connected
subset of $\SS^b_*$ with the property that
if $T\in K$ then $[T]\subset K$.
Then $\ST^{-1}(K)$ is connected.
\end{theorem}

\begin{remark}\label{rem:LimitTn} Note that $\SS^b_*$ is not a closed
subset of $\SS^b$.  Therefore we merely assume that $K$ is a closed subset
in the relative topology of  $\SS^b_*$ meaning that if $T_n\in K$
converges to $T\in \SS^b_*$  then $T\in K$.
\end{remark}

\begin{proof}
Take a closed connected set $K\subset \SS^b_*$, and assume by
contradiction that $\ST^{-1}(K)$ is not connected.
This means that there are
disjoint open sets $U_1,U_2 \subset \PB$
so that $U_1\cup U_2 \supset \ST^{-1}(K)$ and
$C_i:=U_i\cap \ST^{-1}(K)\ne \emptyset$, $i=1,2$.
Write $[\ST(C_i)]:=\cup_{f\in C_i}[\ST(f)]$.\\
{\bf Claim 1:} $[\ST(C_1)]\cup [\ST(C_2)]\supset K$.
Indeed, it  follows from surjectivity (Proposition~\ref{prop:realize}) that for every
$T \in K$ there exists  $f \in \PB$ such that
$T \in [\ST(f)]$. Since $[\ST(f)]\cap K\ne \emptyset$
we have by assumption that $[\ST(f)]\subset K$
and therefore $f\in \ST^{-1}(K)$. Therefore $f\in C_1\cup C_2$ and since
$T\in [\ST(f)]$, the claim follows.
\\
{\bf Claim 2:} $[\ST(C_i)] \cap K$ is closed (again in the relative
topology of $\SS^b_*$).
To see this, take a sequence $T_n \in [\ST(f_n)] \cap K$ with $f_n\in C_i$.
By Theorem~\ref{thm:<T>}\eqref{exist} and Proposition~\ref{prop:realize}
we can assume that $f_n\in \parabolic$. 
By considering subsequences we may assume that
$T_n \to T$ for some $T \in K$ and
$f_n \to f$ for some $f\in U_i$. Since $[\ST(f_n)]\cap K\ne \emptyset$
we have  $[\ST(f_n)]\subset K$.  Because of this and because  $f_n\in \parabolic$,
continuity (Proposition~\ref{prop:cont}) implies that  $[\ST(f)]\subset K$.
Hence $f\in \ST^{-1}(K)\subset U_1\cup U_2$ and, since $f_n\in C_i\subset U_i$ converges to $f$,
also $f\in U_i\cap \ST^{-1}(K)=C_i$.  This completes the proof of Claim 2. \\
{\bf Claim 3:} $[\ST(C_1)]\cap [\ST(C_2)]\ne \emptyset$.
This follows from the connectedness of $K$ and Claims 1 and 2.

\medskip
By Proposition~\ref{prop:realize}, there exist therefore $f_i\in C_i \cap \parabolic$ such that
$[\ST(f_1)]\cap [\ST(f_2)]\ne \emptyset$.
By injectivity (Proposition~\ref{prop:inj}), this implies that
$$
\EC(f_1) \cap \overline{\EC(f_2)}\ne \emptyset \quad \text{ or } \quad 
\overline{\EC(f_1)} \cap \EC(f_2)\ne \emptyset.
$$
Moreover, by Proposition~\ref{prop:PsiPH},\,
$\ST(\EC(f_i))\subset [\ST(f_i)]\subset K$.
Hence  $\EC(f_i)\subset \ST^{-1}([\ST(f_i)])\subset C_i$.
Since $\EC(f_1)$ and $\EC(f_2)$ are both connected,
this contradicts that $C_i\subset U_i$ with $U_1,U_2$ disjoint.
\end{proof}

\begin{proofof}{Main Theorem}
By Theorem~\ref{Thm:Connected} (and the remark below) it follows that level sets
of $h_{top}\colon\SS^b_* \to \R$ are connected.
Moreover, $h_{top} \colon \PB\to \R$ agrees
with $h_{top}\circ \ST$. Because the topological entropy
of each map in $[T]$ is the same,
Theorem~\ref{thm:Psiproper} shows that the isentropes 
lift to connected sets in $\PB$. Similarly, the set 
$I(h_0^)$ is connected.
\end{proofof}

\section{\boldmath Isentropes in
$\SS_*^b$ are contractible. \unboldmath}
\label{sec:connS*}

Recall from Definition~\ref{def:wander}
that $\SS^b_*$ is the collection of  {\em non-degenerate}
stunted sawtooth maps $T \in \SS^b$.  
That is, by definition, if $\J:=[Z_i,Z_j]$ is the convex hull of $Z_i$ and $Z_j$, and
there is $n\ge 0$ such that $T^n(\J)$ is a point
(so $(Z_i, Z_{i+1})$ forms a wandering pair), 
then $T \in \SS^b_*$ means that $T^n(\J)$ is eventually mapped into the closure
of the immediate basin a periodic plateau.

Of course, if $(Z_i,Z_j)$ is a wandering pair, then all
plateaus between $Z_i$ and $Z_j$ form wandering pairs as well.
The subset $\SS_*^b \subset \SS^b$   is chosen because
$\ST\colon \PB \to \SS^b$ fails to be  surjective in a serious way
(whereas $\ST\colon \PB \to \SS^b_*$ is
almost surjective in the sense of Proposition~\ref{prop:realize}).
Indeed, if $T\in \SS^b\setminus \SS^b_*$ has a non-preperiodic
wandering pair $(Z_i,Z_j)$ and $\ST(f) \in [T]$,
then $f$ has a wandering interval $[c_i,c_j]$.
However, it is well-known (see \eg \cite{MS}) that polynomials, and in fact
$C^2$ interval maps with non-flat critical points, have no wandering
intervals.

This section is devoted to proving Theorem~\ref{Thm:ConnectedI}, which we restate as

\begin{theorem}\label{Thm:Connected}
Let $L(h) = \{ T \in \SS^b_\shape \st h_{top}(T) = h \}$ and $L_*(h) = L(h) \cap
\SS_{\shape,*}^b$. Then 
\begin{itemize}
\item for every $h_0\in [0, \log(b+1)]$, the level set $L_*(h_0)$
is a contractible subset of $L(h_0)$;
\item 
$L_*(h_0^+):= L_*(h_0)\cap  \mbox{closure}( \{ T \in \SS^b_* \st h_{top}(T)) > h_0\})$ is contractible. 
\end{itemize}
\end{theorem}

\begin{remark} The sets $L_*(h_0)$ has the property that if $T$ is contained
in one of these sets, then $[T]$ is contained also in this set.
This property holds for $L_*(h_0)$ since each map in $[T]$
has the same topological entropy. 
Since $L_*(h_0^+)$ is connected and each set of the form $[T]$
is connected,  the set
$$[L_*(h_0^+)]:=\cup_{T\in L_*(h_0^+)}[T]$$
is also connected.
\end{remark}

Throughout this section  we fix $h_0\in [0, \log(b+1)]$, although we separate the easier cases $h_0 = 0$ and $h_0 = \log(b+1)$,
see Sections~\ref{subsec:h0=logb+1} and \ref{subsec:h0=0}.

That $L(h_0)$ is contractible was already proved in Theorem 6.1 in \cite{MTr}.
Contractibility of $L_*(h_0)$ is much more difficult, and we have to adjust
the proof of \cite{MTr} in a delicate way.
The proof involves the construction of a retract $R$ composed of 
entropy decreasing deformations (to contract $L(h)$ to a single point)
and entropy increasing deformations (to keep $L(h)$ within itself).
The problem is to keep $L_*(h)$ within itself under continuous action of
the retract.
To this end we are forced to compose $R$ of altogether
five deformations, with some auxiliary deformations.
We use the letters $\Ei, \ei, \Eih$ to indicate
entropy increasing deformations, and  $\ed, \hed,  \Edh$
for entropy decreasing deformations. The deformation $\beta$ will not change
entropy.
The letters $R$ and $r$ stand for retract.

Before we are able to give the proof of this theorem we will
develop the necessary ingredients.

\subsection{The piecewise affine case.}
An interval $K$ is a \emph{renormalization interval} an interval map $f$ if
$f^n(K)\subset K$ for some $n \ge 1$ and $f^n(\partial K)\subset  \partial K$.
If $n = 1$, and $K  = I$, then this is a renormalization interval only in a 
trivial sense, but we still want to consider it as such.
The set  $\orb(K)=K\cup f(K)\cup \dots \cup f^{n-1}(K)$ is called
a \emph{renormalization cycle}.

It is well-known \cite{MT}, that every interval map of entropy $h > 0$
is semi-conjugate to a piecewise affine interval map with slope $\pm e^h$.
The semi-conjugacy is a monotone map, and collapses every interval
that doesn't contribute to the exponential growth rate of the lapnumber;
these are wandering intervals, basins
of periodic attractors and possibly renormalization intervals, as well
as intervals that map into those.

The following lemma will be used at several places in the rest of the proof.

\begin{lemma}\label{lem:decrentr}
Assume that  $F\colon [0,1]\to [0,1]$ is a continuous, piecewise
affine map with at most finitely many plateaus $Z_i$,
and slope $|F'| > 1$ outside these plateaus.
Suppose that $Z$ is a turning plateau (or point) in a minimal\footnote{\ie $K$ contains no strictly smaller renormalization interval.} renormalization cycle $\orb(K)$,
such that no neighborhood of $Z$ is ever mapped into a plateau. 
Let $J$ be a neighborhood of $Z$ 
so that $F(\partial J)$ is a single point, then the function  $\tilde F$ defined by
$$
\tilde F(x) = \left\{ \begin{array}{ll}
F(\partial J) &\mbox{ if } x \in \overline{J},\\
F(x) &\mbox{ if } x \notin \overline{J},
\end{array}\right.
$$
satisfies $h_{top}(\tilde F|\orb(K)) < h_{top}(F|\orb(K))$.
\end{lemma}

\begin{remark}
The origins of the following proof are somewhat nebulous to
us. Jozef Bobok drew our attention to the argument, ascribing it
to Sasha Blokh, but we haven't been able to locate a precise source.
Related results were proved by Boyland \cite{boyland}
and Block \& Ledis \cite{BlockLedis}.
\end{remark}

\begin{proof}
The inequality $\tilde h := h_{top}(\tilde F) \le h := h_{top}(F)$ follows directly
from the definition of $\tilde F$; in fact
$|J| \mapsto h_{top}(\tilde F)$ is a decreasing function.
However, we need to prove strict inequality.
Let $Y = \orb(K) \setminus \cup_n F^{-n}(\interior(\cup_i Z_i))$.
Since $\orb(K)$ is a minimal cycle and the derivatives $|F'| > 1$ on $Y$,
the restriction $F:Y \to Y$ is transitive and supports a unique measure 
of maximal entropy, see \cite{Hof}.
The assumption on $Z$ implies that $\partial Z \subset Y$ and in fact $\mu(J \setminus Z) > 0$.

Now $\tilde F|\orb(K)$ is entropy-preservingly semi-conjugate (say via $\psi$)
to a map with slope $\pm e^{\tilde h}$.
Let $\tilde \nu$ be the measure of maximal entropy of this map, and
$\nu = \tilde \nu \circ \psi$.
Then $0 = \nu(\tilde F(J)) \ge \nu(J)$, because $\nu$ is non-atomic.
It follows that $\supp(\nu) \cap \interior(J) = \emptyset$, and definitely
$\nu \neq \mu$, whilst at the same time $\nu$ is not only 
$\tilde F$-invariant, but $F$-invariant as well.
Since $\mu$ is the unique measure of maximal entropy of $F|\orb(K)$,
it follows that $\tilde h < h$.
\end{proof}

\subsection{\boldmath Increasing  the entropy of maps in $\SS$:
$\ei_t$ and  $\Ei_t$.\unboldmath }\label{subsec:increase}
The stunted seesaw map $T:[-e,e] \to [-e,e]$ is uniquely determined by the parameters $(\zeta_1,\dots,\zeta_b)$, and so we can define the norm
\begin{equation}\label{eq:dist}
\dist(T, \tilde T)=\max_{i \in \{ 1, \dots, b \}} |\zeta_i-\tilde \zeta_i|.
\end{equation}
Write $T \le \tilde T$ if the parameters satisfy $\zeta_i \le \tilde \zeta_i$
for all $i$. Similarly $T < \tilde T$ if $T \le \tilde T$ and 
$\zeta_i < \tilde \zeta_i$ for at least one $i$. 
Notice that 
\begin{equation} \label{eq:order}
T \le \tilde T \mbox{ implies } h_{top}(T)\le h_{top}(\tilde T).
\end{equation}

{\bf Construction of \boldmath $\ei_t$: \unboldmath}
Let $\ei_t$ linearly increase all
parameters: $\ei_t\colon \zeta_i \mapsto \zeta_i + 2e t$, as long as
they do not map to $\pm e$. 

Let $W^o(T)$ be the components of the domain of the first entry map to 
the interior of plateaus of $T$:
$$
\begin{array}{rl}  W^o(T) := 
&\left\{x; \, \exists j\in \{1,\dots,b\} \mbox{ and }k\ge 0\mbox{ such that }\right. \\
&\left.  \quad T^k(x)\in \interior(Z_j) \mbox{ and }x,T(x),\dots,T^{k-1}\notin Z_j\right\} .\end{array}
$$ 
We say that $T$ satisfies the {\bf $\beta$-property} if 
\begin{equation}\label{eq:nointerior}
\begin{array}{rl} &\,\, \mbox{ no interval of the form
$\J_{i}=[Z_{i},Z_{i+1}]$ is contained in }W^o(T).
\end{array} 
\end{equation} 

\begin{lemma} Assume that $T$ satisfies the $\beta$-property (\ref{eq:nointerior}).
Then $\ei_t(T)\in \SS_*^b$  for all $t>0$.
\end{lemma} 
\begin{proof} For every fixed integer $u\ge 0$,
as $t$ increases, the image of the interval
$\J_{i,t}=[Z_{i,t},Z_{i+1,t}]$ under the $u$-th iterate of the map $\ei_t(T)$
becomes larger while
the sizes of plateaus shrink. 
If the $\beta$-property (\ref{eq:nointerior}) holds  then it follows that
$\J_{i,t}$ cannot be mapped into a non-periodic plateau for $t>0$ and therefore no new
wandering pairs can be created by the deformation $\ei_t$
as $t$ increases. Also $\ei_t(T)$ has no non-trivial blocks 
of touching plateaus for $t>0$. This implies the lemma.
\end{proof}

{\bf  Construction of \boldmath $\Ei_t$: \unboldmath}
The deformation $\Ei_t$ uses $\ei_t$
and the following observation. If $\dist(T,\tilde T)<\eps/(2e)$, then
$\dist(\ei_t(T), \ei_t(\tilde T))<\eps$ for every $t>0$.
It follows that
\[
\ei_{t-\eps}(T)<\ei_t(\tilde T) < \ei_{t+\eps}(T)
\]
and so by \eqref{eq:order},
\[
h_{top}(\ei_{t-\eps}(T)) \le
h_{top}(\ei_t(\tilde T)) \le h_{top}(\ei_{t+\eps}(T)).
\]
Hence the function 
\begin{equation}\label{eq:tmax}
t_{max}(T) := \max\{ t \ge 0\st h_{top}(\ei_t(T)) = h_0\}
\end{equation}
is continuous in $T$ provided $h_{top}(T)\le h_0$.
In particular, $\Ei_t(T) := \ei_{t\cdot t_{max}(T)}(T)$ is continuous in $t$ and $T$ as well. 
By construction, 
\begin{equation} \label{eq:makegood}
\begin{array}{rl}
&\left\{T\in \SS^b \ ; \ h_{top}(T)<h_0,
T \mbox{ satisfies the $\beta$-property \eqref{eq:nointerior}}\right\} \implies \\
& \\
&  \quad \quad \quad \Ei_t(T)\in  \cup_{s \le h_0} L_*(s),  \forall  t\in (0,1].\end{array}
\end{equation}
Moreover,  if  $T$ satisfies the $\beta$-property \eqref{eq:nointerior}
then $\Ei_t(T)$ also satisfies the $\beta$-property \eqref{eq:nointerior} for $t\in (0,1]$.

\subsection{\boldmath Decreasing entropy of maps in $\SS^b$: $\ed_t$, $\hed_t$ and $r_t$. \unboldmath}
\label{subsec:deltat}
In this section we define the basic operations to decrease entropy, although later we 
will need refined versions of them.

{\bf  Construction of \boldmath $\ed_t$: \unboldmath}   Define the `sign'
\begin{equation}\label{eq:sgn}
\sgn(Z_i) = \left\{ \begin{array}{ll}
0 & \text{ if } Z_i \text{ touches another plateau, }\\
1 & \text{ otherwise.}
\end{array} \right.
\end{equation}
Now deform $T$ according to the flow
defined by the system of differential equations on the parameters,
where we will now indicate the $t$-dependence by $\zeta_{i,t}(T)$.
\[
\begin{array}{ll}
\dfrac{d\zeta_{i,t}}{dt}(T) &= \quad \left\{
\begin{array}{ll}
-\, 2\, \sgn(Z_i) e &  \mbox{ if } \zeta_{i,t}(T)\in (-e,e), \\[1mm]
0   & \mbox{ otherwise,}\end{array} \right. \\[2mm]
\zeta_{i,t}(T)|_{t=0} &= \quad \quad \zeta_i(T). \end{array}
\]
Let us denote the resulting deformation by $\ed_t$;
it continuously decreases/increases the height of a plateau if
it is local maximum/minimum of $T$ until this plateau
touches a neighboring plateau, or reaches the boundary of $I$.
Furthermore, if $T$ has no touching plateaus then 
$h_{top}( \gamma_s \circ \ed_t(T) )\le h_{top}(T)$ for 
$s\in (0,t)$ and $t>0$ sufficiently small.
Note, however, that $\ed_t(T)=\ed_s(T)$ for
$s,t\ge 1$ but this might mean that each plateau of $\ed_t(T)$
touches another plateau, and therefore this is not a guarantee that
$h_{top}(\ed_t(T))=0$. Therefore this deformation, although necessary 
(see Section~\ref{subsec:careful}), is not sufficient for
our purposes.

{\bf  Construction of \boldmath $\hed_t$: \unboldmath}
A natural variant is the deformation $\hed_t$ which will widen plateaus in order to decrease entropy to $0$.
The difference in the deformations $\ed_t$ and $\hed_t$ is
that $\ed_t$ decreases the parameter $\zeta_i$ until the
corresponding plateau $Z_i$ touches another plateau (or reaches $e$);
blocks (consisting of more than one plateau) do not move under $\ed_t$.
By contrast, the deformation $\hed_t$ will also move blocks of plateaus, provided they form
a local extremum,
which happens whenever the block consists of an odd number of plateaus.
Blocks of an even number of plateaus are not moved by $\hed_t$ (unless
an extra plateau joins the block).
As it turns out, this may introduce new wandering pairs.

Define the `sign' for plateaus that are part of a block of plateaus:
\begin{equation}\label{eq:hsgn}
\hsgn(Z_i) = \left\{ \begin{array}{rl}
0 & \text{ if } Z_i \text{ is part of a block of an even number of
plateaus,}\\
1 &  \text{ if } Z_i \text{ touches no other  plateau or is an odd-numbered}\\
& \qquad \text{plateau in a block of an odd number of plateaus,}\\
-1 &  \text{ if } Z_i \text{ is an even-numbered plateau in a block of
an odd}\\
& \qquad \text{number of plateaus.}
\end{array} \right.
\end{equation}
Note that $\hsgn(Z_i)$ depends not only
on $i$ but also on $T$, and $\hsgn(Z_i) = \pm 1$ means that $T$ has
a local extremum at the block of plateaus that $Z_i$ is part of.
We deform $T$ according to the flow
defined by the system of differential equations on the parameters.
\[
\begin{array}{ll}
\dfrac{d\zeta_{i,t}}{dt}(T) &= \quad \left\{
\begin{array}{ll}
-2\cdot  \hsgn(Z_i) \cdot e &  \mbox{ if } \zeta_{i,t}(T)\in (-e,e),\\[1mm]
0   & \mbox{ otherwise,}\end{array} \right. \\[2mm]
\zeta_{i,t}(T)|_{t=0} &= \quad \zeta_i(T). \end{array}
\]
The differential equation defines
a continuous deformation $\hed_t$
with the property that $t \mapsto \hed_t(T)$
(not necessarily strictly) decreases the topological entropy.
During the deformation blocks
can collide, and then the combined larger blocks are deformed
according to the same rule. (As a result $\hsgn(Z_i)$ can change
during the deformation.)

{\bf  Construction of the retract \boldmath $r_t$: \unboldmath}
If $b$ is odd and $t\ge 1$, then all plateaus of $\hed_t(T)$
will touch and the map $\hed_{t}(T)$ is constant $\pm e$.
If $b$ is even and $t\ge 1$,
then each map  $\hed_{t}(T)$ will be monotone
(with some  blocks of touching  plateaus).
More precisely,  if $t\ge 1$ then  $\hed_t(T)\in \Sigma^b_\shape$
where
$$
\Sigma_\shape^b = \left\{ \begin{array}{ll}
\{ T_0(x) \equiv \pm e \} & \mbox{ if $b$ is odd and } \shape = \mp 1; \\
\{ \mbox{monotone maps in } \SS_\shape^b \}  & \mbox{ if $b$ is even.}
\end{array} \right.
$$
Since $\Sigma_\shape^b$ is a singleton in the first case and a simplex in the second case, there exists a continuous retract
$r_t\colon \Sigma^b_\shape\to \Sigma^b_\shape$
with $r_0=id$ and $r_1 \equiv T_0$ for $T_0$ some map in $\Sigma^b_\shape$.

{\bf Construction of a retract $R_t$ of an isentrope of $\SS^b$:} 
If we only had to construct a retract of an isentrope of $\SS^b$ then we
could finish the construction of a deformation $R_t$ as follows. 
Define 
\[
R_t = \left\{ \begin{array}{ll}
\Gamma_{3t}&\mbox{ for }t\in [0,1/3],\\
\Gamma_1\circ \hed_{3t-1}&\mbox{ for }t\in [1/3,2/3],\\
\Gamma_1\circ r_{3t-2}\circ \hed_1&\mbox{ for }t\in [2/3,1].
\end{array}
\right.
\]
%
%
Obviously this defines the required retract of an isentrope with the
space $\SS^b$ to a point.
However, as we will see, 
this is still insufficient for our purposes because it is not a retract of an isentrope
within the smaller  space $\SS_*^b$. In the remainder of this section we 
will show how to modify this construction to obtain a retract of an isentrope in 
$\SS_*^b$.

\subsection{\boldmath The case $h_0=\log(b+1)$. \unboldmath}\label{subsec:h0=logb+1}

There is only one stunted sawtooth map with entropy $h_{top}(T) = \log(b+1)$
(just as there is only one polynomial of give $\shape$ with entropy
$h_{top}(f) = \log(b+1)$, namely the Chebyshev polynomial.
Hence this level is trivial.

\subsection{\boldmath The retract $R_t$ for the case $h_0=0$. \unboldmath}\label{subsec:h0=0}

Let us give the proof in the case that $h_0=0$, as the construction will be
much easier in this case. Consider a map 
$T\in \SS^b_*$ with zero topological entropy.   Let us first review some results
on the renormalization structure of such maps.

We say that $T\in \SS^b$ has a {\em $2$-renormalizable}, if there exists 
an interval $K$ so that $K$ and $T(K)$ have disjoint interiors
and $T^2(K)\subset K$.

\begin{lemma}\label{lem:2periodic} 
Let $T\in \SS^b$ and assume that $h_{top}(T) = 0$. 
If $T$ is not $2$-renormalizable, then the $\omega$-limit set of 
each point in $I$ is a fixed point of $T$.
\end{lemma}

\begin{proof}  Since $T$ has zero topological entropy, it follows that there exists no interval $J$ 
so that  $T^2(\partial J)\subset \partial J$ and $T^2(J)\supset J$.
From this, and since $T$ is not $2$-renormalizable, 
this well-known lemma easily follows. 
\end{proof}

\noindent
{\bf Remark:} If $T$ is $2$-renormalizable then one can apply 
the lemma again to $T^2 \colon K\to K$. If one can repeat this infinitely often, then the map is {\em  infinitely renormalizable}. In this case, for each $k\ge 0$, the map $T$
has one or more periodic points of period $2^k$  and no other periodic points.
If $T$ is not infinitely often renormalizable, then each point is eventually
mapped into a periodic point of $T$ or is in the basin of a periodic orbit of period $2^k$.

\begin{lemma}\label{lem:gammas}
Let $T \in \SS^b_*$ with $h_0=h_{top}(T) = 0$ and assume that $t>0$.
Then there exists $k_0<\infty$ such that
all periodic attractors of $\hed_t(T)$ have period $2^k$, $k\le k_0$.
Moreover, each point is either (pre-)periodic or in the basin of one
of the periodic attractors of $\hed_t(T)$.
\end{lemma}
\begin{proof}
From the previous lemma it follows that each interval map of zero entropy
and finite modality has only periodic points of period $2^k$
for $k \in \N$.
Take $T \in \SS^b_*$ such that $h_{top}(T) = 0$.
If $Z_i$ is a plateau with an infinite orbit, then
$T$ must be infinitely renormalizable, \ie there exists
a sequence of periodic intervals $K_u$, $u \in \N$, with period $2^u$
such that $\cap_u K_u \supset Z_j$ for some $j$, and
$\omega(Z_i) = \omega(Z_j)$.
In fact, $\orb(K_u)$ can contain more plateaus, but since the period of
$K_u$ tends to infinity as $u \to \infty$, and there are only $b$ plateaus,
we can assume (by an appropriate choice of $K_u$)
that there exists $n_u \to \infty$ as $u \to \infty$
such that $T^{n}(K_u)$ does not intersect any plateau for $0 < n < n_u$.
Therefore $|T(K_u)| \to 0$ as $u \to \infty$.

Since $T \in \SS^b_*$ there exists $t_1 \in (0, t)$ and $\eta > 0$
such that all plateaus of $T$ move at least $\eta$ when
$t'$ moves from $0$ to $t_1$ (for $t'$ small $\hed_{t'}$ agrees with $\ed_{t'}$).
For $u$ so large that $|T(K_u)| < \eta$, this
means that $K_u$ is no longer invariant under $\hed_{t_1}(T)$
and so this map is not infinitely renormalizable anymore.
Instead, there is $k_0$ such that
every plateau of $\hed_{t_1}(T)$ is (eventually) periodic with period
$2^k$ for some $k \le k_0$. If we increase $t'$ further from $t_1$ to $t$,
each periodic attractor remains but  can undergo period halving bifurcations.
So all periodic orbits of $\hed_t(T)$ have period $2^k$, $k \le k_0$.
\end{proof}

Now define a retract $R_t$ of the
zero-entropy level set of $\SS^b_*$ as follows: 
\[
R_t = \left\{ \begin{array}{ll}
\hed_{2t}&\mbox{ for }t\in [0,1/2],\\
r_{2t-1}\circ \hed_1&\mbox{ for }t\in [1/2,1].
\end{array} \right.
\]
Lemmas~\ref{lem:gammas} and \ref{lem:2periodic} (and the remark above)
imply that under $\hed_{2t}(T)$ (resp.\ $r_{2t-1}\circ \hed_1(T)$),
each plateau is contained in the closure of a component of the
basin of one of the periodic attractors of $\hed_{2t}(T)$  (resp.\ $r_{2t-1}\circ \hed_1(T)$).
Hence $\hed_{t}(T),r_t  \circ \hed_1(T)\in \SS_*^b$.
Thus we obtain a retract of $\SS_*^b$ and proved Theorem~\ref{Thm:Connected} in the zero entropy case.

The remainder of this section will deal with the case $h_0>0$,
which is plagued with additional difficulties.

\subsection{\boldmath The retract $R_t$ when $h_0 > 0$ and
the trouble with $\hed_t$.\unboldmath}\label{subsec:trouble}
As mentioned at the end of Subsection~\ref{subsec:deltat}
the retract is insufficient for our purposes: 
we need to construct a retract of an isentrope of $\SS^b_*$
(so the deformation is not allowed to leave the space $\SS^b_*$).
The hurdle we have to overcome is that if $T\in \SS^b_*$,
then $\hed_t(T)$ need no longer be in $\SS^b_*$ for $t>0$,
because the deformation $t\mapsto \hed_t(T)$ can create wandering pairs $(Z_i,Z_j)$.
To resolve this issue, the aim is to ensure that the deformation $\Gamma_t$ (or a similar deformation)
will be able to `undo' these wandering pairs. In view of property 
\eqref{eq:makegood} on page \pageref{eq:makegood}
we will construct a deformation $\beta_t$
with the property that if $T\in \SS^b$ then $\beta_1(T)\in \overline{\SS^b_*}$.
It does this by deforming $T$ in such a way that $\beta_1(T)$ never eventually maps
an interval of the form $[Z_i, Z_{i+1}]$  into the interior of another plateau.
However, \eqref{eq:makegood} only applies to maps
$T\in \overline{\SS^b_*}$ with topological entropy $h_0$. For this reason
we need to define a more subtle way of `decreasing' and `increasing' 
the map $T$ while the topological entropy remains equal to $h_0$. 
These analogues of $\edh$
and $\Gamma_1$ may act on some of the plateaus while leaving some others
alone. The challenge will be to define this as a continuous deformation.
To achieve this, we will introduce three additional deformations.
\begin{itemize}
\item[$\Eih_t$:]  To increase the topological entropy more carefully by increasing some (but possibly not all)
$\zeta_i$'s, so we essentially increase each $\zeta_i$ `as far as possible'.
This is the purpose of $\Eih_t$ defined in Section~\ref{subsec:eih}.
\item[$\hat \Delta_t$:]  To decrease the topological entropy more carefully in such a way that if
$\hat \Delta_t(T')$ or $\Gamma_t(T')$ does not move certain plateaus (because otherwise the entropy would become too large),
then we `can assume' that $T'\in \SS_*^b$. This is the purpose of the deformations  $\Edh_t$ and $\hat \Delta_t$ defined in Section~\ref{subsec:careful}.
\item[$\beta_t$:]
Finally, we want to ensure that we only need to apply $\Eih_s$  to maps
$T'\in  \SS^b$ with the property that a convex hull  $[Z_i, Z_{i+1}]$
is never eventually mapped into the interior of another plateau (\ie only to maps
with $T'\in \overline{\SS^b_*}$).  This means that $\Eih_s(T')\in \SS^b_*$  for {\em every} $s>0$.
Unfortunately, this may not be enough because it may happen that $h_{top}(\Eih_s(T'))>h_0$ for any $s>0$.
The deformation $\beta_t$, defined in Section~\ref{subsec:betat}, 
prevents this, but will not change the entropy.
\end{itemize}

\subsection{\boldmath The construction of $\beta_t$. \unboldmath }
\label{subsec:betat}
Consider a map $T'=\hed_t(T)$ so that the convex hull $\J_i:=[Z_i,Z_{i+1}]$ of
two neighboring plateaus is eventually mapped
into the interior of another plateau. If this happens then for $s>0$ small,
$\gamma_s(T')$ will still have this property and so in general
$\gamma_s(T') \notin\SS_*^b$.
To overcome this problem we define another deformation $\beta_t$.
This deformation does not change topological entropy, and
only moves (certain) plateaus which are mapped into other plateaus.

As before, let 
$$
\begin{array}{rl}  W^o(T) = 
&\left\{x; \, \exists j\in \{1,\dots,b\} \mbox{ and }k\ge 0\mbox{ such that }\right. \\
&\left.  \quad T^k(x)\in \interior(Z_j) \mbox{ and }x,T(x),\dots,T^{k-1}\notin Z_j\right\} .\end{array}
$$ 
Let $\I(T)$ be the  collection
of integers $i\in \{1,\dots,b\}$ such that $T(Z_i)\in W^o(T)$ and $Z_i$ is not periodic.
Next we say that $T$ has a {\em local maximum (resp. minimum)}
at $Z_i$ if the sawtooth map $S_0$ has a local maximum (resp. minimum) at the midpoint of $Z_i$.
Moreover, given an interval $J$, we define  $\partial _l J$ and $\partial_rJ$ to be the left and right endpoint of $J$ respectively. 
For $i\in \I(T)$ define $v_i=T(Z_i)$, let $W_i=W_i(T)$ be the component of $W^o(T)$ containing $v_i$ and 
\begin{equation}\label{eq:rho}
\tau_i(T)= \left \{ \begin{array}{rl} 
\dfrac{\dist(v_i,\partial_l W_i)}{|W_i|} & \mbox{ if } T \mbox{ has a maximum at $Z_i$,} \\[4mm]
\dfrac{\dist(v_i,\partial_r W_i)}{|W_i|} & \mbox{ if } T \mbox{ has a minimum at $Z_i$.} 
 \end{array}\right.
\end{equation}
We define the deformation $T_t=\beta_t(T)$ with parameters $(\zeta_{1,t},\dots,\zeta_{b,t})$ 
of $T$ as the flow of the following differential equation:
$$
\dfrac{d\zeta_{i,t}}{dt} = \left\{ 
\begin{array}{ll}
4\cdot \tau_i(T_t)\cdot e  & \mbox{ if }i\in \I(T_t), \\
0 & \mbox{ if }i\notin \I(T_t), 
\end{array}\right. \quad  i=1,\dots,b.
$$
Note that if $Z_i$ contains a periodic attractor or $Z_i$ is never mapped
into another plateau, then $Z_{i,t}$ does not vary with $t$. As $W^o(T)$
consists of preimages of the interiors of such plateaus, it follows that 
$W^o(T_t)$ is independent of $t$ except if some plateau $Z_{i,t}$ is mapped 
by $T_t$ into the boundary of a component $W$ of $W^o(T_t)$ as
in the right panel of Figure~\ref{fig:begaclaim2a}. Also note that $W^o(T)$ depends 
continuously on $T$ except if some plateau $Z_i$ maps
into the boundary of a component of $W^o(T)$, \ie  eventually maps to the boundary
of another plateau.  In fact only if it is only discontinuous if $Z_i$
{\lq}arches over{\rq} this component as in the right panel of Figure~\ref{fig:begaclaim2a}. 

Even though $W_i$ does not necessarily depend continuously on $T$,
the deformation $(T,t)\mapsto \beta_t(T)$ turns out to be 
well-defined and continuous, see the lemma below.
Since $T_t$ and $T$ agree  outside the union of the closures of components of $W^o(T)$, we have $h_{top}(\beta_t(T)) = h_{top}(T)$.

\begin{lemma}\label{le:betat}
The deformation $(T,t)\mapsto \beta_t(T)$ is well-defined. Furthermore:
 \begin{itemize}
 \item[(a)] $(T,t) \mapsto \beta_t(T)$ is continuous. 
\item[(b)] 
The map $\beta_1(T)$ satisfies the $\beta$-property (\ref{eq:nointerior}).
%
\end{itemize}
\end{lemma}

\begin{proof}
{\bf Part (a): Well-defined and continuity.} If $i \in \I(T)$ and plateaus 
are mapped into the interior of components of $W_i(T)$, 
then $T\mapsto \tau_i(T)$ is locally Lipschitz
and there is a unique solution. Otherwise, if $i\notin \I(T)$ the following considerations show
that $t\mapsto \beta_t(T)$ is well-defined. Note that
it suffices to consider $t\approx 0$ and maps $\tilde T$ near some $T\in \SS^b$. 
Let $\tilde \zeta_{i,t}$ denote the $i$-th parameter of $\beta_t(\tilde T)$. 
If $T(Z_i)$ is contained in a component of $W^o(T)$, then continuity of $(t,\tilde T) \mapsto \tilde \zeta_{i,t}$ obviously holds.
Next assume that $v_i:=T(Z_i)$ is contained in the boundary of a component 
$W$ of $W^o(T)$. 
By definition this means that $t\mapsto \zeta_{i,t}$ is constant for $t \ge 0$. 
Moreover,  for a nearby map $\tilde T$, there are two possibilities:
(i)  $\tilde T(\tilde Z_i)$ is not contained in $W^o(\tilde T)$, which means that 
$t\mapsto \tilde  \zeta_{i,t}$ is also constant, or 
(ii)   $\tilde T(\tilde Z_i)$ is contained in a component $\tilde W$ of $W^o(\tilde T)$.
In this case, since $\tilde T$ is close to $T$, either 
(iia) $\tilde T(\tilde Z_i)$ is near the boundary of a 
component $\tilde W$ of $W^o(\tilde T)$ or 
(iib) there exists  $k$ so that  $T^k(v_i)\in \partial Z_j$ and 
$T^{k+1}(v_i)\in \partial W'$ where $W'$ is a component
of $W^o(T)$.  
In the latter case  $\tilde T(\tilde Z_j)$ is near 
the boundary of a component $\tilde W'$ of $W^o(\tilde T)$ and $\tau_j(\tilde T)\approx 1$. 
These two situations are illustrated in Figures~\ref{fig:begaclaim2a}
and \ref{fig:begaclaim2b}. 

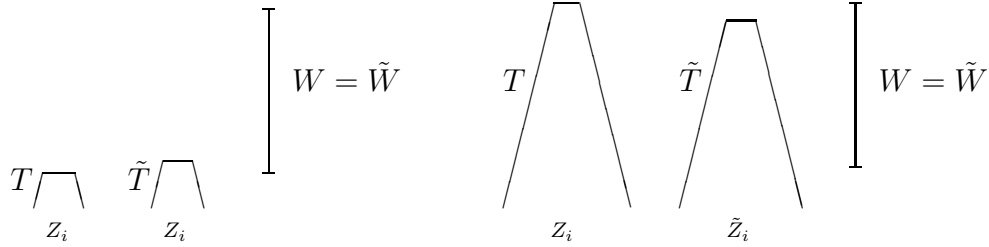
\begin{figure}[h]
\begin{center}
\unitlength=7.8mm
\begin{picture}(20,4.7)(-2.7,0.7)
\thinlines 
\put(1,1){\line(1,4){0,2}}
\put(1.2,1,8){\line(1,0){0.5}}
\put(1.7,1.8){\line(1,-4){0,2}}
\put(3,1.6){\line(0,1){2.8}}
\put(2.9,1.6){\line(1,0){0.2}}\put(2.9,4.4){\line(1,0){0.2}}
\put(0.6, 1.3){$\tilde T$}
\put(1.2,0.5){\tiny $Z_i$}
\put(3.4, 3){$W = \tilde W$}
\put(-1,1){\line(1,4){0.15}}
\put(-0.85,1,6){\line(1,0){0.55}}
\put(-0.3,1.6){\line(1,-4){0.15}}
\put(-0.8,0.5){\tiny $Z_i$}
\put(-1.4, 1.3){$T$}
\put(7,1){\line(1,4){0.87}}
\put(7.86,4.5){\line(1,0){0.42}}
\put(8.3,4.5){\line(1,-4){0.87}}
\put(7.8,0.5){\tiny $Z_i$}
\put(7, 3){$T$}
\put(10,1){\line(1,4){0.8}}
\put(10.8,4.2){\line(1,0){0.5}}
\put(11.3,4.2){\line(1,-4){0.8}}
\put(13,1.7){\line(0,1){2.8}}
\put(12.9,1.7){\line(1,0){0.2}}\put(12.9,4.5){\line(1,0){0.2}}
\put(10.8,0.5){\tiny $\tilde Z_i$}
\put(10, 3){$\tilde T$}
\put(13.4, 3){$W = \tilde W$}
\end{picture}
\end{center}
\caption{\label{fig:begaclaim2a} Case (iia) in the proof of continuity. On the left the situation
where $\tilde T(\tilde Z_i)$ is contained in a component $W$ of $W^o(T)$ 
with $\tau_i(\tilde T) \approx 0$. In this case, $\tilde \zeta_{i,t}$ 
will increase only very slowly with $t$.
On the right the situation where $\tilde T(\tilde Z_i)$ is close to the 
boundary of an endpoint of $W^o(T)$
when $\tau_i(\tilde T)\approx 1$. In this case $\tilde \zeta_{i,t}$ 
increases with speed $\approx 4e$ 
which means that there exists $t>0$ close to $0$ so that $\tilde T_t(\tilde Z_i)$ is contained in the boundary  of $W$ and then stops.
In this figure we set $W = \tilde W$, but also when these are different intervals, the same argument holds. Note that in the situation on the right, the component of $W^o(T)$ containing $Z_i$ does not depend continuously on $T$.}
\end{figure}

If (iia) holds, then either
$\tau_i(\tilde T)\approx 0$ which means that  $\tilde \zeta_{i,t}$ remains close to $\tilde \zeta_{0,i}$ for all $t\in [0,1]$ 
(see the left panel in Figure~\ref{fig:begaclaim2a})
or $\tau_i(\tilde T)\approx 1$ which means that $\tilde \zeta_{i,t}$ is moving with speed $\approx 4e$
towards the nearest boundary point of $\tilde W$ (see the right panel in Figure~\ref{fig:begaclaim2a}). 
Therefore there exists $\tilde t>0$ close to zero, so that 
$\tilde T_{\tilde t}(\tilde Z_i)\in \partial \tilde W$  and therefore $\tilde \zeta_{i,t}$ remains constant for $t\ge \tilde t$.

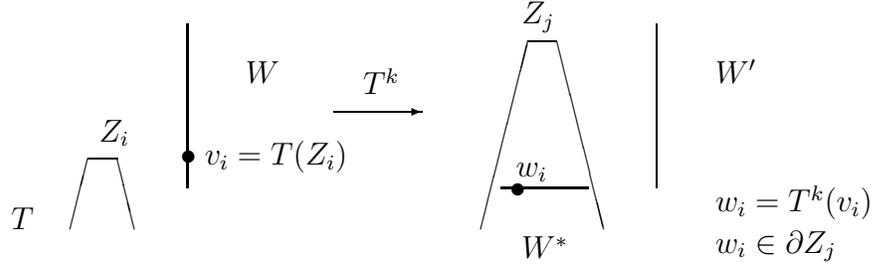
\begin{figure}[h]
\begin{center}
\unitlength=7.8mm
\begin{picture}(20,4.7)(-3,0.7)
\thinlines 
\put(1,1){\line(1,4){0,3}}
\put(1.3,2.2){\line(1,0){0.5}}
\put(1.8,2.2){\line(1,-4){0,3}}
\put(3,1.7){\line(0,1){2.8}}
\put(0, 1){$T$}
\put(1.5,2.5){$Z_i$}
\put(4, 3.5){$W$}
\put(2.88,2.1){$\bullet$}
\put(3.3,2.1){$v_i=T(Z_i)$}
\put(5.5,3){\vector(1,0){1.5}}
\put(6,3.3){$T^k$}
\put(8,1){\line(1,4){0.8}}
\put(8.8,4.2){\line(1,0){0.5}}
\put(9.3,4.2){\line(1,-4){0.8}}
\put(8.7,4.5){$Z_j$}
\put(8.7,0.5){$W^*$}
\put(8.34,1.7){\line(1,0){1.5}}
\put(8.5,1.55){$\bullet$}
\put(8.6,1.9){$w_i$}
\put(12,1.3){$w_i=T^k(v_i)$}
\put(12,0.6){$w_i\in \partial Z_j$}
\put(11,1.7){\line(0,1){2.8}}
\put(12, 3.5){$W'$}
\end{picture}
\end{center}
\caption{\label{fig:begaclaim2b} Case (iib) in the proof of continuity. In this case $W$ is a component of $T^{-k}(W^*)$ where $W^*$ is a component 
of $W^\circ(T)$ containing a plateau $Z_j$ 
and $T(Z_j)$ is close to the boundary of the component $W'$ 
containing $T(Z_j)$.
There exists $t>0$ close to zero, so that  $T_t(Z_{j,t})$ hits the 
boundary of $W'$  and 
then $W^*_t$ splits into three components: to the left and right of $Z_{j,t}$ and the interior of $Z_{j,t}$. 
This means that $W_t$ also splits into three components, and at this moment $Z_{i,t}$ moves
until $T_t(Z_{i,t})$ belongs to the boundary of one of these plateaus (which in this situation means close to $v_i$).}
\end{figure}
If (iib) holds, see Figure~\ref{fig:begaclaim2b}, 
then there exists $\tilde t>0$ small 
so that $\tilde T_{\tilde t}(\tilde Z_{j,t})\in \partial \tilde W'$ and  $\tilde W^*$ no longer
is a component of $W^o(\tilde T_{\tilde t})$ (the interval $\tilde W^*$ is split into 
three components of $W^o(\tilde T_{\tilde t})$).
This means that  $\tilde T_{\tilde t}(Z_i)$ is near a boundary of 
a component $\tilde W_-$ of $W^o(\tilde T_{\tilde t})$ and we can argue as before.

Finally, assume that $T(Z_i)$ is not contained in the closure of a component of $W^o(T)$.
If $\tilde T$ is near $T$ while $\tilde T(\tilde Z_i)$ is contained in a component $\tilde W$
of $W^o(\tilde T)$, then this component $\tilde W$ is small and so 
for $t>0$ small $\beta_t(\tilde T)(\tilde Z_i)$ is contained in the boundary of $\tilde W$
(and for $t'>t$ this plateau no longer moves). 

It follows that in all cases, $(T,t)\mapsto \beta_t(T)$ is continuous.

{\bf Part (b).}  If $Z_i$ is non-periodic and $\J_i$ is mapped into a component  $W$
of $W^o(T)$, then 
either $\tau_i(T)\ge 1/2$ or $\tau_{i+1}(T)\ge 1/2$, see Figure~\ref{fig:begat}.
By construction, $\tau_i(\beta_t(T))\ge 1/2$ or $\tau_{i+1}(\beta_t(T))\ge 1/2$
for all $t\ge 0$. Hence $\frac{d\zeta_i}{dt}(\beta_t(T)) \ge 2e$ or 
$\frac{d\zeta_{i+1}}{dt}(\beta_t(T))\ge 2e$,
for all $t\in [0,t_0]$ where $t_0$ is chosen so that $\beta_{t}(T)(\J_{i,t})$ is still contained in $W$
for all $t\in [0,t_0]$.
Since for each map $T\in \SS^b$, the corresponding $\zeta_i$ can be at most $2e$, 
there exists $t\le 1$ so that  $\beta_{t}(T)(\J_{i,t})\in \partial W$. Note that $x\in \partial W$ 
implies that no iterate of $x$ can be mapped into the interior of a plateau. 
Hence the claim follows. 
\end{proof}

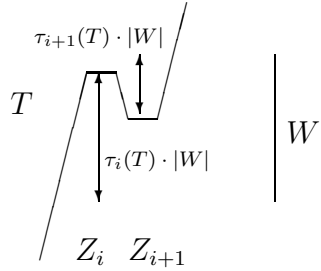
\begin{figure}[h]
\begin{center}
\unitlength=7.8mm
\begin{picture}(20,4.7)(-3,0.7)
\thinlines 
\put(1,1){\line(1,4){0.8}}
\put(1.8,4.2){\line(1,0){0.5}}
\put(2.3,4.2){\line(1,-4){0.2}}
\put(2.5,3.4){\line(1,0){0.5}}
\put(3,3.4){\line(1,4){0.5}}
\put(5,2){\line(0,1){2.5}}
\put(2,2){\vector(0,1){2.2}}\put(2,4.2){\vector(0,-1){2.2}}
\put(2.1,2.6){\tiny $\tau_i(T) \cdot |W|$}
\put(2.7,4.5){\vector(0,-1){1}}
\put(2.7,3.5){\vector(0,1){1}}\put(0.9,4.7){\tiny $\tau_{i+1}(T) \cdot |W|$}
\put(1.6,1){$Z_i$}
\put(2.5,1){$Z_{i+1}$}
\put(0.5, 3.5){$T$}
\put(5.2, 3){$W$}
\end{picture}
\end{center}
\caption{\label{fig:begat}  In this example, $T(Z_i)$ and $T(Z_{i+1})$ are both 
contained in the same component $W$ of $W^o(T)$.
We show a situation where $\tau_i(T)\approx 1$ 
whereas $\tau_{i+1}(T)$
is much smaller. As a result $\zeta_{i,t}$ changes much faster than $\zeta_{i+1,t}$ under the deformation, but as soon as
$\beta_t(T)(Z_{i,t})$ hits the boundary of $W$ then $\zeta_{i,t}$ stops moving.
The component of $W^o(T_t)$ containing $Z_i$ splits in three as soon as
$T_t(Z_{i,t})$ hits the boundary of $W$: one to the left of $Z_i$, one to the right of $Z_{i}$ and the interior of $Z_i$.}
\end{figure}

\subsection{\boldmath Increasing entropy of maps more carefully:
$\Eih_t$. \unboldmath}\label{subsec:eih}

We use the entropy increasing deformation $(T,t) \mapsto \Ei_t(T) = \ei_{t \cdot t_{max}(T)}(T)$ (with $t_{max}$ as in \eqref{eq:tmax} on page \pageref{eq:tmax})
until
$h_{top}(\Ei_t(T)) = h_0$. But it is possible that only part of the
phase space is responsible for reaching this entropy bound, while in
other parts (namely in renormalization cycles),
plateaus have not been lifted \lq sufficiently\rq yet.
It is essential to perform some version of $\ei_s$ for at least some time $s > 0$, so as to
resolve (destroy) wandering pairs that may have been created by $\ed_t$.
Thus in the presence of renormalization intervals, we may need to lift some
plateaus faster than others. This subsection explains how this is done.

\begin{remark}\label{Rem:bifurK} 
Recall that $T\in \SS^b$ has renormalization interval $K$ if
there is $n \ge 1$ such that
$T^n(K)\subset K$ and $T^n(\partial K)\subset  \partial K$.
The orbit $\orb(K)=K\cup T(K)\cup \dots \cup T^{n-1}(K)$ is called
a \emph{renormalization cycle}.
Note that $\partial K$ consists of (pre)periodic points which do not
depend on $T$, unless they disappear in a saddle node bifurcation.
The only other way by which $K$ can disappear is when $T^n(Z) \subset \partial
K$ for some plateau $Z$ compactly contained in $K$.
\end{remark}

Fix $h_0 \in (0, \log(b+1))$. 
Let $K_i(T)$ be the smallest renormalization
interval for $T$ which contains the omega-limit set
$\omega(Z_i)$. If there exists no smallest renormalization interval
then we take $K_i(T)=\emptyset$.
Let
$$
\left\{ \begin{array}{l}
L_i(\lh0) = \mbox{closure}(\{ T \in \SS^b \st h_{top}(T|\orb(K_i(T))) \le h_0 \}), \\[3mm]
L_i(\gh0) = \mbox{closure}( \{ T \in \SS^b \st h_{top}(T|\orb(K_i(T))) > h_0\}).
\end{array} \right.
$$
Note that in this definition $\orb(K_i(T))$ does indeed depend on $T$.
Consequently, the common boundary of these sets contains maps $T$ for which 
$h_{top}(T|\orb(K_i(T))) \le h_0$ but which can be increased by an
arbitrarily small change in parameter $\zeta_j$ for some $Z_j \subset
\orb(K_i(T))$.
It does not follow, however, that
$h_{top}(T|\orb(K_i(T)) = h_0$, see Figure~\ref{fig:noth0}. In fact,
$T \mapsto h_{top}(T|\orb(K_i(T))$ is discontinuous because
if $Z_i$ is not in the basin of an attractor, 
a small change in $T$ can make $Z_i$ periodic itself, and 
then $h_{top}(T|\orb(K_i(T)) = 0$.

\begin{figure}[h!]
\begin{center}
\unitlength=1mm
\begin{picture}(30,30)(10,-15)
\thinlines
\put(-15,-15){\line(1,0){30}}
\put(-15,-15){\line(0,1){30}}
\put(15,15){\line(-1,0){30}}
\put(15,15){\line(0,-1){30}}
\put(-15,-15){\line(1,1){30}}
\put(0,0){\line(1,0){12}}
\put(0,0){\line(0,1){12}}
\put(0,12){\line(1,0){12}}
\put(12,0){\line(0,1){12}}
\put(2,-3){\tiny $K_2(T)$}
\thicklines
\put(-15,15){\line(1,-5){6}}
\put(-3,-15){\line(1,5){5.4}}
\put(15,-15){\line(-1,5){5.4}}
\put(-9,-15){\line(1,0){6}}
\put(2.52,12){\line(1,0){7}}
\put(-8,-18){\tiny $Z_1$}\put(4,-18){\tiny $Z_2$}
\put(-12,7){\small $T$}

\put(25,9){$T \in L_2(\lh0)\cap L_2(\gh0)$ for} 
\put(25,3){$h_0 = \log(1+\sqrt2)$, but}
\put(25,-3){$h_{top}(T|\orb(K_2(T)) = \log 2 \ne h_0$.}
\end{picture}
\end{center}
\caption{\label{fig:noth0} For this $T$ with  $h_{top}(T) = h_0 = \log(1+\sqrt{2})$ we have
$h_{top}(T|\orb(K_2(T))) = \log 2 < h_0$, but lifting $Z_2$ by any amount pushes $h_{top}(T) = h_{top}(T|\orb(K_2(T)))$ above $h_0$.}
\end{figure}

If the period of $K_i(T)$ is $m$, then 
since the intervals $K_i(T),\dots, T^{m-1}(K_i(T))$ have disjoint interiors,
the first return map of $\orb(K_i(T))$ has at most
$2^b$ branches. It follows that if $h_{top}(T)=h_0$ then
\begin{equation}\label{eq:upperm}
0 < h_0 = h_{top}(T|\orb(K)) \le (b \log 2)/m
\end{equation}
which gives the upper bound $m \le (b \log 2)/h_0$.

Define
$$
\Phi_j(T)=\mbox{dist}(T,L_j(\lh0)\cap L_j(\gh0))
$$
where $\dist$ is as in \eqref{eq:dist} on page \pageref{eq:dist}.
Define the following deformation of maps $T\in \SS^b$:
\begin{equation}\label{eq:eih}
\Eih_t\big(\zeta_1,\dots,\zeta_b) = (\min(\zeta_1+\Phi_1(T) t,e),
\dots,\min(\zeta_b+\Phi_b(T) t,e)\big).
\end{equation}

Let us say that $T\in \SS^b_{*,j}$ if whenever $\J=[Z_{j-1},Z_j]$ or $\J=[Z_j,Z_{j+1}]$ (assuming $1\le j-1$ and $j+1\le b$ respectively)
is eventually mapped into a plateau, then $\J$ is contained in the closure of a component
of the basin of $T$.

\begin{lemma}\label{lem:propU}
Let $\Eih_t$ be as above and
let $T\in \SS^b$ be such that $h_{top}(T)\le h_0$.
Then the following hold:
\begin{enumerate}
\item The deformation $(T,t)\mapsto \Eih_t(T)$
is continuous in $T$ and $t$, and $t \mapsto h_{top}(\Eih_t(T))$ is non-decreasing.
\item $h_{top}(\Eih_t(T))\le h_0$ for all $0\le t\le 1$.
\item  Assume that $T$ satisfies 
\begin{enumerate}
\item  the $\beta$-property \eqref{eq:nointerior} on page \pageref{eq:nointerior}  and
\item for any $j=1,\dots,b$,
\begin{equation}
T\in L_j(\lh0) \cap L_j(\gh0)\mbox{ implies }T\in \SS^b_{*,j}.
\label{eq:s*j}
\end{equation}
\end{enumerate}
Then
$\Eih_t (T)\in \SS_*^b$ for each $t>0$.
\end{enumerate}
\end{lemma}

\begin{proof}
The continuity and monotonicity of statement (1) are obvious.

For statement (2), take $t \in [0,1)$ and $T \in \SS^b$ with $h_{top}(T) \le h_0$.
Let $j_1$ be such that $m_1 := \Phi_{j_1}(T)$ is maximal among
$\{ \Phi_1(T), \dots, \Phi_b(T)\}$, and let $M_1$ be the
open $b$-dimensional $m_1$-cube centered at $T$, parallel to the
coordinate hyperplanes and of side length $2m_1$.
Then $M_1$ is disjoint from
$L_{j_1}(\lh0) \cap L_{j_1}(\gh0)$, and in particular disjoint from
$L_{j_1}(\gh0)$.
Therefore  $h_{top}(T'|\orb(K_{j_1}(T'))) \le h_0$ for any $T' \in M_1$
and in particular for $\Eih_t(T)$.

Now let $j_2$ be such that  $m_2 := \Phi_{j_2}(T)$ is second largest among
$\{ \Phi_1(T), \dots, \Phi_b(T)\}$.
The corresponding $m_2$-cuboid $M_2$
is the set of $T'$ with parameters $\{ \zeta_1, \dots, \zeta_b \}$
such that $|\zeta_j - \zeta_j(T) | < m_2$ for all $j \neq j_1$ and
$|\zeta_{j_1} - \zeta_{j_1}(T) | < m_1$.
(This is the Cartesian product of a $b-1$-dimensional cube and an arc of
length $2m_1$ in the $\zeta_{j_1}$-direction.)

\noindent
{\bf Claim 1:}  $M_2$ is disjoint  from $L_{j_2}(\gh0)$. To prove  this claim,  consider $T' \in M_2$.
Let $M_2''$  be the $m_2$-cube centered at $T$, and choose 
$T''\in M_2''$ be so that $T'$ and  $T''$ agree except at $Z_{j_1}$.
Let $T_t$, $t\in [0,1]$ be the one-parameter family of maps connecting $T''$ to $T'$ 
corresponding to maps for which the parameter $\zeta_{t,j_1}$ varies linearly  
and so that  $\zeta_{t,i}=\zeta_{0,i}$ for all $t\in [0,1]$ and $i \ne j_1$.
Denote $K_{j_2,t}^*:=\orb(K_{j_2}(T_t))$ and $K_{j_1,t}^*:=\orb(K_{j_1}(T_t))$.
By definition  of $m_1$ and $m_2$, we have 
$h_{top}(T_t|K^*_{j_1, t})\le h_0$ for each $t\in [0,1]$
and $h_{top}(T_0|K^*_{j_2, 0})\le h_0$. 

Next let 
$$
X_1=\{t\in [0,1]; K^*_{j_2,t}\subseteqq K^*_{j_1,t}\}\ ,\
X_2=\{t\in [0,1]; K^*_{j_2,t}\supsetneqq K^*_{j_1,t}\}
$$
and 
$$
X_3=\{t\in [0,1]; K^*_{j_2,t}\cap K^*_{j_1,t}=\emptyset\}.
$$
From the above properties we obtain 
$h_{top}(T_t|K^*_{j_2,t})\le h_{top}(T_t|K^*_{j_1,t})\le h_0$ 
for each $t\in X_1$. 
Consider a component $C$ of $X_3$ or of $X_2$.  If there exists $t\in C$ so that
$h_{top}(T_t|K^*_{j_2,t})\le h_0$ then 
$h_{top}(T_s|K^*_{j_2,s})\le h_0$ for each $s\in C$; here we use that 
$h_{top}(T_t|K^*_{j_1,t})\le h_0$ for each $t\in [0,1]$ if $C$ is a component
of $X_2$.  

Claim 1 follows if there exists a component of $X_1$, $X_2$ or 
$X_3$ containing both $0$ and $1$.
So from now on we assume that this is not the case, and we can assume that
$t'$ in the next claim is not $0$ or $1$.

\noindent 
{\bf Subclaim 1:} for each boundary point $t'$ of $C$
there exists $t_n\in X_1$ so that $t_n\to t'$. In fact, 
$K_{j_1,t_n}\supset K_{j_1,t_n}=K_{j_1,t'}$.

To prove this, let $\tilde K_{j_1,t'}$ be the
{\em maximal} renormalization interval containing $K_{j_1,t'}$ which is either contained in 
$K_{j_2,t'}$ (if $t'\in X_2$) or which is disjoint form $K_{j_2,t'}$ (if $t'\in X_3$). 
Observe that the maximality of $\tilde K_{j_1,t'}$ implies that there are backwards iterates of
$K_{j_2,t'}$ accumulating to the boundary points of $\tilde K_{j_1,t'}$.
That $t'$ is a boundary point of $C$ means that some iterate of $Z_{j_1,t'}$
under $T_{t'}$ is mapped into the boundary of a component of $\tilde K_{j_1,t'}$. 
Because of the above observation this implies that there exists a sequence of points $x_n$ converging to either boundary point  $\tilde K_{j_1,t'}$ such that 
the omega-limit set $\omega(x_n) \subset K_{j_2,t'}^*$. One can choose 
$x_n$ even so that  $K_{j_2,t'}$ is the smallest renormalization interval 
whose orbit contains $\omega(x_n)$. 
It follows that there exists a sequence $t_n\to t'$ with 
$x_n \in \orb(Z_{j_1, t_n})$ so that
$K_{j_1,t_n}\supseteqq K_{j_2,t_n} = K_{j_2,t'}$.
This completes the proof of Subclaim 1.

{\bf Proof of Claim 1 continued.}
Let us say that a renormalization $K$ of $T_t$ is {\em created} at parameter $t$,
if $K$ is not a renormalization interval for $T_s$ for all $s\in (t-\delta,t)$ with $\delta>0$ small. Since $\zeta_{j_1}$ is the only parameter moving, 
this implies that 
\begin{equation}\label{eq:remark}
Z_{j_1,t} = K \text{ has a one-sided attracting periodic boundary point.}
\end{equation}
Again, since $\zeta_{j_1}$ is the only parameter moving, $Z_{j_1,t}$ cannot 
be part of block of plateaus if $K$ is created.

First consider the case that  $\zeta_{j_1,1}>\zeta_{j_1,0}$. 
If $t\in (X_2\cup X_3)$, then $Z_{j_2,t}$ is never mapped into 
$Z_{j_1,t}$.  It follows that for each $s\in [t,1]$
no iterate under $T_s$ of $Z_{j_2,s}=Z_{j_2,t}$ is mapped into $Z_{j_1,s}$.
Hence $\omega(Z_{j_2,s})$ remains the same for each $s\in [t',1]$ where we can
take $t^*:=\inf(X_2\cup X_3)$.  
If $K$ is a renormalization interval for $T_{s'}$ which is created at 
parameter $s'$ and $K$ intersects $\omega(Z_{j_2, s'})$, 
then by remark \eqref{eq:remark} 
it follows that $\omega(Z_{j_2,s'})$ would have to intersect $Z_{j_1,s'}$,
which is impossible when $s\ge t'$. Hence, if the 
renormalization interval $K_{j_2,s}$ 
changes at some $s\in [t^*,1]$, then it
is replaced by some larger renormalization interval $K_{j_2,s'}$ for $T_{s'}$ 
with $s'>s$ close to $s$. Thus the set $K_{j_2,s}$ only gets larger 
as $s$ increases from $t^*$ to $1$.
Let $C$ be a component as before and let $t'$ be an endpoint 
of $C$ (with $t'\ne 0,1$). By Subclaim 1 there exists $t_n\to t'$ with $t_n\in X_1$
and so $K_{j_1,t_n}\supset K_{j_1,t_n}=K_{j_1,t'}$. This, together with 
the fact that $h_{top}(T_{t_n}|K^*_{j_1, t_n})\le h_0$, implies that 
$h_{top}(T_{t'}|K^*_{j_2, t'})\le h_0$.
Using the first part of the proof (above Subclaim 1), Claim 1 follows.

Now consider the case that $\zeta_{j_1,1} < \zeta_{j_1,0}$. In this case, 
$K_{j_2,s}$ can shrink as $s$ increases, namely when $Z_{j_1,s}$ is contained in 
$K^*_{j_2,s}$ and the image of $Z_{j_1,s}$ is mapped
to a boundary point of $K^*_{j_2,s}$. Since in this case $K^*_{j_2,s}$ can only 
shrink as $s$ increases, and $Z_{j_2,s}$ only gets wider as $s$ increases, 
it follows that  $h_{top}(T_{t'}|K^*_{j_2, s})$ is a decreasing function of $s$ 
in this case. 
The set $K^*_{j_2,s}$ can also change if at some parameter $s'$ it is
{\em destroyed}
(\ie created in reverse direction), but remark \eqref{eq:remark} 
implies that some iterate of $Z_{j_2, s'}$ is contained in $Z_{j_1, s'}$,
so $\omega(Z_{j_2, s'}) = \omega(Z_{j_1, s'})$ and we are in the parameter 
set $X_1$ and in particular 
$h_{top}(T_{s''}|K_{j_2, s''}^*)=h_{top}(T_{s''}|K_{j_1, s''}^*) \le h_0$
for all $s'' \ge s'$.
Thus Claim 1 holds also if $\zeta_{j_1,1} < \zeta_{j_1,0}$.

\medskip

Continuing inductively, we see that if $T' \in \cap_{k=1}^b M_k$,
then $h_{top}(T'|\orb(K_j(T'))) \le h_0$
for each $j$, and this holds in particular for $\Eih_t(T)$ for each $t\in [0,1]$.
If $K_j(\Eih_t(T)) = [-e,e]$ for some $j$ (\ie  if $\hat \Gamma_t(T)$
has no renormalization interval), then this proves statement (2).

If, however, every plateau belongs to some renormalization cycle and
the entropy of $\Eih_t(T)$ is carried by the Cantor set of points that
never enter these renormalization cycles, then we argue as follows.
Write $T_t=\Eih_t(T)$ and  assume by contradiction that $h_{top}(T_{t_1}) > h_0$
for some $t_1 \in [0,1)$. Take $t_0 \in [0,t_1)$ maximal such that $h_{top}(T_{t_0}) \le h_0$.
Since $T_t\in \cap_{k=1}^b M_k$ for each $t \in [0,1)$,
the first part of the proof gives
\begin{equation}
h_{top}(T_t| \orb(K_j(T_t)) \le h_0\mbox{ for each }t\in [0,1]\mbox{ and each }j=1,\dots,b.
\label{eq:lem77}
\end{equation}
Now we need the following

\noindent
{\bf Claim 2:} There exists $j\in \{1,\dots,b\}$ and $t'\in (t_0,t_1)$ so that 
$h_{top}(T_{t'}|\orb(K_j(T_{t'})))>h_0$ for the minimal interval of renormalization
$K_j(T_{t'})$ containing $\omega(Z_{j,T_{t'}})$.
\\[3mm]
{\bf Proof of Claim 2:} For each $t\in (t_0,t_1)$ consider the semi-conjugacy of $T_t$
with the map $F_t$  with slope $\pm \exp(h_{top}(T_t))$ as above Lemma~\ref{lem:decrentr}.
Since $h_{top}(F_t) = h_{top}(T_t)$ depends continuously on $t$, and
is not constant on $(t_0, t_1)$, $h_{top}(F_t)$ assumes uncountably many 
values as $t$ moves through $(t_0, t_1)$.
But there are only countably many slopes for which all turning points of $F_t$ are periodic.
Therefore there exists $t'\in (t_0,t_1)$ so that at least one of the turning points of $F_{t'}$, say $c_j$, is non-periodic.
Let $X_{j,t'}$ be the smallest renormalization interval of $F_{t'}$
containing this turning point. Since $F_{t'}$ has constant slope,
$h_{top}(F_{t'}|\orb (X_{j,t'})) = h_{top}(F_{t'}) > h_0$. Since the $j$-th critical
point of $F_{t'}$ is not periodic, the smallest renormalization $K_j(T_{t'})$ of $T_{t'}$ containing
$\omega(Z_{j,t'})$ is mapped by the semi-conjugacy onto $X_{j,t'}$. It follows that
$h_{top}(T_{t'}|\orb(K_j(T_{t'})))>h_0$, completing the proof of the claim.

Obviously the claim contradicts \eqref{eq:lem77}, so we can conclude that $h_{top}(T_1)\le h_0$.
Statement (2) follows.

Finally we turn statement to statement (3). For this we need to show that 
$\Eih_t(T)\in \SS^b_{*,j}$ for each $j$ and each $t>0$. If  $T \in L_j(\lh0) \cap L_j(\gh0)$
then this holds by assumption \eqref{eq:s*j}.  So assume that $T\notin L_j(\lh0) \cap L_j(\gh0)$.
Then the $\beta$-property \eqref{eq:nointerior} implies that whenever $T^n([Z_j, Z_{j+1}])$ is contained in 
a plateau $Z_k$ then $T^n([Z_j, Z_{j+1}])\cap \partial Z_k\ne \emptyset$.
Since the plateau $Z_{j,t}$ shrinks at $t$ increases (as $T\notin L_j(\lh0) \cap L_j(\gh0)$),
this means that for $t>0$,  $(\Eih_t(T))^n$ maps $[Z_j, Z_{j+1}]$ at least partially 
outside $Z_k$.  So again $\Eih_t(T)\in \SS^b_{*,j}$. 
This proves statement (3).
\end{proof}
 
\def\zyk{\mbox{cycle}}
\subsection{\boldmath Decreasing the entropy more carefully: $\Edh_t$. \unboldmath }
\label{subsec:careful}

Take $T\in \SS_*^b$ with $h_{top}(T)=h_0 > 0$.
Even though $t\mapsto h_{top}(\hed_t(T))$
is non-increasing,  it is possible that for fixed $t > 0$,
$h_{top}(\ei_s \circ \hed_t(T)) > h_0$ for all $s>0$.
The reason is that (in the notation of Section~\ref{subsec:deltat})
$\hsgn(Z_i)$ can change from $1$ to $-1$ (or vice versa)
during the deformation.
To explain what can happen, let us discuss two examples.

\begin{example} Consider the map 
from Figure~\ref{fig} on page \pageref{fig}.
Although $T\in \SS_*^b$, the map   $T'=\hed_t(T)$ has a wandering pair that
does not map into a periodic basin
(so $T'$ is no longer in $\SS_*^b$). There is a periodic interval $\hat K$
(here of period $1$) and $T'$ maps the convex hull $\J = [Z_1', Z_3']$
into $\partial \hat K$.
(Note that, in this example, $\hed_t(T)$ first decreases $\zeta_2$ and then, 
after the plateaus $Z_1$ and $Z_3$ touch $Z_2$, increases
it again. Even though $\ei_t$ initially is `the inverse' of the deformation
$\hed_t$, the map $T'=\hed_t(T)$ will have some touching plateaus.)
Because the entropy within the renormalization interval $\hat K$
is $\le h_0$, the movements of plateaus $Z_1, Z_2, Z_3$ under $\hed_t$ have no
effect on the global entropy.
Therefore  $h_{top}(T') = h_{top}(T)$,
whereas $h_{top}(\ei_s(T')) > h_{top}(T)$ for any $s>0$,
because $\ei_s$ decouples the plateaus $Z_1, Z_2, Z_3$ again and
$\ei_s(T')([Z_1, Z_3])$ is a closed neighbourhood of the left
endpoint of (and therefore not entirely contained in) $\hat K$.
The effect is that within $\Eih_s$,
the deformation $\ei_s$ will not be applied at all, and hence it
will not be able to remove the wandering pair created by $\hed_t$.
\end{example} 

\begin{example} In fact, a similar problem can occur even when we consider the family
$\beta_1\circ \ed_t(T)$. It is possible that for some $t\in (0,1)$
a pair of plateaus $[Z_{1,t},Z_{2,t}]$ of $\ed_t(T)$
is mapped by $\ed_t(T)$ into a plateau $Z_{3,t}$, which in turn is mapped into a renormalization
interval $\hat K_t$. As $t$ increases, the parameters $\zeta_t$ associated to $\beta_1\circ \ed_t(T)$ no longer decrease
with $t$ and a similar situation as in Figure~\ref{fig} can arise (but with $Z_{3,t}$ the image of $[Z_{1,t},Z_{2,t}]$
and $Z_{3,t}$ mapped into the boundary of a periodic plateau).
\end{example}

\begin{figure}[htp]
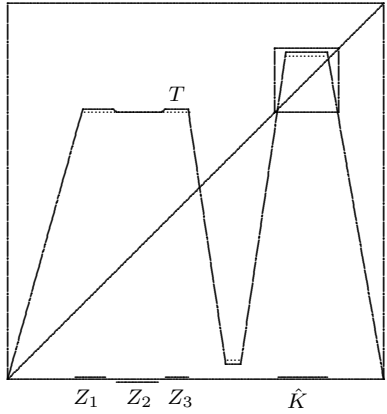
 \hfil
\beginpicture
\dimen0=0.5cm
\setcoordinatesystem units <\dimen0,\dimen0>
\setplotarea x from 0 to 10, y from 0 to 10
\setlinear
\plot 0 0 10 0 10 10 0 10 0 0 10 10 /
\plot 0 0  2 7.2
2.8 7.2  2.9 7.1 4.1 7.1 4.2 7.2 4.8 7.2  5.8 0.4
6.2 0.4  7.4 8.7 8.5 8.7 10 0  / %
\plot 7.1 7.1 8.8 7.1 8.8 8.8  7.1 8.8 7.1 7.1 /
\setdots <0.5mm>
\plot 2.05 7.1 4.9 7.1 /
\plot 7.4 8.6 8.4 8.6 /
\plot 5.85 0.5 6.2 0.5 /
\put {\tiny $T$} at 4.5 7.6
\setsolid
\plot 1.8 0.05 2.6 0.05 /
\plot 4.2 0.05 4.8 0.05 /
\plot 2.9 -0.07 4 -0.07 /
\plot 7.2 0.05 8.5 0.05 /
\put {\tiny $Z_1$} at  2.1 -0.5
\put {\tiny $Z_2$} at  3.5 -0.5
\put {\tiny $Z_3$} at  4.6 -0.5
\put {\tiny $\hat K$} at  7.7 -0.5
\endpicture
\caption{\label{fig}
{\small
The maps $T$ and $T' := \hed_t(T)$ (in dotted lines).  For $t>0$ small, $t\mapsto \hed_t(T)$
increases the height of the plateau $Z_2$ and decreases those of $Z_1,Z_3$.
Once they merge, this deformation decreases the height of all of them together.
In this example, the plateaus $Z_1$, $Z_2$ and $Z_3$ are mapped
into $\partial \hat K$, \ie $T'(Z_1)=T'(Z_2)=T'(Z_3) \in \partial \hat K$.
The map is constructed so that $T'|\hat K$ is unimodal with entropy
$ h_{top}(T'|\hat K) < h_{top}(T)$. In this case, $h_{top}(\ei_s(T')) > h_{top}(T)$ for any $s>0$, because points near 
$Z_2$ will then be 
mapped outside (\ie to the left of) $\hat K$.
Therefore $\hed_t(T)\in L_{2}(\lh0) \cap L_{2}(\gh0)$ where $h_0=h_{top}(T)$
and $\hed_t(T) \notin \SS_*^b$ unless the left boundary point of $\hat K$ is in the boundary
of a component of the basin of a periodic attractor.}}
\end{figure}

To overcome the issues caused by such examples, we introduce the deformation $\Edh_t$. 

{\bf Construction of $\Edh_t$:}
Let us define a modification $\Edh_t$ of the
deformation $\hed_t$, which allows some
of the plateaus (namely those within renormalization intervals
of `low entropy') to move before others.

Fix $h_0 \in (0, \log(b+1))$. 

\begin{definition}\label{def:zyk}
Given a periodic interval $K$ of period $m$,
we can find intervals $K_i \supset T^i(K)$ such that $T(\partial K_i) \subset \partial K_{i + 1 \bmod m}$ for all $0 \le i < m$.
For the minimal choice of such intervals $K_i$,
we write $\zyk(K) = \cup_{i=0}^{m-1} K_i$. 
Given a renormalization interval $K$ of $T$, we define
$$
\begin{array}{rl}\Omega(K)=\left\{z\in \zyk(K)\, ; \, \right.  & z,T(z),T^2(z),\dots \notin \hat K   \mbox{ for any }\\
& \left.  \mbox{ renormalization interval }\hat K\subsetneqq K  \right\}.\end{array}
$$
We say that $K$ has {\em entropy $h_0$} 
(for $T$) if $h_{top}(T|\Omega(K))=h_0$.
\end{definition}

It is possible that $h_{top}(T|\Omega(K))<h_0$ but that there exists a renormalization interval $\hat K\subsetneqq K$ so that $h_{top}(T|\Omega(\hat K))=h_{top}(T|\zyk(\hat K))=h_0$.

\medskip

\begin{definition}\label{def:arching}
Given an interval $L$, we say that a plateau $Z_j$ {\em  arches over } 
$\partial L$ if there exists $m$ so that $T^m(Z_j) \subset \partial L$ and 
if $k < m$ is maximal such that $T^k(Z_j)$ is contained in a plateau
 $Z$, then $T^{-(m-k)}(L)$ strictly contains the maximal block of touching plateaus containing $Z$, see Figure~\ref{fig:arching}.
\end{definition}

\begin{figure}[ht]
\begin{center}
\unitlength=1mm
\begin{picture}(30,30)(10,-15)
\thinlines
\put(-32,-11){\line(0,1){22}}
\put(-22,-15){\tiny $\J$}
\put(-35,0){\tiny $L$}
\put(-18,0){\tiny $T^m$}
\thicklines
\put(-25.5,11){\line(-1,-4){4.5}}
\put(-14.5,11){\line(1,-4){4.5}}
\put(-25.5,11){\line(1,0){11}}
\put(-22,10.5){\line(0,1){1}}
\put(-19,10.5){\line(0,1){1}}
\thinlines
\put(12,-11){\line(0,1){22}}
\put(23,-15){\tiny $\J$}
\put(28,0){\tiny $T^m$}
\put(9,0){\tiny $L$}
\thicklines
\put(13,-15){\line(1,4){1}}
\put(30.5,-11){\line(1,4){4.5}}
\put(14,-11){\line(1,0){16.5}}
\put(22,-11.5){\line(0,1){1}}
\thinlines
\put(58,-11){\line(0,1){22}}
\put(68,-15){\tiny $\J$}
\put(83,-12){\tiny $T^m$}
\put(55,0){\tiny $L$}
\thicklines
\put(58,-15){\line(1,4){1}}
\put(82,-15){\line(-1,4){1}}
\put(59,-11){\line(1,0){22}}
\put(67,-11.5){\line(0,1){1}}
\put(74,-11.5){\line(0,1){1}}
\end{picture}
\end{center}
\caption{\label{fig:arching} 
Three basic possibilities how a block of plateaus can map onto 
$\partial L$. In the situations described in the left two panels, 
the interval $\J$ arches over $\partial L$.}
\end{figure}
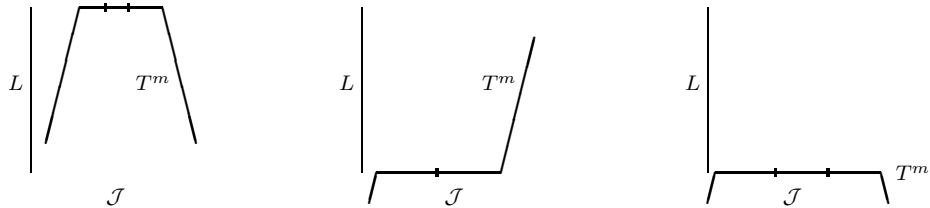

\begin{definition}\label{def:xi}
We say that $T\in \xi_i$ if there exists renormalization interval $K$ of entropy $h_0$
and a convex hull $\J \subset K$ of plateaus  which is 
{\bf non-trivial, \ie $T(\J)$ is not a singleton} and so that
\begin{enumerate}[topsep=-0.1cm,itemsep=0.05ex,leftmargin=0.8cm]
\item there exists a maximal renormalization interval $\hat K \subsetneqq K$ so that $T^m(\J)\subset \partial \hat K$ for some $m \ge 2$ (in particular
$T^m(\J)$ is a singleton); 
\item the first return map to $\hat K$ is non-monotone; 
\item  $T^j(\J)$
intersects $Z_i$ for some $0 \le j < m-1$;
\item $Z_i\cap \zyk(\hat K)=\emptyset$.
 \end{enumerate}
\end{definition}

It will be useful to stratify the space $\SS^b$, dividing the boundary of this space
 into subspaces according to which  plateaus touch. To do this, let 
 $\Xi$ be a partition $\{1,\dots,b\}$.
Then we define the {\em stratum} $\SS^b(\Xi)\subset \SS^b$ as follows:
$T\in \SS^b(\Xi)$  if and only if for each $i\in \{1,\dots,b-1\}$
the plateaus $Z_i$ and $Z_{i+1}$
touch whenever $i,i+1$ are in the same partition element of $\Xi$.
For example $\SS^b(\{1\},\{2\},\dots,\{b\})$ consists of all maps for 
which none of the plateaus touch, and $\SS^b(\{1,2\},\{3\},\dots,\{b\})$
consists of the space of maps for which the first two plateaus
touch (and no other two plateaus do). Note that when $\Xi_1,\Xi_2$ are two distinct partitions,
then $\SS^b(\Xi_1)$ and $\SS^b(\Xi_2)$ are disjoint. Clearly $\SS^b$ is
the disjoint union of $\SS^b(\Xi)$ where the union runs over all partitions $\Xi$ of 
$\{1,\dots,b\}$, and each $T\in \SS^b$ is associated to a partition $\Xi(T)$
of $\{1,\dots,b\}$. 

Note that in Definition~\ref{def:xi}, a non-trivial convex hull $\J$ 
always contains two plateaus $Z_i,Z_j$ where $i,j$ are in distinct
subsets from the partition $\Xi(T)$. Also note that property (2) Definition~\ref{def:xi}
in particular implies   
that the first return map to $\hat K$ is non-constant.

For each partition $\Xi$ of $\{1,\dots,b\}$,
choose a $C^\infty$ function $\rho_i^\Xi\colon \SS^b(\Xi)\to [0,1]$ which is zero 
on $\overline{\xi_i\cap \SS^b(\Xi)}$ and positive elsewhere.
Next define $\rho_i\colon \SS^b\to [0,1]$ by 
$\rho_i(T)=\rho_i^\Xi(T)$ whenever $T\in \SS^b(\Xi)$
for some partition $\Xi$ of $\{1,\dots,b\}$.

Define the modification $\Edh_t$ of $\hed_t$ as the flow of the 
differential equation
\begin{equation}\label{eq:diffeq}
\begin{array}{ll}
\dfrac{d\zeta_{i,t}}{dt} &=\left\{
\begin{array}{ll}
-\rho_i(T_t)
\cdot \hsgn(Z_{i,t})&  \mbox{ when } \zeta_{i,t}\in (-e,e),\\[2mm]
0   & \mbox{ otherwise,}\end{array} \right. \\[2mm]
\zeta_{i,t}|_{t=0}&= \quad \zeta_i(T). \end{array}
\end{equation}
Here $\hsgn$ is defined as in \eqref{eq:hsgn}
on page \pageref{eq:hsgn} and $T_t$ is the 
map corresponding to $\zeta_{i,t}$ and so $T_t=\Edh_t(T)$.
 
\begin{proposition}\label{prop:deltadeform}
For each $T\in \SS^b$, $\Edh_t(T)$ exists for all $t>0$ and 
$$\SS^b_*\times \R^+ \ni (T,s)\mapsto \Edh_s(T)$$ is continuous. 
Moreover: 
\begin{enumerate}[topsep=0cm,itemsep=0.05ex,leftmargin=1cm]
\item[(a)] 
For each $T\in \SS^b_*$ there exists $t_\Delta(T)\in (0,\infty)$ so that 
the map $\Edh_{t_\Delta(T)}(T)$ is trivial (\ie monotone).
\item[(b)]   $\SS^b_*\ni T\mapsto t_\Delta(T)$ is continuous.
\item[(c)] For each $t\in [0,t(T)]$, 
$\Edh_t(T)\in L_j(h_0^-) \cap L_j(h_0^+)$ 
implies that 
if $Z_j$ is part of a wandering pair then $Z_j$
  is contained in the closure 
 of a component of the basin of a periodic attractor of $\Edh_t(T)$.
 In particular, $\Edh_t(T)\in L_j(\lh0) \cap L_j(\gh0)$ implies 
that $\Edh_t(T)\in \SS^b_{*,j}$.
\end{enumerate}
\end{proposition}
\begin{proof}
Note that the right hand side of differential equation \eqref{eq:diffeq} is 
smooth on each stratum $\SS^b(\Xi)$. The discontinuities of the right hand side
occur when two plateaus start to touch, and the nature of the
equation is that once they touch they remain touching. From this and 
the existence and uniqueness theorem of differential equations (applied to each stratum
separately),  it follows that the flow of the differential equation is well-defined. However, 
continuity of $\SS^b_*\times \R^+ \ni (T,t)\to \Edh_t(T)$ still needs to be proved.

For simplicity write $T_s=\Edh_s(T)$ and let
$\Xi_s$ be the partition associated to $T_s$.
As $s$ increases, each plateau only widens under the flow $T_s$, until 
it touches another plateau in which case these plateaus widen jointly unless
the corresponding block has an even number of touching plateaus (or touches $\pm e$).
In this sense each coordinate of $T_s$ depends monotonically 
on $s>0$, and  therefore the limit
$\tilde T := \lim_{s\to \infty} T_s \in \SS^b$ 
exist.
Also $s\mapsto h_{top}(T_s)$ is non-increasing in $s$.

Now take $T\in \SS^b_*$ with $h_{top}(T)=h_0$, 
and assume by contradiction that property (a) does not hold.

%
{\bf Step 1.} Assume that $K$ is a renormalization interval of entropy $h_0$ for $T_t$ so that there exists no renormalization interval $K'\supsetneqq K$ of the same period. We claim that $K$ is a renormalization interval for $T_s$
for each $s\in [0,t]$. Indeed, otherwise there would exist $0\le s_0<t$ so that
one of the plateaus $Z_{j,s}\subset K$ of $T_{s_0}$ arches over 
$\partial K$ under the first return map of $T_{s_0}$ to $K$. 
If $Z_{j,s}$ is independent  
of $s\in [0,s_0]$, then $K$ is still a renormalization interval for $T_s$
and there is nothing to show. If $Z_{j,s}$ does depend on $s$, then there 
exists $s'\in [0,s_0)$ so that the 
first return map  to $K$ under $T_{s'}$ has another branch, see the left 
panel in Figure~\ref{fig:Zjboundary}.
Here we use that if for $T_{s_0}$ there exists a plateau outside $K$, 
but which touches an endpoint of $K$ (as is shown in the right panel of Figure~\ref{fig:Zjboundary}), then there would have existed a renormalization 
interval $K'\supsetneqq K$  with the same return time, contradicting the choice of $K$.
Since the first return map  to $K$ under $T_{s'}$ has another branch, 
it follows by Lemma~\ref{lem:decrentr} that $h_0\ge h_{top}(T_{s'}|\Omega(K))>
h_{top}(T_t|\Omega(K))=h_0$, contradicting the assumption made in this step.

\begin{figure}[ht]
\begin{center}
\unitlength=1mm
\begin{picture}(30,30)(-6,-17)
\thinlines
\put(-31,-11){\line(1,0){22}}
\put(-31,-11){\line(0,1){22}}
\put(-9,11){\line(-1,0){22}}
\put(-9,11){\line(0,-1){22}}
\put(-22,13){\tiny $K$}
\thicklines
\put(-32,-15){\line(1,4){6.5}}
\put(-23.5,11){\line(1,-4){3}}
\put(-20.5,-1){\line(1,0){2}}
\put(-18.5,-1){\line(1,4){2}}
\put(-16.5,7){\line(1,0){2.7}}
\put(-8,-15){\line(-1,4){5.5}}
\thinlines
\put(-26,-11){\line(1,6){1}}
\put(-25,-5){\line(1,0){1}}
\put(-24,-5){\line(1,-6){1}}
\put(-22,-9){\tiny $\Edh_{s'}(T)$}
\thinlines
\put(29,-11){\line(1,0){22}}
\put(29,-11){\line(0,1){22}}
\put(51,11){\line(-1,0){22}}
\put(51,11){\line(0,-1){22}}
\put(38,13){\tiny $K$}
\thicklines
\put(28,-15){\line(1,4){6.5}}
\put(36.5,11){\line(1,-4){3}}
\put(39.5,-1){\line(1,0){2}}
\put(41.5,-1){\line(1,4){2}}
\put(43.5,7){\line(1,0){2.7}}
\put(51,-11){\line(-1,4){4.5}}
\put(51,-11){\line(1,0){4}}

\end{picture}
\end{center}
\caption{\label{fig:Zjboundary} The situation considered in Step 1 of the proof of Proposition~\ref{prop:deltadeform}. The first return map to $K$ for 
$\Edh_s(T)$  and (schematically) for $\Edh_{s'}(T)$ for $s'<t$ is drawn in the left panel.
The situation as in the right cannot occur, see the proof in Step 1.}
\end{figure}
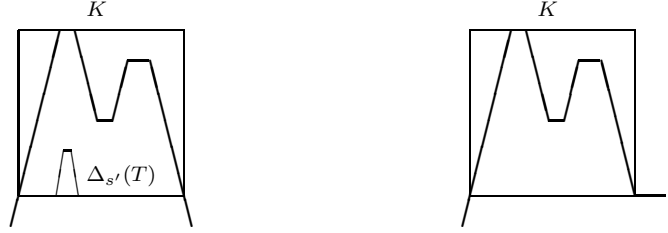

{\bf Step 2.}
Take $t>0$. 
We claim that there exists $\epsilon>0$  so that for each $s\in [t-\epsilon,t+\epsilon]$ there exists $i\in \{1,\dots,b\}$ with $T_s\notin \overline{\xi_i\cap \SS^b(\Xi_s)}$ and
$Z_i\subset \zyk(K)$.

We may assume that $K$ is a renormalization interval of entropy $h_0$ for $T_t$
because otherwise this claim holds trivially.
To prove this claim,  we first note that by Step 1, $K$ is a renormalization interval for $T_s$ for each $s\in [0,t]$.    Let us consider all the possible situations. 

{\bf Case A.}   $T_t$ does not have a renormalization interval 
$\hat K\subsetneqq K$.  We claim that in this case $h_{top}(T_t|\Omega(K))\le h_{top}(T_s|\Omega(K))<h_{top}(T_0|\Omega(K))\le h_0$. By the definition of $\xi_i$ this implies 
that $T_s\notin  \overline{\xi_i\cap \SS^b(\Xi_s)}$ for each $0\le s\le t$ and each $i$ so that $Z_{i,t}\subset K$; therefore, in this case Step 2 follows from the claim.
To prove this, first note that $T_s$ has no renormalization interval $\hat K\subsetneqq K$ 
for any $s\in [0,t]$. Indeed, if there exists $s \in [0,t]$ 
so that $T_{s}$ does have a renormalization interval $\hat K\subsetneqq K$, 
then by Lemma~\ref{lem:decrentr} we obtain $h_{top}(T_t|\zyk(K))<h_{top}(T_s|\zyk(K))\le h_0$ for any $s<t$, which gives a contradiction.
Next take $T_s$ with $s>0$ small, and let $T_s(n)\in \SS^b(\Xi_s)$ be a sequence of maps
so that $T_s(n)\to T_s$. 
Since $T_t$  has no renormalization intervals $\hat K\subsetneqq K$,
the map $T_t$ and therefore 
$T_0=T$ has no periodic attractors in $K$ either (and since $T\in \SS^b_*$, 
therefore no two plateaus of $T$ touch).  Since $T\in \SS^b_*$,
the convex hull $\J_n$ of  two  adjacent plateaus $Z_{i,n},Z_{i+1,n}$ 
for $T_s(n)$ for $n$ large and $s>0$ small,  does not form a wandering pair, see 
Lemma~\ref{lem:s*open}(b).  It follows that $T_s\notin  \overline{\xi_i\cap \SS^b}$ for each $i$ and each $s\ge 0$ small, and therefore for each $i$ so that $Z_i\subset K$ we have that 
$\rho_i(T_s) \cdot \hsgn(Z_{i,s})\ne 0$ for each $s\ge 0$ small.
But this implies by Lemma~\ref{lem:decrentr}  that $h_{top}(T_t|\Omega(K))\le h_{top}(T_s|\Omega(K))<h_{top}(T_0|\Omega(K))$ as claimed. 

%


{\bf Case B}.  $T_t$ has a maximal renormalization interval $\hat K\subsetneqq K$
on which the map $T_t$ is non-monotone. We claim that 
$T_s\notin \overline{\xi_i\cap \SS^b(\Xi_s)}$ for each $s\in [0,t]$ and for each $i$ so that
$Z_i\subset \hat K$. To see this, first note that we may assume that 
each maximal renormalization interval $\hat K\subsetneqq K$
is a renormalization interval for $T_s$ for all $s\in [0,t]$. Indeed,  otherwise by 
Lemma~\ref{lem:decrentr}, $h_{top}(T_t|\Omega(K))<h_{top}(T_0|\Omega(K))\le h_0$, 
and then the claim follows immediately. 
To prove the claim, we consider the following cases:

{\bf (i)} Each plateau $Z_{i,t}$ in $\hat K$ is mapped by $T_t$ 
 into the interior of $\zyk(\hat K)$. Then a nearby map $\tilde T$
 also has $\tilde Z\subset \zyk(\hat K)$ and so 
part (4) of Definition~\ref{def:xi} fails, and $\tilde T \notin \xi_i$. 
 It follows $T_s\notin \overline{\xi_i\cap \SS^b(\Xi_s)}$ for each $i$ associated to a plateau  $Z_{i,s}\subset \zyk(\hat K)$ and for each $s$ near $t$. 

{\bf (ii)}  One of the plateaus $Z_{i,t}$ in $\hat K$ arches by $T_t$ over the boundary of a 
component of $\zyk(\hat K)$. Then consider the following two subcases:

{\bf (iia)} There exists $\epsilon>0$ so that the  cardinality of the
block of plateaus touching  $Z_{i,s}$ does  not change for $s\in [t-\epsilon,t]$.
Next take a convex hull  $\J_s$ connecting $Z_{i,s}$ and another plateau (outside the block of plateaus touching  $Z_{i,s}$). (By Definition~\ref{def:xi} we need to take $\J_s$
non-trivial, so $\J_s$ contains two plateaus $Z_{i,s},Z_{j,s}$ from distinct
subsets of the partition $\Xi_s$.) Because $Z_{i,s}$ arches, either $\J_s$
is contained in $\hat K$ or the block of plateaus touching $Z_{i,s}$ (and that of
$Z_{i,t}$) contains an even number of plateaus. In the former
case for $s\in [t-\epsilon,t)$ the width of this block is smaller, and in the latter case, 
$\hsgn(Z_{i,s})=0$ for each $s\in [t-\epsilon,t]$. It follows that in both cases 
$T_s(\J_s)$ contains a repelling (pre-)periodic boundary point of $\hat K$ for each 
$s\in [t-\epsilon,t]$. The same holds for each $\tilde T\in \Xi(T_s)$ near $T_s$
and therefore $\tilde T\notin \xi_i\cap \SS^b(\Xi_s)$ for each $s\in [t-\epsilon,t]$.
It follows that $T_s\notin \overline{\xi_i\cap \SS^b(\Xi_s)}$ for each $s\in [t-\epsilon,t]$.
Provided $\epsilon>0$ is small, for each $s\in [t,t+\epsilon]$ the  cardinality of the
block of plateaus touching  $Z_{i,s}$ does not change either, and for the same reason
$T_s\notin \overline{\xi_i\cap \SS^b(\Xi_s)}$ for each $s\in [t,t+\epsilon]$.

{\bf (iib)} There exists no such $\epsilon>0$. Then there exists a plateau $Z_{j,s}\subset \hat K$ 
(so that $Z_{j,t}$ is in the block of plateaus touching $Z_{i,t}$) which creates an extra 
branch for the return map of $T_s$ to $K$ (compared to the return map of 
$T_t$). Hence, by Lemma~\ref{lem:decrentr}, 
$h_{top}(T_t,\Omega(K))<h_{top}(T_s,\Omega(K)) \le h_0$, which gives a contradiction. 

{\bf (iii)}  $T_t$ is constant on one of the components of $\zyk(\hat K)$ and for simplicity
assume that $\hat K$ is this component.  Let $Z_{i,t}$
be contained in $\hat K$. 

{\bf (iiia)} There exists $\epsilon>0$ so that the cardinality of the
block of plateaus touching  $Z_{i,s}$ (in $\hat K$) 
does  not change for $s\in [t-\epsilon,t]$.
In this case the number of plateaus in this block is odd
 and so this block of plateaus is mapped into the interior of $\hat K$ for 
 $s\in [t-\epsilon,t)$. By case (i) it follows that for $s\in [t-\epsilon,t)$ one has that 
$T_s\notin \overline{\xi_i\cap \SS^b(\Xi_s)}$. To prove this conclusion also 
for $s=t$, consider a sequence of maps $T_t(n)\to T_t$ where $T_t(n)\in \SS^b(\Xi_t)$.  
By maximality of $\hat K\subsetneqq K$
it follows that each boundary point of $\hat K$ is a repelling periodic or 
pre-periodic 
point which is not the common boundary point of two basins. Hence
there exists a sequence of repelling periodic points accumulating on $\partial \hat K$
(from outside $\hat K$). It follows that any non-trivial convex hull $\J(n)$ connecting $Z_{i,t}$ with another  plateau (outside this component of $\hat K$) will contain a repelling periodic point of $T_t(n)$ and therefore  iterates of $\J(n)$ are not singletons. It follows that
$T_t(n)\notin \xi_i$ and therefore $T_t\notin \overline{\xi_i\cap \SS^b(\Xi_t)}$. Moreover, for
$s\in (t,t+\epsilon]$ one has $h_{top}(T_s|\orb(K))<h_{top}(T_t|\orb(K))=h_0$,
where we choose $\epsilon>0$ so that $K$ remains periodic for $T_s$ for all 
$s\in [t,t+\epsilon]$. It again follows that $T_s\notin \overline{\xi_i\cap \SS^b(\Xi_s)}$
for $s\in (t,t+\epsilon]$ for each $i$ with $Z_i\subset K$. 

{\bf (iiib)}  There exists no such $\epsilon>0$. Then there exists a plateau $Z_{j,s}\subset \hat K$ 
(so that $Z_{j,t}$ is in the block of plateaus touching $Z_{i,t}$) which creates an extra 
branch for the return map of $T_s$ to $K$ (compared to the return map of 
$T_t$), and so Lemma~\ref{lem:decrentr} yields $h_{top}(T_t,\Omega(K))<h_{top}(T_s,\Omega(K))\le h_0$, which again shows that $T_s\notin \overline{\xi_i\cap \SS^b(\Xi_s)}$
for $s\in [t,t+\epsilon]$. 

{\bf Case C.}  $T_t$ is monotone on each component of $\zyk(\hat K)$ for each 
maximal renormalization interval $\hat K\subsetneqq K$. We claim that in this case
there exists $\epsilon>0$ so that $T_s\notin   \overline{\xi_i\cap \SS^b(\Xi_s)}$ 
for $s\in [t,t+\epsilon]$ and for each $i$ so that $Z_i\subset K$. Indeed, in this 
setting there exists $\epsilon>0$ so that the  
cardinality of the block of plateaus touching $Z_{i,s}$ in $\zyk(\hat K)$
is constant for $s\in [t,t+\epsilon]$. Moreover, 
for each sequence $T_s(n)\to T_s$ with $T_s(n)\in \SS^b(\Xi_s)$,
the map $T_s(n)$ is also monotone on each component of  $\zyk(\hat K)$.
It follows that for each plateau $Z_{i,s}$ which is not contained in a
renormalization interval $\hat K\subsetneqq K$ one has 
$T_s(n)\notin \xi_i$ (because of Definition~\ref{def:xi}(2)).
Hence
 $T_s\notin \overline{\xi_i\cap \SS^b(\Xi_s)}$
for each such $i$ and $s\in [t,t+\epsilon]$. Note that for $s\in [t-\epsilon,t]$
either $T_s$ is also monotone on each component of $\zyk(\hat K)$ 
(and therefore $T_s(n)\notin  \overline{\xi_i\cap \SS^b(\Xi_s)}$ for each $i$ as above) 
or one of the plateaus $Z_{i,s}\subset \hat K$ corresponds to a local extremum and then
$T_s(n)\notin  \overline{\xi_i\cap \SS^b(\Xi_s)}$.

{\bf Step 3.}  We claim that if $h_{top}(T_t|\Omega(K))=h_0$  for some $t>0$,
then for each $Z_i\subset K$ and $s\in [0,t]$ some iterate of $Z_{i,s}$ 
is contained in a renormalization interval $\hat K\subsetneqq K$
(and this renormalization interval $\hat K$ does not depend on $s$). 
Indeed, if $Z_{i,t}$ is not eventually mapped into some renormalization interval 
$\hat K\subset K$ but this is the case for some $s\in [0,t]$ then Lemma~\ref{lem:decrentr} 
implies that  $h_0\ge h_{top}(T_{s}|\Omega(K))> h_{top}(T_t|\Omega(K))$,
and so we are done.  Therefore it suffices to consider the case that no iterate of
$Z_{i,s}$, $s\in [0,t]$, is contained in  a renormalization interval $\hat K$. 
Let us show that this implies that $Z_{i,s}$ moves when $s$ small. Indeed, consider a sequence of maps $T_s(n)\to T_s$
as $n\to \infty$ and a non-trivial convex hull 
$\J_s(n)$ containing the plateau corresponding to $Z_i$. Since $T\in \SS^b_*$,  
Lemma~\ref{lem:s*open}(b) implies that if $\J_s(n)$  forms a wandering pair for  $T_s$ for $s>0$ small,   then the corresponding convex hull $\J$  is contained in the basin of a periodic attractor for $T$. Since we assumed that $Z_i$ is not eventually mapped into a
renormalization interval $\hat K$, this is impossible. It follows that $T_s(n)\notin \xi_i\cap \SS^b(\Xi_s)$ and therefore $T_s\notin \overline{\xi_i\cap \SS^b(\Xi_s)}$ when $s$ is small.
Hence $Z_{i,s}$ moves when $s$ small, and so Lemma~\ref{lem:decrentr}  implies again 
that  $h_0\ge h_{top}(T_{s}|\Omega(K))> h_{top}(T_t|\Omega(K))$.

%
%
%
%
%

{\bf Step 4.} We claim that for each $T\in \SS^b_*$ there exists $t>0$ 
so that $T_t$ is trivial. Indeed, consider a renormalization interval
$K$ for $T$. By Step 3, either $h_{top}(T_t|\Omega(K))<h_0$ 
or each plateau $Z_{i,s}$ in $K$ is eventually mapped into a renormalization
interval $\hat K\subsetneqq K$ for each $s\in [0,t]$. 
If $h_{top}(T_t|\Omega(K))<h_0$ for each $t>0$,
then each plateau in $K$ will move with positive speed and since
$h_{top}(T_t|\Omega(K))$ is decreasing, this speed will not tend to zero as $t$ increases
(unless this plateau becomes part of a block consisting of an even number of plateaus). 
If each plateau $Z_{i,s}$ in $K$ is eventually mapped into a renormalization interval 
$\hat K\subsetneqq K$ for each $s\in [0,t]$, then the period of these intervals $\hat K$
will only depend on $T$ (and not on $t$). It follows by Step 2, Case B that the speed of at least one plateau in $\hat K$ is bounded away from zero (until all plateaus are in blocks consisting of an even number plateaus), and so in finite time the first return map to 
$\hat K$ will be either monotone or constant. In the former case other plateaus in $K$ start moving, see Step 2, Case C. In the latter case the renormalization interval 
$\hat K$ disappears, see Step 2 Case C(iii), and $h_{top}(T_t|\Omega(K))$ becomes $<h_0$.

Note that $h_{top}(T_t|\Omega(K))=h_0$  implies that the period of $K$
is bounded from below, and so there are only a bounded number of intervals
$K$ to be considered in the previous paragraph.  Step 4 and therefore  
part (a) of the proposition follow.

{\bf Step 5.} From Step 4 it follows that if $T\in \SS^b_*$ then $T_t$ never
enters the set where the speed of a plateau is actually zero, unless 
plateaus touch. It follows that if $T$ and $T'$ are nearby maps in $\SS^b_*$
so that for at time $s=t$, two plateaus of $T_s$ start to touch, then 
the same two plateaus will start touching for the map $T'_{s'}$ for some $s'\approx t$.  
It follows that the map $(T,t) \to \Edh_t(T)$ is continuous and that
the map $\SS^b_*\ni T\mapsto t_\Delta(T)$ is continuous. This finished part (b)
of the proposition.

{\bf Step 6.} Let us now prove part (c) of the proposition and assume 
that  $t>0$ and
$\Edh_t(T)\in L_j(h_0^-) \cap L_j(h_0^+)$.  Amongst other things, 
this implies that there exists
a renormalization interval $K$ for  $\Edh_t(T)$ on which the 
entropy is $h_0$. By Lemma~\ref{lem:decrentr}, this  implies that each plateau is either contained in a maximal renormalization interval $\hat K\subsetneqq K$ or is mapped into such an interval. 
Since $\Edh_t(T)\in L_j(h_0^-) \cap L_j(h_0^+)$, the plateau $Z_{j,t}$ is eventually
mapped into the boundary of such a renormalization interval $\hat K$.
If $Z_{j,t}$ maps in an arching way over $\partial\hat K$, 
then this implies that either
there exists $s\in [0,t)$ so that $Z_{j,s}$ is not mapped into $\hat K$,
contradicting that the entropy of $T_t$ on $K$ is $h_0$ (using Lemma~\ref{lem:decrentr}),
or $Z_{j,s}$ maps to $\partial\hat K$ for each $s\in [0,t]$.
Since $T\in \SS^b_*$, this implies that $Z_{j,t}$ maps to the boundary of the
basin of a periodic attractor of $T_t$, and so we are done.
Next assume that $Z_{j,t}$ maps in a non-arching way over $\partial\hat K$ 
and that $Z_{j,t}$ is part of a wandering pair. 
In this case,
either the first return map to $\hat K$ is monotone or $T_t\in \xi_j$, 
in which case the speed of $Z_{j,s}$ at $s=t$ is zero. If the former holds, 
each point in $\hat K$ is in the boundary of a component of the basin 
of a periodic attractor and we are done. If the latter holds, 
the speed of $Z_{j,s}$ is zero for each $s\in [0,t]$ and
$T_s\in \xi_j$ for  {\em each} $s\in [0,t]$. But since $T\in \SS^b_*$, this 
implies that a boundary point of $\hat K$ for $T$ (and therefore for 
$T_s$ for each $s\in [0,t]$) is in the boundary of a component of the basin of a
periodic attractor, and again we are done. 
\end{proof}

Let us define 
$$\hat \Delta_t(T)=\Delta_{t_\Delta(T)\cdot t}(T)$$
so that $\hat \Delta_1(T)$ is a trivial map for each $T\in \SS^b_*$. 

\subsection{The proof of Theorem~\ref{Thm:Connected}.}

Now that we have developed the ingredients of the proof, we can
define the retract for a fixed $h_0 \in (0,\log(b+1))$.
(The cases $h_0=\log(b+1)$ and $h_0 = 0$ were dealt with in Sections~\ref{subsec:h0=logb+1} and \ref{subsec:h0=0}.)
\[
R_t= \left\{ \begin{array}{ll}
\beta_{5t} & \mbox{ for } t \in [0,\frac15], \\[0.1cm]
\Ei_{5t-1} \circ  \beta_{1} & \mbox{ for } t \in [\frac15,\frac25],\\[0.1cm]
\Ei_1\circ \Eih_{5t-2} \circ  \beta_{1} & \mbox{ for }
t \in [\frac25,\frac34],\\[0.1cm]
\Ei_1 \circ \Eih_1\circ
\beta_1\circ \hat \Delta_{(5t-3)}  & \mbox{ for } t \in [\frac35,\frac45],
\\[0.1cm]
\Ei_1 \circ \Eih_1\circ
\beta_1\circ r_{5t-4}\circ \hat \Delta_1  & \mbox{ for } t \in [\frac45,1].
\end{array} \right.
\]
Obviously, $R_0(T) = T$, and since for $t = 1$, the retract $r_{6t-5}$
has been carried out completely, $R_1(T)$ is the same map for
each $T \in L_*(h_0)$ of the same shape $\epsilon$.
All components of $R_t$ are continuous in $t$ and $T$, so
the same holds for $R_t$.

Let us show that $R_t$ keeps maps within 
$\SS_*^b$. First note that the only deformation which takes a map
outside the space $\SS^b_*$ is $\hat \Delta_t$.
Take $T'$ of the form $T'=\hat \Delta_t(T)$ or $T'= r_t\circ \hat \Delta_1(T)$.
The deformation $\beta_t(T')$ moves plateaus
$Z_i,Z_{i+1}$ whose convex hull is mapped into other plateaus. It does
so in such a way that $\beta_1(T')$ never eventually maps $[Z_i, Z_{i+1}]$
into the interior of another plateau and so the $\beta$-property 
\eqref{eq:nointerior} on page \pageref{eq:nointerior} will hold.

If $h_{top}(T')<h_0$ then $h_{top}(\beta_1(T'))<h_0$
and  of course $\beta_1(T)$ satisfies the $\beta$-property \eqref{eq:nointerior}.
Because of \eqref{eq:makegood} this gives that
$\Eih_1 \circ \beta_1(T')\in \SS^b_*$.
In particular, we are done if $T'=r_t\circ \hat \Delta_1(T)$ satisfies $h_{top}(T')<h_0$.

Let us now consider the case that $h_{top}(T')=h_0$.
By the third part of Lemma~\ref{lem:propU},
$\Ei_1 \circ \Eih_1\circ  \beta_1 (T')\in \SS^b_*$ provided that 
$T'=\beta_1(T')\in L_j(\lh0)\cap L_j(\gh0)$ implies
$T'\in \SS^b_{*,j}$.  But in Proposition~\ref{prop:deltadeform}
 it is shown that any map $T'$ of the form
$T'=\beta_1\circ \hat \Delta_t(T)$ 
with $t>0$ and $T\in \SS_*^b$ has indeed this property, and so again the resulting map belongs to $\SS_*^b$.
This concludes the proof of the first part of Theorem~\ref{Thm:Connected}.
The 2nd part follows from the construction of $\Gamma_t$ and $\beta_t$.


\begin{thebibliography}{99}



\bibitem{BlKe} L.\ Block, J.\ Keesling,
{\em Computing the topological entropy of maps of the interval with three
monotone pieces,}  J.\ Statist.\ Phys.\  {\bf 66}  (1992), 755--774.

\bibitem{BlockLedis} L.\ Block, D.\ Ledis, {\em Topological conjugacy, transitivity, and patterns,}
Preprint 2013.


\bibitem{Bru} H.\ Bruin, {\em Non-monotonicity of entropy of interval maps, } Phys. Lett. A {\bf 202} (1995), no. 5-6, 359--362.

\bibitem{BSnon} H. Bruin, S. van Strien, {\em On the structure of isentropes of polynomial maps}, Dynamical Systems: An International Journal  {\bf 28} (2013), 381-392 

\bibitem{boyland} P.\ Boyland,
{\em Semiconjugacies to angle-doubling},
Proc.\ Amer.\ Math.\ Soc.\ {\bf 134} (2006), 1299--1307.

\bibitem{CG}
L.\ Carleson, T.\ W.\ Gamelin, {\em Complex dynamics},
Universitext: Tracts in Mathematics, Springer-Verlag, New York, 1993.

\bibitem{CS} D. Cheraghi, S van Strien, {\em Towards Tresser's conjecture}. In preparation. 

\bibitem{DGMT} S.\ P.\ Dawson, R.\ Galeeva, J.\ Milnor, C.\ Tresser,
{\em A monotonicity conjecture for real cubic maps},
In: Real and Complex dynamical systems
B. Branner and P. Hjorth Eds, Kluwer, Dordrecht, (1995), 165--183.

\bibitem{Dou} A.\ Douady, {\em Topological entropy of unimodal maps: Monotonicity for quadratic polynomials}. 
In: Real and Complex Dynamical Systems, B. Branner and P. Hjorth, Eds. Dordrecht: Kluwer, (1995), 65--87.

\bibitem{DH} A.~Douady, J.\ H.\ Hubbard, 
{\em Etude dynamique des polyn\^omes quadratiques complexes}, 
I, (1984) \& II, (1985), Publ. Mat. d'Orsay.

\bibitem{Epstein} A.\ Epstein, {\em Transversality in holomorphic dynamics}. Preliminary version  available from 
http://www.warwick.ac.uk/$\sim$mases/Transversality.pdf

\bibitem{GMT} R.\ Galeeva, M.\ Martens, C.\ Tresser,
{\em Inducing, slopes, and conjugacy classes}, 
Israel J.\ Math.\ {\bf 99} (1997), 123--147.

\bibitem{GS} J.\ Graczyk, G.\ \'Swi{\c{a}}tek, 
{\em Generic hyperbolicity in the logistic family}. 
Ann.\ of Math.\ {\bf 146} (1997), 1--52. 

\bibitem{Hai1} P.\ Ha{\"\i}ssinsky, {\em Chirurgie parabolique},
C.\ R.\ Acad.\ Sci.\ Paris, {\bf 327} S\'erie 1 (1998) 195--198.

\bibitem{Hai2} P.\ Ha{\"\i}ssinsky, {\em Deformations J-\'equivalentes de
polyn\^omes g\'eome\-tri\-que\-ment finis},
Fund.\ Math.\ {\bf 163} (2000), 131--141.

\bibitem{Hec} C.\ Heckman, {\em Monotonicity and the Construction of Quasiconformal Conjugacies in the Real Cubic Family},
PhD Thesis, Stony Brook (1996).

\bibitem{Hof} F.\ Hofbauer,
{\em On intrinsic ergodicity of piecewise monotonic transformations
with positive entropy, II.}
Israel J.\ of Math.\ {\bf 38} (1981) 107--115.

\bibitem{KKY} I.\ Kan, H.\  Ko\c{c}ak, J.\ Yorke,
{\em Antimonotonicity: concurrent creation and annihilation of periodic
orbits,}  Ann.\ of Math.\ {\bf 136} (1992), 219--252.

\bibitem{Kol} S.\ F.\ Kolyada {\em One-parameter families represented by integrals with negative
Schwarzian derivative violating monotone bifurcations}, 
Ukr.\ Mat.\ Z.\ {\bf 41}   (1989), 258--261 (in Russian)

\bibitem{KSS} O.\ Kozlovski, W.\ Shen, S.\ van Strien,
{\em Rigidity for real polynomials},  
Ann.\ of Math.\ {\bf 165} (2007), 749--841.

\bibitem{KSS1} O.\ Kozlovski, W.\ Shen, S.\ van Strien,
{\em Density of hyperbolicity in dimension one},  
Ann.\ of Math.\ {\bf 166} (2007), 145--182.


\bibitem{Levin} G.\ Levin, {\em Multipliers of periodic orbits in spaces of rational maps}, 
Ergod.\ Th.\  Dynam.\ Sys.\, {\bf 31} (2011), 197-243.


\bibitem{Lyu} M.\ Lyubich, 
{\em  Dynamics of quadratic polynomials, I and II}, 
Acta Math.\ {\bf 178}  (1997), 185--247, 247--297.

\bibitem{MacTr} R.\ MacKay, C.\ Tresser,
{\em Boundary of chaos for bimodal maps of
the interval,} J.\ London Math.\ Soc.\  {\bf 37} (1988), 164-181.


\bibitem{McM} C. McMullen, {\em Complex dynamics and renormalization}, 
Annals of Mathematics Studies {\bf 13}, Princeton University Press, (1994).

\bibitem{MS} W.\ de Melo, S.\  van Strien,
{\em One-Dimensional Dynamics,}
Springer, Berlin Heidelberg New York, (1993).

\bibitem{MetSS} N.\ Metropolis, M.\ L.\ Stein, P.\ R.\ Stein,  {\em On finite limit sets for transformations on the unit interval}. J. Combinatorial Theory Ser. A {\bf 15} (1973), 25--44.

\bibitem{Mil} J.\ Milnor,
{\em Remarks on iterated cubic maps},
Experimental Math.\ {\bf 1} (1992), 5--24.

\bibitem{Mil1} J.\ Milnor,
{\em Hyperbolic components. 
With an appendix by A. Poirier}, Contemp. Math.,  {\bf 573}, 
Conformal dynamics and hyperbolic geometry, 183--232, Amer. Math. Soc., Providence, RI, 2012.



\bibitem{MT} J.\ Milnor, W.\ Thurston,
{\em  On iterated maps of the interval: I, II},
Preprint 1977. Published in Lect. Notes in
Math.\ {\bf 1342}, Springer, Berlin New York (1988), 465--563.

\bibitem{MTr} J.\ Milnor, C.\ Tresser,
{\em On entropy and monotonicity for real cubic maps},
Comm.\ Math.\ Phys.\ {\bf 209} (2000), 123--178.


\bibitem{Mis} M.\ Misiurewicz,
{\em Continuity of entropy revisited,}
Dynamical systems and applications, World Sci.\ Ser.\ Appl.\ Anal.\ {\bf 4} (1995), 495--503.

\bibitem{MSz} M.\ Misiurewicz, W.\ Szlenk,
{\em Entropy of piecewise monotone mappings,}
Studia Math.\ {\bf 67} (1980), 
45--63.


\bibitem{NY}  H.\ E.\ Nusse, J.\ A.\ Yorke, {\em Period halving for
  $x_{n+1} = M \cdot F(x_n)$ where $F$ has negative Schwarzian derivative}. 
Physics Letters A {\bf 127} (1988), 328--334

\bibitem{Parry} W.\ Parry,
{\em On the $\beta$-expansion of real numbers,}
Acta Math. Acad. Sci. Hungar {\bf 11} (1960), 401-416.

\bibitem{Radu} A.\ Radulescu, 
{\em The connected isentropes conjecture in a space of quartic
  polynomials,} 
Discrete Contin.\ Dyn.\ Syst.\ {\bf 19} (2007), 139--175. 

\bibitem{RvS} L.\ Rempe, S.\ van Strien,
{\em Density of hyperbolicity for real transcendental entire functions with real singular values},
http://arxiv.org/abs/1005.4627.

\bibitem{Shen_C2} W.\ Shen,  {\em On the metric properties of multimodal interval maps and {$C^2$} density of {A}xiom {A}},
Invent.\ Math.\ {\bf 156} (2004), 301-403.

\bibitem{Str1} S.\ van Strien, 
{\em One-dimensional dynamics in the new millennium},
Discrete and Continuous Dynamical Systems, A, {\bf 27} (2010), 557-588.

\bibitem{Str2} S.\ van Strien,
{\em One-parameter families of smooth interval maps: density of hyperbolicity and robust chaos}, 
Proc.\ Amer.\ Math.\ Soc.\ {\bf 138} (2010), 4443-4446.

\bibitem{Str3} S.\ van Strien, {\em Milnor's Conjecture on Monotonicity of Topological Entropy:
results and questions}, In {\lq}Frontiers in Complex Dynamics: Celebrating John Milnor's 80th Birthday{\rq}, 673--687. Editor(s): Bonifant, Lyubich, Sutherland, Princeton University Press, ISBN:9780691159294.

\bibitem{Tsu} M.\ Tsujii, 
{\em A simple proof for monotonicity of entropy in the quadratic family,}  
Ergod.\ Th.\ Dynam.\ Sys.\  {\bf 20} (2000), 925--933.

\bibitem{Y} Y.\ Yomdin,
{\em Volume growth and entropy,}
Isr.\ J.\ Math.\ {\bf 57} (1987), 285--300.

\bibitem{Za} S.\ Zakeri, 
{\em On critical points of proper holomorphic maps on he unit disk}, Bulletin of the LMS, 
{\bf 30} (1998), 62-66.

\bibitem{Zdu} A.\ Zdunik, 
{\em Entropy of transformations of the unit interval}, 
Fund.\ Math.\ {\bf 124} (1984), 235--241


\end{thebibliography}
\end{document}